\documentclass[11pt]{article}
\usepackage{amsmath,amsthm,amssymb}
\usepackage[usenames,dvipsnames]{xcolor}
\usepackage{tabularx}
\usepackage{enumerate}
\usepackage{graphicx}
\usepackage[noadjust]{cite}
\usepackage{comment}
\usepackage{oands}
\usepackage{tikz}
\usepackage{changepage}
\usepackage{bbm}
\usepackage{mathtools}
\usepackage[margin=1in]{geometry}
\usepackage[pagewise,mathlines]{lineno}
\usepackage{appendix}
\usepackage{stmaryrd} 
\usepackage{multicol}
\usepackage{microtype}
\usepackage[colorinlistoftodos]{todonotes}
\usepackage[normalem]{ulem}

\usepackage[pdftitle={},
  pdfauthor={Ewain Gwynne, Nina Holden, Xin Sun},
colorlinks=true,linkcolor=NavyBlue,urlcolor=RoyalBlue,citecolor=PineGreen,bookmarks=true,bookmarksopen=true,bookmarksopenlevel=2,unicode=true,linktocpage]{hyperref}

\theoremstyle{plain}
\newtheorem{thm}{Theorem}[section]

\newtheorem{definition}[thm]{Definition}
\newtheorem{lem}[thm]{Lemma}
\newtheorem{prop}[thm]{Proposition}
\newtheorem{conj}[thm]{Conjecture}

\def\@rst #1 #2other{#1}
\newcommand\MR[1]{\relax\ifhmode\unskip\spacefactor3000 \space\fi
  \MRhref{\expandafter\@rst #1 other}{#1}}
\newcommand{\MRhref}[2]{\href{http://www.ams.org/mathscinet-getitem?mr=#1}{MR#2}}

\theoremstyle{definition} 
\newtheorem{defn}[thm]{Definition}
\newtheorem{example}[thm]{Example}
\newtheorem{remark}[thm]{Remark}

\newtheorem{prob}[thm]{Problem}

\numberwithin{equation}{section}

\newcommand{\dsb}{\begin{adjustwidth}{2.5em}{0pt}
\begin{footnotesize}}
\newcommand{\dse}{\end{footnotesize}
\end{adjustwidth}}

\newcommand{\ssb}{\begin{adjustwidth}{2.5em}{0pt}}
\newcommand{\sse}{\end{adjustwidth}}

\newcommand{\aryb}{\begin{eqnarray*}}
\newcommand{\arye}{\end{eqnarray*}}
\def\alb#1\ale{\begin{align*}#1\end{align*}}
\def\allb#1\alle{\begin{align}#1\end{align}}
\newcommand{\eqb}{\begin{equation}}
\newcommand{\eqe}{\end{equation}}
\newcommand{\eqbn}{\begin{equation*}}
\newcommand{\eqen}{\end{equation*}}

\newcommand{\BB}{\mathbbm}
\newcommand{\ol}{\overline}
\newcommand{\ul}{\underline}
\newcommand{\op}{\operatorname}

\newcommand{\re}{\operatorname{Re}}
\newcommand{\frk}{\mathfrak}
\newcommand{\eqD}{\overset{d}{=}}
\newcommand{\ep}{\epsilon}
\newcommand{\rta}{\rightarrow}

\newcommand{\wt}{\widetilde}
\newcommand{\wh}{\widehat} 
\newcommand{\mcl}{\mathcal}

\newcommand{\bdy}{\partial}

\newcommand{\D}{\mathbb{D}}
\newcommand{\CC}{\mathbb{C}}

\newcommand{\N}{\mathbb{N}}
\newcommand{\Q}{\mathbb{Q}}
\newcommand{\Z}{\mathbb{Z}}
\newcommand{\R}{\mathbb{R}}

\def\cZ{\mathcal{Z}}

\def\cW{\mathcal{W}}
\def\cV{\mathcal{V}}

\def\cT{\mathcal{T}}
\def\cS{\mathcal{S}}
\def\cR{\mathcal{R}}

\def\cM{\mathcal{M}}
\def\cL{\mathcal{L}}
\def\cK{\mathcal{K}}

\def\cH{\mathcal{H}}

\def\cE{\mathcal{E}}
\def\cD{\mathcal{D}}

\def\cB{\mathcal{B}}

\newcommand{\MT}{\cM\cT}
\newcommand{\LW}{\cL\cW}
\newcommand{\Dh}{\cD\cH}
\newcommand{\defeq}{:=}
\newcommand{\notion}[1]{{\bf\textit{#1}}}
\newcommand{\DP}{\mathfrak{P}}
  
\DeclareMathOperator{\SLE}{SLE}  
\DeclareMathOperator{\CLE}{CLE}  
\newcommand{\p}{\bdy}
\newcommand{\fh}{\mathsf h}

\let\originalleft\left
\let\originalright\right
\renewcommand{\left}{\mathopen{}\mathclose\bgroup\originalleft}
\renewcommand{\right}{\aftergroup\egroup\originalright}

\title{Mating of trees for random planar maps and Liouville quantum gravity: a survey}

\date{  }
\author{
\begin{tabular}{c} Ewain Gwynne\\[-5pt]\small Cambridge \end{tabular} 
\begin{tabular}{c} Nina Holden\\[-5pt]\small ETH Z\"urich \end{tabular}
\begin{tabular}{c} Xin Sun \\[-5pt]\small Columbia \end{tabular}
}

\begin{document}

\maketitle
  
\begin{abstract}
We survey the theory and applications of mating-of-trees bijections for random planar maps and their continuum analog: the mating-of-trees theorem of Duplantier, Miller, and Sheffield (2014). 
The latter theorem gives an encoding of a Liouville quantum gravity (LQG) surface decorated by a Schramm-Loewner evolution (SLE) curve in terms of a pair of correlated linear Brownian motions.
We assume minimal familiarity with the theory of SLE and LQG.

Mating-of-trees theory enables one to reduce problems about SLE and LQG to problems about Brownian motion and leads to deep rigorous connections between random planar maps and LQG.  
Applications discussed in this article include scaling limit results for various functionals of decorated random planar maps, estimates for graph distances and random walk on (not necessarily uniform) random planar maps, computations of the Hausdorff dimensions of sets associated with SLE, scaling limit results for random planar maps conformally embedded in the plane, and special symmetries for $\sqrt{8/3}$-LQG which allow one to prove its equivalence with the Brownian map. 
\end{abstract}
  
\tableofcontents

\section{Introduction}
\label{sec-intro}

\subsection{Overview}
\label{sec-overview}

A planar map is a graph embedded in the plane, viewed modulo orientation preserving homeomorphisms. 
Planar maps have been an important object of study in combinatorics since the pioneering work of Tutte~\cite{tutte}. 
More recently, random planar maps have been a major focus in probability theory and mathematical physics since they are the natural discrete analogs of so-called \emph{Liouville quantum gravity (LQG)} surfaces. LQG surfaces are canonical models of random fractal surfaces which are important in statistical mechanics, string theory, and conformal field theory.

In this article, we will be interested both in \emph{uniform random planar maps}, where each possible planar map satisfying certain constraints (e.g., on the total number of edges and/or the degrees of the faces) is assigned equal probability; and in planar maps weighted by --- i.e., sampled with probability proportional to --- the partition function of a statistical mechanics model on the map. This type of weighting is especially natural since it arises when we sample a planar map decorated by a statistical mechanics model. 
For example, suppose we sample a uniform random pair $(M,T)$ consisting of a planar map $M$ with $n$ edges and a spanning tree on $M$. Then the marginal law of $M$ is the uniform measure on planar maps with $n$ edges weighted by the number of possible spanning trees $T$ which they admit.
 
One of the most fruitful approaches to the study of random planar maps is to encode them in terms of simpler objects --- such as random trees and random walks --- via combinatorial bijections. An important class of such bijections are the so-called \emph{mating-of-trees bijections} which represent a planar map decorated by a statistical mechanics model as the gluing of a pair of discrete trees. Since discrete trees can be represented by their contour functions (see, e.g.,~\cite{legall-tree-survey}), this is equivalent to encoding the map by a two-dimensional walk. The simplest mating-of-trees bijection is the \emph{Mullin bijection}~\cite{mullin-maps,bernardi-maps,shef-burger} which encodes a planar map decorated by a spanning tree via a nearest-neighbor walk in $\BB Z^2$. Here the two trees being mated are the spanning tree and the corresponding dual spanning tree, see Section~\ref{sec-bijections}.  
There are also mating-of-trees bijections for many other decorated planar map models, such as site percolation on a triangulation~\cite{bernardi-dfs-bijection,bhs-site-perc} and planar maps decorated by an instance of the critical Fortuin Kasteleyn (FK) cluster model~\cite{shef-burger}. See Sections~\ref{sec-bijections} and~\ref{subsec:rmp} for more on these bijections.\footnote{Mating-of-trees bijections are fundamentally different from bijections of ``Schaeffer type"~\cite{schaeffer-bijection,bdg-bijection} which encode an undecorated planar map by means of a labeled tree, where the labels describe graph distances to a marked root vertex. This latter type of bijection has been used to great effect to study distances in uniform random planar maps (see, e.g.,~\cite{legall-uniqueness,miermont-brownian-map}), but this work is not the focus of the present survey.}

In this article,  we will survey the theory of mating-of-trees bijections and their continuum analog: the mating-of-trees theorem for LQG of Duplantier, Miller, and Sheffield~\cite{wedges}, 
and their applications. Let us now give a brief overview of the results we will present.   
A $\gamma$-LQG surface is, heuristically speaking, the random two-dimensional Riemannian manifold parameterized by a domain $D\subset\BB C$ whose Riemannian metric tensor is $e^{\gamma h } \, (dx^2 + dy^2)$, where $h$ is some variant of the Gaussian free field (GFF) on $D$ and $dx^2 + dy^2$ is the Euclidean metric tensor. 
The surfaces we now call LQG surfaces were, albeit in a different form, first described by Polyakov in the 1980's~\cite{polyakov-qg1} in the context of bosonic string theory. 

The above definition of LQG does not make literal sense since $h$ is a distribution, not a function. However, one can rigorously define LQG surfaces via various regularization procedures. For example, one can define the area measure (volume form) $\mu_h$ associated with an LQG surface as a limit of regularized versions of $e^{\gamma h} \,d^2z$, where $d^2z$ denotes Lebesgue measure.\footnote{
Recent work has shown that there is also a metric $\frk d_h$ on $D$ which is the limit of regularized versions of the Riemannian distance function associated with $e^{\gamma h } \, (dx^2 + dy^2)$~\cite{dddf-lfpp,gm-uniqueness}. This metric is expected to describe the scaling limit of graph distances on random planar maps, but this is only known for $\gamma=\sqrt{8/3}$~\cite{lqg-tbm1,lqg-tbm2,lqg-tbm3,legall-uniqueness,miermont-brownian-map}.} 
This measure is a special case of the theory of \emph{Gaussian multiplicative chaos (GMC)}, which was initiated in~\cite{kahane,hk-gmc}.
See~\cite{rhodes-vargas-review,berestycki-gmt-elementary,aru-gmc-survey} for an introduction to GMC theory, along with \cite{dkrv-lqg-sphere} for a physically motivated construction.
See also~\cite{shef-kpz} for a number of results which are specific to the LQG measure $\mu_h$.  
	
In the physics literature, random planar maps are used as discrete models for 2D quantum gravity. 
The heuristic connection between LQG surfaces and random planar maps comes by way of the so-called \emph{DDK ansatz}~\cite{david-conformal-gauge,dk-qg}.
This ansatz implies that LQG surfaces as defined above should be the same as samples from the ``Lebesgue measure on surfaces" weighted by the partition function of a certain matter field, and hence these surfaces should be related to ``weighted discrete random surfaces", i.e., random planar maps. 
See Section~\ref{subsec:scaling} for more detail. 
Further evidence for the connection between random planar maps and LQG can be obtained by matching formulas in the discrete and continuum setting. 
Indeed, results for random planar maps (e.g., the computations of various exponents) obtained using random matrix techniques~\cite{bipz-planarmaps} can be shown, at a physics level of rigor, to agree with corresponding results in continuum Liouville theory. 
In particular, Knizhnik, Polyakov, Zamolodchikov \cite{kpz-scaling} established a relation between scaling  exponents for statistical physics models in Euclidean and quantum environments using the continuum theory. This can be used to derive quantum exponents  based on their Euclidean counterparts.
In some cases it has been verified that results based on these  derivations agree with  the ones from random planar map calculations, 
see e.g.\ \cite{fgz-95} and references therein.
	
Because of the above correspondence in the physics literature, it is believed that planar maps should converge (in various topologies) to LQG surfaces; see Section~\ref{subsec:scaling} for more details. 
The parameter $\gamma$ depends on the type of random planar map under consideration. 
For example, uniform random planar maps are expected (and in some senses known) to converge to $\sqrt{8/3}$-LQG. The same is true if we place local constraints on the map, e.g., if we consider a uniform triangulation (a map in which all faces have 3 edges) or quadrangulation (a map in which all faces have 4 edges). 
For $\gamma\not=\sqrt{8/3}$, $\gamma$-LQG surfaces arise as the scaling limits of random planar maps sampled with probability proportional to the partition function of an appropriate (critical) statistical mechanics model on the map. For example, if we sample with probability proportional to the number of spanning trees, we get $\sqrt 2$-LQG. 
If we instead weight by the partition function of the critical FK cluster model with parameter $q\in (0,4)$, we get $\gamma$-LQG where $\gamma\in (\sqrt 2 , 2)$ satisfies $q = 2+2\cos(\pi\gamma^2/2)$.

It is natural to look at the scaling limits of statistical mechanics models on random planar maps in addition to just the underlying map. For many such models, at the critical point it is known or expected that the scaling limit of the statistical mechanics model should be described by one or more Schramm-Loewner evolution curves (SLE$_\kappa$)~\cite{schramm0} sampled \emph{independently} from the GFF-type distribution corresponding to the $\gamma$-LQG surface which is the limit of the random planar map. The case when $\kappa \in \{\gamma^2,16/\gamma^2\}$ describes the scaling limit of statistical mechanics models which are compatible with the weighting of the underlying random planar map. More precisely, if we look at a planar map weighted by the partition function of a statistical model which converges to SLE$_\kappa$ for one of these values of $\kappa$, then the joint law of the planar map and the statistical mechanics model on it should converge to the law of a $\gamma$-LQG surface decorated by an independent SLE$_\kappa$.

The mating-of-trees theorem of~\cite{wedges}  (see Theorem~\ref{thm-mating}) is a deep result in the theory of SLE and LQG which gives a way of encoding a $\gamma$-LQG surface decorated by an SLE$_\kappa$ curve for $\kappa=16/\gamma^2$ in terms of a correlated two-dimensional Brownian motion, with the correlation of the two coordinates given by $-\cos(\pi\gamma^2/4)$. 
As we will explain, this encoding is an exact continuum analog of the aforementioned mating-of-trees bijections for random planar maps.

The mating-of-trees theorem has a huge number of applications to random planar maps, LQG, and SLE, which we will review in Section~\ref{sec-applications}. Some highlights of these applications include the following. 
\begin{itemize}
	\item The first rigorous versions of the long-expected convergence of random planar maps toward LQG, due to the precise correspondence between mating of trees in the discrete and continuum settings. 
	\item The first scaling limit results for random planar maps conformally embedded in the plane. In fact, all current rigorously proven convergence results for random planar maps toward LQG use the mating-of-trees theorem.\footnote{Scaling limit results for random planar maps in the Gromov-Hausdorff topology have been obtained without using the mating-of-trees theorem (see e.g.\ \cite{legall-uniqueness,miermont-brownian-map}) but, as explained in Section \ref{sec-metric-pure-lqg}, identifying the limiting object with an LQG surface requires mating-of-trees theory.} 
	\item Computations and bounds for exponents related to graph distance, random walk and statistical mechanics models on random planar maps.  
	\item A general framework for computing Hausdorff dimensions and for constructing natural measures on random fractals associated with SLE. 
\end{itemize}

We remark that there are other approaches to the theory of LQG besides the one considered in this article.
In particular, David, Kupiainen, Rhodes, Vargas, and others have studied LQG from the path integral perspective, which is much more closely aligned with the original physics literature than the ideas surveyed in this article. See~\cite{dkrv-lqg-sphere,grv-higher-genus,krv-dozz} for results in this direction, \cite{vargas-dozz-notes} for a survey article on such results, and Section~\ref{subsec:rv} for further discussion.
A recent paper by Dub\'edat and Shen~\cite{ds-ricci-flow} defines a notion of a \emph{stochastic Ricci flow} under which the $\gamma$-LQG measure is invariant. 
This can be seen as an alternative approach to LQG theory based on stochastic quantization.

There is also a substantial physics literature on LQG and related topics, which is mostly outside the scope of this paper. We refer to~\cite{shef-kpz} for an extensive list of references to relevant physics literature.

\subsection{Guidance for reading}
\label{sec-outline}

The article has two main purposes.  First, it explains the motivation and main  ideas in the theory of mating of trees.  This is the focus of Sections~\ref{sec-bijections}, \ref{sec-sle-lqg}, and \ref{sec-mating}.
Second, it reviews various applications of mating of trees and further research directions.  For readers who are mainly interested in this aspect, we recommend reading Sections~\ref{subsec:UST} and~\ref{subsec:scaling}, skimming through the basic definitions and theorem statements in the rest of Sections~\ref{sec-sle-lqg} and~\ref{sec-mating}, and then reading whatever parts of Sections~\ref{sec-applications} and~\ref{sec-open-problems} that are of interest. We note that most of the different applications and open problems can be read independently of each other. 

We will try to assume as little background as possible, but the reader may have an easier time if he/she has some basic familiarity with the Gaussian free field (see~\cite{shef-gff,pw-gff-notes} and the introductory sections of~\cite{ss-contour,shef-zipper,ig1,ig4}) and the Schramm-Loewner evolution (see~\cite{lawler-book,werner-notes,bn-sle-notes}). 
See also~\cite{berestycki-lqg-notes} for introductory notes on the GFF and LQG and~\cite{gwynne-ams-survey} for a short introductory article on LQG aimed at a reader with no prior knowledge of the subject.

We intend for this article to be a useful reference for experts.
To this end, we have tried to make the reference list as comprehensive as possible and we have included precise statements of several useful lemmas about SLE and LQG which may be hard to find in the existing literature. 

This article is structured as follows. In Section~\ref{sec-bijections}, we present two of the simplest and most important mating-of-trees bijections: the Mullin bijection for spanning-tree weighted maps and the bijection of Bernardi, Holden, and Sun for uniform triangulations decorated by site percolation. 
This section is not needed to understand the continuum theory, but may provide useful intuition and motivation. 

In Section~\ref{sec-sle-lqg},  we first discuss the conjectured link between random planar maps and LQG and its motivations in Section~\ref{subsec:scaling}.
We then review the definition of the Gaussian free field, LQG, several special types of LQG surfaces (quantum cones, wedges, disks, etc.), and the definitions of ordinary and space-filling SLE curves. 
Along the way, we make note of some important properties of these objects which are established in various places in the literature. 
We give precise definitions of most objects involved for the sake of completeness, but it is not essential to  know these precise definitions in order to understand the rest of the paper.

In Section~\ref{sec-mating}, 
we first review the ``conformal welding" results for LQG surfaces cut by SLE curves, which were first established in~\cite{shef-zipper} and later generalized in~\cite{wedges} and which constitute the first rigorous connections between SLE and LQG.
We then state the continuum mating-of-trees theorem of~\cite{wedges} and outline its proof. We also state and discuss two variants of this theorem: a variant for LQG on the disk and a variant for an LQG surface cut by an ordinary (not space-filling) SLE$_\kappa$ curve for $\kappa \in (4,8)$. 

In Section~\ref{sec-applications} we review the applications of mating-of-trees theory, including the ones mentioned at the end of Section~\ref{sec-overview}. 
In Section~\ref{sec-open-problems}, we state and discuss several open problems in mating-of-trees theory, which should be of interest to readers with a variety of different backgrounds including mathematical physics, combinatorics, and complex analysis.

\subsubsection*{Acknowledgments} 
We thank Scott Sheffield for useful discussions. We thank Juhan Aru, Nathan\"ael Berestycki, Linxiao Chen, Gr\'egory Miermont, 
Yiting Li, Ellen Powell, Wei Qian, R\'emi Rhodes, Steffen Rohde, Lukas Schoug, Scott Sheffield, and an anonymous referee for helpful comments on an earlier version of this article.
E.G.\ was supported by a Clay Research Fellowship and a Junior Research Fellowship from Trinity College, Cambridge.
N.H.\ was supported by Dr.\ Max R\"ossler, the Walter Haefner Foundation, and the ETH Z\"urich Foundation. 
X.S.\ was supported by the Simons Foundation as a Junior Fellow at the Simons Society of Fellows, by NSF grant DMS-1811092, and by Minerva fund at the Department of Mathematics at Columbia University. 
N.H.\ was invited to contribute with a survey article in \emph{Panorama and Syntheses} in relation with her talk at the conference \emph{Etats de la recherche SMF: Statistical mechanics} of the French Mathematical Society in December 2018. She thanks the organizers for the invitation.  

\section{Discrete mating of trees}
\label{sec-bijections}
There are numerous combinatorial bijections between planar maps (possibly decorated by additional structures) and simpler objects such as trees and walks.
In \cite{shef-burger}, Sheffield observed that a bijection due to Mullin~\cite{mullin-maps} and its generalization due to Bernardi~\cite{bernardi-maps,bernardi-sandpile} can be interpreted as follows.
A random planar map decorated by a uniform spanning tree or an instance of the critical Fortuin-Kasteleyn (FK) cluster model can be bijectively encoded by a certain random walk on $\BB Z^2$ in such a way that the decorated map is obtained by ``gluing together'' the two trees whose contour functions are (minor variants of) the two coordinates of the walk.
In~\cite{shef-burger}, it is proved that these random walks converge in the scaling limit to 2D correlated Brownian motions with the correlation depending explicitly on the parameter $q \in (0,4)$ of the FK cluster model. 
Together with \cite{wedges}, this sets the foundation for mating-of-trees theory. Since then, several additional encodings of decorated planar maps by random walks along the same lines have been found. 
We call such encodings \emph{mating-of-trees bijections}.

In Sections~\ref{subsec:UST} and~\ref{subsec:site}, we describe  two mating-of-trees  bijections: the aforementioned Mullin bijection for spanning-tree-decorated random planar maps and the one for site-percolated loopless triangulations due to Bernardi, Holden, and Sun~\cite{bernardi-dfs-bijection,bhs-site-perc}. We choose to focus on these two bijections for three reasons. First, they are especially simple in comparison to other mating-of-trees bijections. Second, they provide an instrumental intuition for understanding the continuum picture of the mating-of-trees theory. Third, they are both fundamental in many later applications (see Section~\ref{sec-applications} for more details).
We will briefly review other known mating-of-trees bijections in Section~\ref{subsec:rmp} to give the reader an idea of the range of models this theory applies to. The reader may want to skip Section~\ref{subsec:site} in the first reading because much of the discrete intuition can be gained from Section~\ref{subsec:UST}.
For understanding some additional features of mating-of-trees theory for the case $\gamma\in(\sqrt{2},2)$, as well as its applications, Section~\ref{subsec:site} will be instrumental.

In Section~\ref{subsec:scaling}, 
after introducing the necessary background for LQG and SLE, we will explain why the bijections in this section lead to the concept of mating-of-trees theory in the continuum.

\subsection{Spanning-tree-decorated planar maps}\label{subsec:UST}

A planar map is called \emph{rooted} if there is a  marked directed edge, which we call  the \emph{root edge}. The \emph{root vertex} is the terminal endpoint of the root edge.
A (rooted) planar map with a single face is called a (rooted) planar tree. 
For a rooted planar tree $T$ with $n$ edges, its unique face has degree $2n$ (each of the $n$ edges has multiplicity 2). 
Suppose we trace along the boundary of this face started from the root edge.
We define $C_0 :=0$ and for each $i\in \{1,\dots,2n\}$, we let $C_i$ be the graph distance in $T$ from the $i$th vertex we encounter to the root vertex.
Then $i\mapsto C_i$ is a walk with $2n$ steps in $\{+1,-1\}$ which takes values in $\Z_{\ge 0}$ and starts and ends at the origin. 
This walk is called the \emph{contour function} of $T$. 
It is not hard to see that each walk with these properties is the contour function of a unique planar tree $T$ with $n$ edges, so we have a bijection between trees and walks. 
See~\cite{legall-tree-survey} for more on this bijection.

Given a finite connected graph $G$, 
a \emph{spanning tree} of $G$ is a connected subgraph of $G$ without cycles whose vertex set agrees with that of $G$. 
Let $\MT^n$ be the set of triples of the form $(M,e_0 ,T)$ where $M$ is  planar map with $n$ edges  rooted at the edge $e_0$ and $T$ is a spanning tree of $M$. 
Let $\LW^n$ be the set of  walks on $\Z^2_{\ge 0}$ of $2n$ steps, with steps in $\{(1,0),(-1,0),(0,1),(0-1) \}$, starting and ending at the origin.
The Mullin bijection is a bijection between $\MT^n$ and $\LW^n$, which is a two-dimensional analog of the contour function of a rooted planar tree in the previous paragraph.

\begin{figure}[ht!]
	\begin{center}
		\includegraphics[scale=.8]{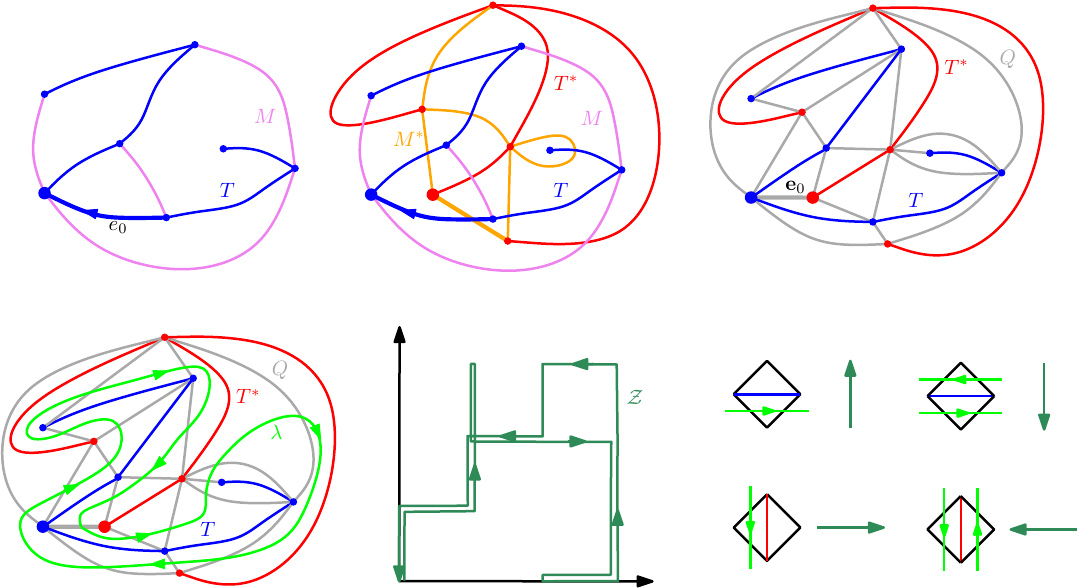} 
		\caption[Mullin bijection]{
			\textbf{Top left:} a planar map $M$ with an oriented root edge $e_0$ and a spanning tree $T$. 
			\textbf{Top middle:} the dual map $M^*$ and the dual spanning tree $T^*$. 
			\textbf{Top right:} the quadrangulation $Q$ (whose vertices correspond to the vertices and faces of $M$) with its root edge $\BB e_0$. Note that the rightmost edge of $T$ corresponds to two prime ends of the face to its left, which gives rise to a double edge of $Q$.
			\textbf{Bottom left:} the space-filling path $\lambda$ on $\mcl T$.  
			\textbf{Bottom middle:} the lattice walk $\cZ$ corresponding to $(M,e_0,T)$. The different segments of the path are slightly shifted to make it easier to read off the walk.
			\textbf{Bottom right:} correspondence between steps of $\cZ$ and triangles of $Q\cup T\cup T^*$.
		}\label{fig-burger-bijection}  
	\end{center}
\end{figure}

We now explain the Mullin bijection; see Figure~\ref{fig-burger-bijection} for an illustration.
Given $(M , e_0 , T)\in \MT^n$, let $M^*$ be the dual map of $M$ (whose vertices correspond to faces of $M$) and let $T^*$ be the dual spanning tree of $M^*$, which consists of edges of $M^*$ which do not cross edges of $T$. Let $Q$ be the graph whose vertex set is the union of the vertex sets of $M$ and $M^*$, with two such vertices joined by an edge if and only if they correspond to a face of $M$ and a vertex incident to that face. If $v$ is a vertex of $M$ which corresponds to $k\in\BB N$ prime ends on the boundary of the face $f\in M^*$, then there are $k$ edges joining $v$ and $f$, one for each prime end. We define the root edge $\BB e_0$ of $Q$ to be the edge which is adjacent to the terminal endpoint of $e_0$ and which is the first edge in the counterclockwise direction among all such edges. Then $Q$ is a quadrangulation, and each face of $Q$ is bisected by either an edge of $T$ or an edge of $T^*$, which divides it into two triangles. Let $\mcl T$ be the set of such triangles, and view $\mcl T$ as the vertex set of a graph whereby two triangles are joined by an edge if they share a common edge of $Q\cup T\cup T^*$ (i.e., $\mcl T$ is the dual map of $Q\cup T \cup T^*$). 

There is a unique path $\lambda : \{1,\dots,2n\} \rta \mcl T$ which snakes between the trees $T$ and $T^*$, begins and ends at the midpoint of $\BB e_0$, traverses each triangle of $\mcl T$ exactly once, and always keeps $T$ to its right and $T^*$ to its left (see Figure~\ref{fig-burger-bijection}).  
We call $\lambda$ the \emph{peano curve} of $(M,e_0,T)$. 

We now use $\lambda$ to construct a walk $\cZ \in \LW^n$ associated with $(M,e_0,T)$. 
The peano curve $\lambda$ crosses each face of $Q$ twice (once for each triangle). 
Set $\cZ(0)=(0,0)$.
For each integer $i\in [1,2n]$, let $\lambda(i)$ be the $i$-th triangle traversed by $\lambda$ and $q(i)$ be the quadrilateral of $Q$ containing $\lambda(i)$.
If $q(i)$ is  bisected by an edge of $T$ (resp. $T^*$) and $\lambda(i)$ the first triangle traversed by $\lambda$ among the two triangles in $q(i)$, we set $\cZ(i) - \cZ(i-1) = (0,1)$ (resp. $\cZ(i) - \cZ(i-1) = (1,0)$). 
If $q(i)$ is  bisected by an edge  of $T$ (resp. $T^*$)  and $\lambda(i)$ the second triangle traversed by $\lambda$ among the two triangles in $q(i)$,  we set $\cZ(i) - \cZ(i-1) = (0,-1)$ (resp. $\cZ(i) - \cZ(i-1) = (-1,0)$). 

With the above definition, the first (resp.\ second) coordinate of $\cZ$ is the same as the contour function of $T$ (resp.\ $T^*$), except with some extra constant steps. 
As a consequence, we have $\cZ \in \LW^n$. 
We call $\cZ$ the \emph{contour function} of $(M,e_0,T)$.
The following is easily verified.

\begin{thm}[Mullin bijection] \label{thm:Mullin}
	The mapping $(M,e_0,T)\mapsto \cZ$ is  a bijection between $\MT^n$ and $\LW^n$.
\end{thm}

Theorem~\ref{thm:Mullin} was first observed by Mullin~\cite{mullin-maps}. 
Our formulation is essentially due to Sheffield~\cite{shef-burger}, and is equivalent to the formulation due to Bernardi~\cite{bernardi-maps}.  

We now discuss the infinite-volume variant of the Mullin bijection. 
This bijection is in some ways easier to work with than the finite-volume version since the corresponding random walk is just a bi-infinite simple random walk, with no conditioning. It is a discrete analog of the infinite-volume mating-of-trees theorem for SLE/LQG (see Section~\ref{sec-mating-main}).

Let $(M^n,e^n_0,T^n)$ be a uniform sample from $\MT^n$. 
Using the local nature of the Mullin bijection, it can be shown that for each integer $r>0$,  the graph distance ball of radius $r$ centered at the root vertex of $M^n$ has a weak limit, viewed as a planar map marked with a directed edge and an edge subset \cite{chen-fk}. 
(This is a generalization of the so-call Benjamini-Schramm convergence \cite{benjamini-schramm-topology}.)
As $r$ varies, the limiting objects consistently define a triple $(M^\infty,e^\infty,T^\infty)$ where $e^\infty $ is a marked directed edge on an infinite planar map $M^\infty$, and $T^\infty$ is an edge subset of $M^\infty$. It is also easy to see that  almost surely $M^\infty$ is one ended and $T^\infty$ is a spanning tree on $M^\infty$.
We call the random planar map $(M^\infty,e^\infty,T^\infty)$ the \emph{infinite spanning-tree-decorated map}.

We now summarize some basic facts about the mating-of-trees encoding of the infinite spanning-tree-decorated map which can be found, e.g., in~\cite{chen-fk,blr-exponents,gms-burger-cone}.
Given $(M^\infty,e^\infty,T^\infty)$, we can perform the construction of the Mullin bijection as in the finite-volume case by considering the dual tree of $T^\infty$ and define the adjacency graph of triangles $\cT^\infty$ analogously to $\cT$. Then  we can still consider the Peano curve $\lambda^\infty$ between $T^\infty$ and its dual tree as a path on $\cT^\infty$. In this setting,   $\lambda^\infty$ can be parametrized as  a function $\Z \rta \cT^\infty$ instead of $\{1,\dots,2n\} \rta \cT$.
To fix the parametrization, we require $\lambda^\infty(1)$ to be the triangle in $\cT^\infty$ on the right side of $e^\infty$.
By defining the increments in the same way as in the finite-volume case,  $\lambda^\infty$ produces a contour function $\cZ^\infty$, whose laws is a two-sided simple random walk $\cZ^\infty$ on $\Z^2$ with $\cZ^\infty(0)= (0,0) $.
This encoding is also bijective in the following sense. 

\begin{prop}\label{prop:Mullin-inf}
	$(M^\infty,e^\infty,T^\infty)$ and $\cZ^\infty$ almost surely determine each other.
\end{prop}

\newcommand{\TP}{\mathcal{TP}}

\subsection{Site-percolated loopless triangulations}\label{subsec:site}

In this section we review the bijection of Bernardi, Holden, and Sun \cite{bernardi-dfs-bijection,bhs-site-perc}. We will focus on the disk version. Our presentation  is close to \cite[Section~2]{ghs-metric-peano}. 

Given a rooted planar map, we call the face to the right of the root edge the \emph{root face}.
A planar map $M$ is called a \emph{triangulation with boundary} if every  face of $M$  has degree 3 
except possibly the root face. Given a triangulation with boundary, we always embed it in the plane  so that the root face is the unbounded face, namely, the face containing $\infty$.
This way, edges and vertices on the root face naturally form the \emph{boundary} of $M$, which we denote by $\bdy M$.

A graph is called \emph{2-connected} if removing any vertex does not disconnect the  graph. 
If a  triangulation with boundary is 2-connected, then it has no self-loops (i.e., edges whose two endpoints coincide) and  its boundary is a simple curve. 
For an integer $\ell\ge 2$, let $\frk T(\ell)$ be the  set of such maps whose boundary has $\ell$ edges.  
By convention, we view a map with a single edge as an element in $\frk T(2)$ and call it \emph{degenerate}. 
Given $(M, e)\in \cup_{\ell\ge 2} \frk T(\ell)$, a \emph{site percolation} on $M$ is a coloring $\omega$ of its vertices in two colors, say, red and blue. 
The coloring of the boundary vertices is called the \emph{boundary condition} of $\omega$. We say that $\omega$  has \emph{dichromatic boundary condition} if the following condition  is satisfied. The tail (resp.\ head) of $e$ is red (resp.\ blue), and, moreover, if $M$ is non-degenerate, there exists a unique edge $\wh e \neq e$ on $\bdy M$ such that the colors of its two endpoints are  different. 
We call $\wh e$ the \emph{target edge} of $(M,e,\omega)$. 

Let $\DP$ be the set of triples $(M,e,\omega)$  where $(M, e)\in \cup_{\ell\ge 2} \frk T(\ell)$ and $\omega$ is a site percolation on $M$ with dichromatic boundary condition. 
For each $(M, e,\omega)\in \DP$, we now associate a total ordering $\prec$ on the edge set $\cE(M)$ of $M$. 
We will do it iteratively via an induction on $\#\cE(M)$, where $\#\cE(M)$ is the cardinality of $\cE(M)$. 
The construction is closely related to the so-called \emph{peeling process} of $(M,\omega)$~\cite{angel-peeling,curien-legall-peeling}. But, instead of just exploring a single interface between red and blue vertices, it enters the bubbles which are disconnected from $\infty$ by the interface.
See Figure~\ref{fig-site-bij} for an illustration.

The case $\#\cE(M)=1$, where $M$ is degenerate, is trivial.
For $\#\cE(M)>1$, we first declare that the marked edge $e$ is the smallest edge in the ordering  $\prec$.
Let $t$ be the unique triangle of $M$ which is incident to $e $ and let $v$ be the vertex on $t$ which is not an endpoint of $e$. Such a vertex exists since $M$ has no self-loops. If $v$ is not in $\bdy M$, then $M \setminus \{e\}$ is still a 2-connected planar map. We set $M' := M\setminus \{e\}$.  
In this case, we let  $e'$ be the edge on $t$ other than $e$ for which the two endpoints have opposite colors.
If $v\in\bdy M$, then $M \setminus \{e\}$ is connected but has two \emph{2-connected components}, i.e., $M \setminus \{e\}$ has two 2-connected maximal subgraphs. In this case, we let $M'$ be the $2$-connected component of $M \setminus \{e\}$ containing the target edge $\wh e$ and let $e'$
be the edge shared by $t$ and $M'$.  
 
In both cases, we let $\omega'$ be the restriction of $\omega$ to $\cV(M')$, where $\cV(\cdot)$ denotes the vertex set. We orient  $e'$ such that $(M',e',\omega')\in \DP$ and require that the restriction of $\prec$ to $\cE(M')$ is defined by $(M',e',\omega')$  via the inductive hypothesis, which we can do since $(M',e',\omega')$  belongs to $\DP$ and has at most $\#\cE(M)-1$ edges.

If $M \setminus \{e\}$ has two 2-connected components, let $M''$ be the component other than $M'$. Let   $e''$
be the edge shared by $t$ and $M''$.  
Define a coloring  $\omega''$ on $\cV(M'')$ by  letting $\omega''=\omega$ on $\cV(M'')\setminus\{v\}$ and requiring that $\omega''(v) \not= \omega(v)$.  
We  orient  $e''$ such that $(M'',e'',\omega'')\in \DP$.
We require that the restriction of $\prec$ to  $\cE(M'')$ is defined by  $(M'',e'',\omega'')$ via the inductive hypothesis.
Moreover, $e_2\prec e_1$ for all $e_1\in \cE(M')$ and $e_2\in \cE(M'')$. 

These rules allow us to inductively define $\prec$ on $\cE(M)$. We note that $e$ and $\wh e$ are the first and last, respectively, edges in this ordering.  
\begin{definition}\label{def:Peano-site}
	Given  $(M, e,\omega)\in \DP$, for $1\le i\le \#\cE(M)$,  let $\lambda(i)=e_i$ where $e_i$ is the $i$-th edge in $\cE(M)$ according to $\prec$. We call $\lambda$ the \emph{Peano curve} of $(M,e,\omega)$.
\end{definition}
In the setting of Definition~\ref{def:Peano-site}, for $0\le k<\#\cE(M)$, 
once edges in $\lambda([1,k] \cap \Z )$ are  removed from $M$, the remaining map $M_k$ is still a triangulation with boundary where both $e_{k+1}$ and $\wh e$ are on $\bdy M_k$ (note that $M_0=M$ and that $\bdy M_k$ is not necessarily simple for $k\geq 1$).
Therefore we can define the left and right boundary between $e_{k+1}$ and $\wh e$, and their boundary lengths $\cL_k$ and $\cR_k$, respectively. 
When $\bdy M_k$ is a simple curve, $\cL_k$ and $\cR_k$ are simply the number of edges on the clockwise and counterclockwise, respectively, arcs from $e_{k+1}$ to $\wh e$, not counting $e_{k+1}$ and $\wh e$.  
See \cite[Section~2]{ghs-metric-peano} for a more detailed description.
Set $\cZ(k)=(\cL(k),\cR(k))$ for $0\le k<\#\cE(M)$ and  note that $\cZ(\#\cE(\cM)-1)=(0,0)$. We call $\cZ$ the \emph{boundary length process} of $(M,e,\omega)$.

\newcommand{\KW}{\cK\cW}

For integer $\ell\ge 2$, let $\TP^\ell$ be the set of triples  $(M, e,\omega)$ where $(M,e)\in \cup_{\ell\ge 2} \frk T(\ell)$, and $\omega$ is  a site percolation on  $M$ with  \emph{monochromatic} boundary condition, namely,  all boundary vertices of $M$ have the same color.
If $(M, e,\omega)\in \TP^\ell$, by flipping the color of one of the endpoints of $e$ so that its tail is red and its head is blue, we can identify $(M, e,\omega)$ with an element in $\DP$. 
Therefore we can associate a Peano curve $\lambda$ and a boundary length process $\cZ$ to $(M,e,\omega)$. 
Let $ \KW^\ell $ be the set of walks  on $[0,\infty)^2$ taking steps in $\{(1,0), (0,1), (-1,-1) \}$, starting from $(\ell-2,0)$ or $(0,\ell-2)$, and ending at $(0,0)$.

\begin{figure}
	\centering
	\includegraphics[scale=0.92]{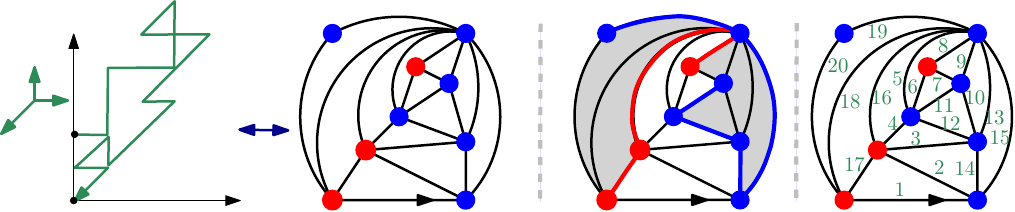}
	\caption{ \label{fig-site-bij}
	{\bf Left}: Illustration of the bijection in Proposition \ref{prop:site-mating}. 
	{\bf Middle}: The planar map on gray background is $M_6$, and the boundary length process satisfies $(\cL (6) ,\cR(6) )=(3,5)$ since the length of the red (resp.\ blue) path is 3 (resp.\ 5).
	{\bf Right}: The edges of the map are labeled based on the order in which they are hit by the discrete Peano curve $\lambda$. 
}
\end{figure}

The following is a consequence of \cite[Corollary 2.12]{bhs-site-perc}.
\begin{prop}\label{prop:site-mating}
	In the setting just above,  $(M, e,\omega)\mapsto \cZ$ is a bijection between $\TP^\ell$ and $\KW^\ell$.
\end{prop}
The inverse bijection can also be described easily in an explicit way by associating each of the three steps $(1,0), (0,1), (-1,-1)$ with a certain operation on a planar map, and then building the map dynamically by performing the operations associated with the steps of the walk one by one. See e.g.\ \cite[Figures 3 and 4]{bhs-site-perc}.
\begin{remark}\label{rmk:sphere}
	A triangulation  without self-loops can be identified with an element in $\frk T(2)$ by splitting the root edge into two edges with common endpoints.
	This identification and Proposition~\ref{prop:site-mating} give a bijection between percolated triangulations  and  $\KW^{2}$, which is the sphere version of the Bernardi-Holden-Sun bijection. 
\end{remark}
\begin{remark}
	Walks in the first quadrant $[0,\infty)^2$ starting from the origin and with steps $(0,-1)$, $(-1,0)$, and $(1,1)$ (i.e., the reversal of the steps we consider) are called \emph{Kreweras walk}s. Bernardi \cite{bernardi-dfs-bijection} found a bijection between Kreweras walks ending at the origin of length $3n$ and bridgeless cubic maps of size $n$, decorated by a depth-first search tree. This gave a combinatorial proof of a formula discovered by Kreweras \cite{kreweras65} for the number of Kreweras walks. The bijection in \cite{bhs-site-perc} is a generalization of the bijection in \cite{bernardi-dfs-bijection} where the depth-first search tree in \cite{bernardi-dfs-bijection} can be interpreted as an exploration of the percolation.
\end{remark}

Given $(M, e,\omega)\in \DP$, there is an interface between the  two clusters containing the red and  blue boundary arcs between $e$ and $\wh e$, which we call the \emph{percolation interface} of $\omega$.  Interfaces are  convenient observables when considering the scaling limit of percolation.
For some $\ell \in \BB N$ consider $(M,a,\omega)\in \TP^\ell$ and let $\lambda$ denote the Peano curve as in Definition~\ref{def:Peano-site}. 
Let $b\neq a$  be a boundary edge of $M$ and let $\omega_a^b$ be the site percolation obtained by flipping the color on one of the arcs between $a$ and $b$ such that $(M,a,\omega_a^b)\in \DP$ and $b$ is the target edge.  
Let $\lambda^{ab}$ be the percolation interface of $(M,a,\omega_a^b)\in \DP$.  As explained in~\cite[Section~2.2]{ghs-metric-peano}, $\lambda^{ab}$ can be viewed as an ordered set of edges and 
$\lambda$  traces $\lambda^{ab}$ in the same order as $\lambda^{ab}$. Moreover,  after removing $\lambda^{ab}$ from $M$, given two edges $e_1$ and $e_2$ in two different 2-connected components, it is easy to tell which one is visited by $\lambda$ first by inspecting the boundary condition of $\omega$ restricted to these components.  
This is the discrete intuition behind the definition of space-filling SLE in Section~\ref{sec-space-filling}.

There is a particularly natural probability measure on loopless triangulations with simple boundary called the \emph{(critical) Boltzmann measure}.  Given some perimeter $\ell\geq 2$, it assigns weight proportional to $(2/27)^n$\footnote{The number of  loopless triangulations with $n$ vertices grows as $(27/2)^n(1+o_n(1))$. See e.g. \cite{angel-schramm-uipt}.} to each triangulation with simple boundary having $n$ vertices. 
Fix $\ell$ and consider a Boltzmann triangulation with perimeter $\ell$ and Bernoulli-$\frac 12$ site percolation, i.e., a uniform and independent coloring of the inner vertices in red or blue. We assume the boundary is monocolored either red or blue. 
If we apply the bijection in Proposition~\ref{prop:site-mating} to this object, then the right side will be a simple random walk in the first quadrant with step distribution uniform on $\{(1,0), (0,1) , (-1,-1)\}$, conditioned to start at $(\ell-2,0)$ or $(0,\ell-2)$ (depending on the boundary data) and end at the origin.

One consequence of the previous paragraph is that the infinite-volume limit of the Boltzmann loopless triangulation exists. We call this the (type II) \emph{uniform infinite planar triangulation (UIPT)}. The existence of the UIPT was first proved by Angel and Schramm~\cite{angel-schramm-uipt}. The proof of its existence based on the Bernardi-Holden-Sun bijection is done in \cite[Section~8]{bhs-site-perc}. It is also proved there that there is an infinite-volume variant of the bijection in Proposition \ref{prop:site-mating}. In this infinite-volume bijection, walks with i.i.d.\ steps $(1,0),(0,1),(-1,-1)$ are related to the site-percolated UIPT.

\section{SLE and LQG}
\label{sec-sle-lqg}
In this section  we introduce the necessary background for continuum  mating-of-trees theory, 
including the motivation from the scaling limit of decorated random planar maps, and the definitions of the GFF, LQG, certain special LQG surfaces, and SLE. 
We give precise definitions of most of the objects involved for the sake of completeness, but if the reader only wants to understand the main ideas in Sections~\ref{sec-mating} and~\ref{sec-applications}, he or she can read only Section~\ref{subsec:scaling} and skim the rest of this section.  

Throughout this section we apply the following notions and conventions. 
Given two random variables $X$ and $Y$, $X\eqD Y$ means that they agree in law. 
If we say that $X$ is almost surely  determined by $Y$ we mean that  $X$ and $Y$ are defined on the same probability space and $X=f(Y)$ a.s.\ for some measurable function $f$. 
The upper half-plane is denoted by $\BB H$ and the unit disk is denoted by $\D$.
Given a planar domain $D$, $\p D$ is understood as the set of prime ends in complex analysis \cite{pom-book} and $\ol D=D\cup\p  D$. Unless explicitly mentioned otherwise, for a simply connected domain $D$ we assume that $\p D$ can be parametrized as a closed curve on the Riemann sphere $\CC\cup\{\infty\}$, so that conformal maps from $\D$ to $D$ can be continuously extended to $\p\D$. Note that we do not require that $\bdy D$ is a \emph{simple} curve.

\subsection{SLE on LQG as the scaling limit of decorated random planar maps}\label{subsec:scaling}

In this section, we review some precise scaling limit conjectures for random planar maps decorated with statistical mechanics models, which motivates the theory of Liouville quantum gravity and mating of trees.
The continuum objects mentioned in this section will be  defined in the later subsections.

Liouville quantum gravity (LQG) is a one-parameter family of random geometries which describe 2D quantum gravity coupled with matter.
For simplicity, in most of this paper we focus on the case when the underlying topological surface is simply connected, but see Section~\ref{subsec:rv} for discussion and references concerning more general topologies.
From the differential geometry point of view, we consider a parameter $\mathbf c \in (-\infty,1)$ which we call the \notion{matter central charge} (i.e., the central charge of the matter field which can be coupled with the LQG). 
Heuristically speaking, for a simply connected topological surface $S$, the LQG surface with the $S$-topology and matter central charge $\mathbf c$ is the measure on all 2D Riemannian manifolds homeomorphic to $S$ whose probability density with respect to the ``Lebesgue measure on surfaces homeomorphic to $S$" is proportional to $(\det \Delta)^{-\mathbf c/2}$, where $\Delta$ is the Laplace-Beltrami operator. The determinant  $(\det  \Delta)^{-\mathbf c/2}$ should be understood as the partition function of a statistical mechanics model in 2D conformal field theory (see e.g.,~\cite{bpz-conformal-symmetry,dms-cft-book}) with central charge $\mathbf c$.

By formal differential geometric considerations, Polyakov~\cite{polyakov-qg1}, David~\cite{david-conformal-gauge}, and Distler-Kawai~\cite{dk-qg} argued that this random Riemannian geometry, after applying the uniformization theorem,  can be realized on Riemann surfaces with $S$-topology endowed with a volume form $e^{\gamma h} \, d^2z $, 
where $h$ is a variant of Gaussian free field and $\gamma$ is the unique solution in $(0,2)$ of the equation
\begin{equation}\label{eq:parameter}
\mathbf c=25-6\Big(\frac{\gamma}2+\frac2\gamma\Big)^2.
\end{equation}
Since \eqref{eq:parameter} gives a bijection between $\mathbf c\in(-\infty,1)$ and $\gamma\in(0,2)$, we can parametrize LQG using $\gamma$ instead of $\mathbf c$ and talk about $\gamma$-LQG. Note that $\gamma=\sqrt{8/3}$ corresponds to $\mathbf c = 0$ and $\gamma=\sqrt 2$ corresponds to $\mathbf c = -2$.

Polyakov's reasoning is hard to make rigorous directly, as it involves ``measures'' on the infinite dimensional space  of Riemannian manifolds. As in the construction of Wiener measure using simple random walk,  we can use measures on planar maps to approximate LQG surfaces. 
For example, let  $(M^n,e^n)$ be sampled from the uniform measure on all rooted planar maps with $n$ edges.
We can think of $M^n$ as a discrete uniform random surface by endowing each face with the surface structure of a polygon.
From this perspective, it is natural to believe that when $n$ is large, $M^n$ is a good approximation of LQG with the spherical topology and matter central charge $\mathbf c=0$  (equivalently, $\gamma=\sqrt{8/3}$).

As another example, let $(M^n,e^n,T^n)$ be a uniform sample from the set $\MT^n$ of edge-rooted, spanning tree-decorated planar maps with $n$ edges as in Section~\ref{subsec:UST}.
Then the marginal law of $(M^n,e^n)$ is the uniform measure on all rooted planar maps with $n$ total edges reweighted by the number of spanning trees they admit. 
Conditioning on $(M^n,e^n)$, the conditional law of $T^n$ is the uniform measure on the spanning trees of $M^n$, i.e., the UST on $M^n$.
By Kirchhoff's theorem the number of spanning trees on $M^n$ is equal to $\det \Delta_{M^n} $, where $\Delta_{M^n}$ is the graph Laplacian matrix of $M^n$.
Therefore it is natural to conjecture that as $n\to\infty$, $(M^n,e^n)$  is a good approximation of LQG with the spherical topology and $\mathbf c=-2$ (equivalently, $\gamma=\sqrt{2}$).

For each $\mathbf c\in(-\infty,1)$, there are statistical mechanics models whose partition functions are expected to be asymptotically equivalent to  $(\det  \Delta)^{-\mathbf c/2}$ (see e.g.,~\cite{bpz-conformal-symmetry,dms-cft-book}). 
For example, the critical Ising model corresponds to $\mathbf c=1/2$ (i.e., $\gamma=\sqrt 3$).
This generates a family of conjectures on the convergence  of random planar maps to LQG. 

To make a precise statement, let  $(M^\infty,e^\infty,T^\infty)$ be a sample of the infinite spanning-tree-decorated planar map as in Section~\ref{subsec:UST}. 
For each face in $M^\infty$, we decompose it into triangles by drawing an edge from its center to its vertices, 
and endow each triangle with the surface structure of an equilateral triangle with unit side length. 
This makes  $M^\infty$ a piecewise linear manifold with conical singularities.
As proved in \cite{gill-rohde-type},  $M^\infty$ is conformally equivalent to $\CC$. That is, 
there is a unique conformal map from  $M^\infty$ to the complex plane $\CC$ modulo complex affine transformations.
For $n\in\N$, let $\mu^n$ be the measure on $\CC$  which is the pushforward of $n^{-1}$ times the counting measure on the vertex set of $M^\infty$.
We further require that the tail of $e^\infty$ is mapped to $0$ and $\mu^n(\D)=1$, which fixes $\mu^n$ up to rotation. The scaling limit conjecture can be stated as follows. See, e.g.,~\cite{shef-zipper,legall-sphere-survey,dkrv-lqg-sphere,curien-glimpse} for similar conjectures.  

\begin{conj}\label{conj:lqg}
	There exists a variant of the Gaussian free field, denoted by $\fh$, such that 
	as $n\to\infty$, $\mu^n$  converges in law  to  $\mu_\fh=e^{\sqrt2 \fh} \,d^2 z$ with respect to the vague topology, modulo rotations.  
\end{conj}

The  field $\fh$ should be a certain embedding of the so-called \notion{$\sqrt2$-quantum cone}, see Definition~\ref{def:cone} and Lemma~\ref{lem:circle}.
It is only a generalized function (rather than a true function), thus $\mu_\fh=e^{\sqrt2 \fh} d^2 z$ has to be defined via regularization. See Section~\ref{subsec:GFF}. 
By ``convergence modulo rotations" we mean that for each $n\in\BB N$ there is a random variable $\theta_n \in [0,2\pi)$ such that $\mu^n(e^{i \theta_n} \cdot)$ converges in law to $\mu_\fh$.

There are a number of variants of Conjecture~\ref{conj:lqg}.  If we weight our planar maps by a statistical mechanics model with central charge $\mathbf c\in(-\infty,1)$, 
then $\fh$ will be the $\gamma$-quantum cone with $\gamma$ as in  \eqref{eq:parameter} and $\mu_\fh$ will be $e^{\gamma \fh}  d^2 z $. For example, we can consider the UIPT as introduced in Section \ref{subsec:site}, which is in the $\mathbf c=0$ universality class.
We can also replace the uniformization map by other discrete approximations of a conformal map such as a circle packing or Tutte embedding. Conjecturally, this will not change the limiting object.
We can also consider surfaces with other topologies, such as the sphere, disk, or torus topologies.
As in the whole-plane case, there are explicit descriptions of the variant of the GFF which one should get in the limit (see Sections~\ref{subsec:cone-wedge} and~\ref{subsec:finite}).
In the case of surfaces with boundary, one also expects the convergence of the counting measure on vertices of the boundary of the map to the $\gamma$-LQG boundary length measure.

There is also a metric variant of Conjecture~\ref{conj:lqg}: the scaling limit of the graph distance on the embedded planar map is supposed to be the Riemannian distance function $\frk d_\fh$ associated with the Riemannian metric tensor $ e^{2\gamma\fh/\dim_\gamma}(d^2x+d^2y)$ where $\dim_\gamma$ is a $\gamma$-dependent constant, equal to the Hausdorff dimension of the metric space $(\BB C , \frk d_\fh)$. The metric $\frk d_\fh$ has been constructed recently in~\cite{dddf-lfpp,gm-uniqueness}.
The metric properties of random planar maps and LQG are mostly outside the scope of the present article, but see Sections~\ref{sec-strong-coupling} and~\ref{sec-kappa6} for some results related to graph distances in random planar maps.

It is believed that lattice models which have conformally invariant scaling limits on regular lattices also have conformally invariant scaling limits on the random lattice defined by a conformally embedded random planar map. Furthermore, the scaling limit result should hold in a quenched sense (i.e., the conditional law of the statistical mechanics model given the random planar map should converge). This means in particular that the limit of the statistical mechanics model should be independent of the GFF-type distribution describing the limit of the random planar map. 

For example, it is believed that the simple random walk on a large class of conformally embedded random planar maps converges to \emph{Liouville Brownian motion}, the natural quantum time parametrization of Brownian motion on an LQG surface which was introduced in~\cite{berestycki-lbm,grv-lbm}. This convergence is rigorously proven for one type of random planar map in~\cite{gms-tutte,bg-lbm}; see Section~\ref{sec-tutte-conv}. Other models one could consider where the convergence result on a regular lattice has been established include critical percolation on the triangular lattice \cite{smirnov-cardy}, the Ising model \cite{smirnov-ising}, the loop erased random walk/uniform spanning tree \cite{lsw-lerw-ust}, and level lines of the discrete Gaussian free field \cite{ss-dgff}. For one of these models the analogous result on a random planar map has also been established: the Peano curve on a uniform site-percolated triangulation converge to a space-filling $\SLE_6$ under the Cardy embedding~\cite{ghs-metric-peano,hs-cardy-embedding} (see Section~\ref{sec-kappa6}). For many models whose partition function is expected to be asymptotically equivalent to $(\det\Delta)^{-\mathbf c/2}$, such as random cluster models and O($n$) loop models, there are natural ways to find curves which conjecturally have $\SLE_\kappa$- or $\SLE_{\ul\kappa}$-type curves as the scaling limit, where 
\begin{equation}\label{eq:parameter2}
	\kappa=16/\gamma^2 \quad\textrm{and}\quad \ul\kappa =\gamma^2 \qquad\textrm{with }\gamma\textrm{ as in }\eqref{eq:parameter}.
\end{equation}	

It is particularly natural to consider the statistical mechanics model on the random planar map whose partition function is used for the reweighting of the map. 
In this case, we often have exact Markov properties which are useful in the study of the planar map. For example, in the case of the infinite spanning-tree decorated random planar map $(M^\infty,e^\infty,T^\infty)$ studied above, the union of the infinite branch in $T^\infty$ starting from $e^\infty$ and the infinite branch in the dual tree also starting from $e^\infty$ divide $(M^\infty,T^\infty)$ into two \emph{independent} connected components. 
The continuum analogue of this property follows from the theory of \emph{conformal welding} of quantum surfaces (Lemma \ref{lem-past-future-ind}) and is essential in the mating-of-trees theory.

We illustrate the general conjecture about statistical physics models on random planar maps using the UST.
Let us start from the case of a regular lattice. Consider the UST on $[-N,N]^2\cap \Z^2$. 
For any fixed $n\in\BB N$, the intersection of this UST with $[-n,n]^2$ converges in the total variation sense as $N\rta\infty$~\cite{pemantle-ust}. 
This defines a limiting law on random subsets of $\BB Z^2$ which is a.s.\ supported on spanning trees of $\BB Z^2$.
We call this object the UST on $\BB Z^2$.
Given  a sample of the UST on $\Z^2$, we can define its dual tree and the associated Peano curve as in Section~\ref{subsec:UST}.  Lawler, Schramm, and Werner~\cite{lsw-lerw-ust} showed that if we rescale space by $\ep$ and send $\ep\rta 0$, then the Peano curve of the UST on $\BB Z^2$ converges in law to a random space-filling curve on $\CC$ which is called the \emph{whole-plane SLE$_8$ from $\infty$ to $\infty$}.\footnote{In \cite{lsw-lerw-ust} the result is proved in the setting of a finite domain with two boundary points, from which the whole-plane case  can be deduced, see~\cite{hs-euclidean}.} See Definition~\ref{def:whole-plane}.

\begin{conj}\label{conj:scaling}
	In the setting of Conjecture~\ref{conj:lqg}, let $\lambda^{\infty,n}$ be the curve on $\CC$ which is the image of the Peano curve $\lambda^\infty$  of $T^\infty$ under the uniformization map.
	Let $\eta$ be a whole-plane $\SLE_8$ from $\infty$ to $\infty$ independent of $\mu_\fh$. 
	Viewing $\lambda^{\infty,n}$ and $\eta$ as curves modulo monotone reparametrizations, $(\mu^n, \lambda^{\infty,n})$ converges jointly in law to  $(\mu_\fh,\eta)$ modulo rotations.
\end{conj}

We emphasize that $\lambda^{\infty,n}$ is not independent from the planar map $(M^\infty , e^\infty)$ or the measure $\mu^n$, but nevertheless the limiting objects $\mu_\fh$ and $\eta$ are independent.
The reason why we expect this to be the case is that the conditional law of $\lambda^{n,\infty}$ given $(M^\infty,e^\infty)$ should converge to the law of $\eta$. 

\begin{remark}\label{rmk:conj}
	If we view $ \lambda^{\infty,n}$ and $\eta$ as parametrized curves, then  the most natural way to parametrize $\lambda^{\infty,n}$ is to require $\lambda^{\infty,n}(0)=0$ and let $\lambda^{\infty,n}$ traverse one unit of $\mu^n$-mass in each unit of time. Under this parametrization, we expect that $\lambda^n$ converges to $\eta$ parametrized by $\mu_\fh$ with respect to the local uniform topology, modulo rotations.
\end{remark}

For $(M^\infty,e^\infty,T^\infty)$ in Conjecture~\ref{conj:scaling}, we have a random walk $\cZ^\infty$ by the Mullin bijection (Proposition~\ref{prop:Mullin-inf}). 
Let $Z^n(\cdot)=n^{-1/2}\cZ^\infty(n\cdot)$ so that $Z^n$ converge to $Z$, which is  a two-sided Brownian motion with independent coordinates.  
It is natural to conjecture that $(\mu^n,\lambda^n,\cZ^n)$ jointly converges to $(\fh,\eta,Z)$ where $(\fh,\eta)$ is encoded by $Z$ as in the discrete setting. 
A precise description of the limiting coupling of the triple for general $\gamma$ is the content of the continuum mating-of-trees theorem (Theorem~\ref{thm-mating}). 

Conjectures~\ref{conj:lqg},~\ref{conj:scaling}, and their variants constitute a  guiding principle for the study of random planar maps and remain central questions in this subject. 
There have been various degrees of success for different random planar maps models. Mating of trees plays a key role in these developments. See Sections~\ref{subsec:rmp}, \ref{sec-tutte-conv}, \ref{sec-kappa6}, and the references therein.

\subsection{The Gaussian free field} \label{sec-gff-prelim}

In this subsection we review the definition of the Gaussian free field. The reader who is already familiar with the GFF can safely skip this subsection. 
Most of the material in the subsection is taken from~\cite{shef-gff} and the introductory sections of~\cite{shef-zipper,ig4,wedges}, to which we refer for more details.
The reader may also want to consult~\cite{berestycki-lqg-notes,pw-gff-notes} for a more detailed exposition of the GFF. 

\subsubsection{Zero-boundary GFF}

Let $D \subset \BB C$ be a proper open domain with harmonically non-trivial boundary (i.e., Brownian motion started from a point in $D$ a.s.\ hits $\bdy D$).
We define $\mcl H_0(D)$ to be the Hilbert space completion of the set of smooth, compactly supported functions on $D$ with respect to the \emph{Dirichlet inner product},
\eqb \label{eqn-dirichlet}
(f ,g)_\nabla = \frac{1}{2\pi} \int_D \nabla f(z) \cdot \nabla g(z) \,d^2z.
\eqe
The \emph{(zero-boundary) Gaussian free field} on $D$ is defined by the formal sum
\eqb \label{eqn-gff-sum}
h = \sum_{j=1}^\infty X_j f_j 
\eqe
where the $X_j$'s are i.i.d.\ standard Gaussian random variables and the $f_j$'s are an orthonormal basis for $\mcl H_0(D)$. The sum~\eqref{eqn-gff-sum} does not converge pointwise, but it is easy to see that for each fixed $f \in \mcl H_0(D)$, the formal inner product $(h ,f)_\nabla := \sum_{j=1}^\infty X_j (f_j  ,f )_\nabla $ is a mean-zero Gaussian random variable and these random variables have covariances $\BB E [(h,f)_\nabla (h,g)_\nabla] = (f,g)_\nabla$. 

One can use integration by parts to define the ordinary $L^2$ inner products $(h,f) := -2\pi (h,\Delta_0^{-1}f)_\nabla$, where $\Delta_0^{-1}$ is the inverse Laplacian with zero boundary conditions, whenever $\Delta_0^{-1} f \in \mcl H_0(D)$. Then the random variables $(h,f)$ are jointly centered Gaussian with covariances
\eqb
\op{Cov}\left( (h,f) , (h,g) \right) = \frac{1}{2\pi}\int_D \int_D f(z) g(w) G_0^D(z,w) \,d^2 z\, d^2 w
\eqe
where $G_0^D(z,w)$ is the Green's function on $D$ with zero boundary conditions. With this definition, one can check that $h$ belongs to the Sobolev space $H^{-\ep}(D)$ for any $\ep>0$~\cite[Section 2.3]{shef-gff}.

It is easily seen from the conformal invariance of the Dirichlet inner product that the law of the GFF is conformally invariant, meaning that if $\phi : \wt D \rta D$ is a conformal map, then $h\circ \phi$ is the GFF on $\wt D$. 

\subsubsection{Whole-plane and free-boundary GFF}

We now also allow $D=\BB C$. 
Let $\mcl H(D)$ be the Hilbert space completion of the set of smooth (not necessarily compactly supported) functions $f$ on $D$ such that $(f,f)_\nabla  <\infty$ and $\int_D f(z) \,d^2 z = 0$ with respect to the inner product~\eqref{eqn-dirichlet}. Note that we need to require that $\int_D f(z)\,d^2 z = 0$ to make the inner product $(\cdot,\cdot)_\nabla$ positive definite.
The \emph{free-boundary} (if $D\not=\BB C$) or \emph{whole-plane} (if $D=\BB C$) GFF on $D$ is defined by the formal sum~\eqref{eqn-gff-sum} but with the $f_j$'s equal to orthonormal basis for $\mcl H(D)$ instead of for $\mcl H_0(D)$. 
As in the zero-boundary case, the formal inner products $(h ,f)_\nabla$ for $f\in \mcl H(D)$ are well-defined and are jointly centered Gaussian random variables with covariances $\BB E [(h,f)_\nabla (h,g)_\nabla] = (f,g)_\nabla$. 

Now let $\Delta^{-1}$ be the inverse of the Laplacian restricted to the space of functions (or generalized functions) with $\int_D f(z) \,d^2 z = 0$, normalized so that $\int_D \Delta^{-1} f(z) \,dz = 0$, with Neumann boundary conditions in the case when $D\not=\BB C$.  
Whenever $\Delta^{-1} f \in \mcl H(D)$, we can define the $L^2$ inner product $(h,f)  = -2\pi (h ,\Delta^{-1}f)_\nabla$. 
These $L^2$ inner produces are jointly centered Gaussian with variances 
\eqb
\op{Cov}\left( (h,f) , (h,g) \right) = \frac{1}{2\pi}\int_D \int_D f(z) g(w) G^D(z,w) \,d^2 z\, d^2 w
\eqe
where $G^D$ is the Green's function with Neumann boundary conditions in $D\not=\BB C$ and $G^D(z,w) = - 2\pi \log|z-w|$ if $D=\BB C$. 

We want to also define $(h,f)$ when $(\Delta^{-1} f , \Delta^{-1} f)_\nabla < \infty$ but $\int_D f(z) \,d^2 z$ is not necessarily equal to zero. 
To do this fix some $f_0$ with $(\Delta^{-1} f_0 , \Delta^{-1} f_0)_\nabla < \infty$ and $\int_D f_0(z) \,d^2 z  = 1$.
If we declare that $(h,f_0) := c$ for some $c\in\BB R$, then this together with the requirement that $f\mapsto (h,f)$ must be linear gives a unique way of defining $(h,f)$ for every function (or generalized function) $f$ with $(\Delta^{-1} f , \Delta^{-1} f)_\nabla < \infty$. Indeed, the function $\ol f := f - \left(\int_{D} f(z) \,d^2 z\right) f_0$ has total integral zero and so we can set $(h,f) := (h,\ol f) + c \int_{D} f(z) \,d^2 z$. As in the zero-boundary case, with this definition $h$ is an element of the Sobolev space $H^{-\ep}(D)$ for every $\ep > 0$~\cite[Section 3.3]{shef-kpz}. 

We made an arbitrary choice of $c$ in the above definition (all of the possible degrees of freedom can be captured by varying $c$, regardless of the choice of $f_0$). As such, the whole-plane and free-boundary GFF's are each only defined modulo a global additive constant. That is, we can view $h$ as a random equivalence class of distributions under the equivalence relation whereby $h_1 \sim h_2$ if $h_1-h_2$ is a constant for two (non-random) distributions $h_1,h_2$.  
In the case when $D = \BB C$ (resp. $D=\BB H$), we will typically fix the additive constant by requiring that the circle average $h_1(0)$ of $h$ over the unit circle $\bdy\BB D$ (resp.\ the unit semi-circle $\bdy\BB D\cap \BB H$), the definition of which we recall just below, is zero, i.e., we consider the field $h - h_1(0)$ which is well-defined not just modulo additive constant.  

The law of the whole-plane or free-boundary GFF is conformally invariant modulo additive constant, i.e., if $\phi : \wt D \rta D$ is a conformal map then $h\circ  \phi$ agrees in law with the (whole-plane of free-boundary, as appropriate) GFF on $\wt D$ modulo additive constant.

\subsubsection{Circle averages}
\label{sec-circle-average}

Suppose $D\subset\BB C$ and $h$ is a zero-boundary, whole-plane, or free-boundary GFF on $D$ (in the latter two cases we assume that we have fixed a choice of additive constant).  Then for $z\in D$ and $r > 0$ such that $\bdy B_r(z) \subset D$ we can define the circle average $h_r(z)$ over $\bdy B_r(z)$, following~\cite[Section 3.1]{shef-kpz}. 
Indeed, if we let $\rho_{z,r}$ be the uniform measure on $\bdy B_r(z)$ then the inverse Laplacian of $\rho_{z,r}$ has finite Dirichlet energy so as explained in the previous subsections we can define $h_r(z) := (h,\rho_{z,r})$. 

It is shown in~\cite[Section 3.1]{shef-kpz} that the circle average process a.s.\ admits a modification which is continuous in $z$ and $r$. We always assume that $h_r(z)$ has been replaced by such a modification. Furthermore,~\cite[Section 3.1]{shef-kpz} provides an explicit description of the law of the circle average process (it is a centered Gaussian process with an explicit covariance structure). For our purposes, the main fact which we will need about this process is the following. If $h$ is a whole-plane GFF and $z\in\BB C$ is fixed, then the process $t\mapsto h_{e^{-t}}(z) - h_1(z)$ is a standard two-sided linear Brownian motion.

If $h$ is a free-boundary GFF on a domain $D$ whose boundary has a linear segment $L$, then for $z\in L$ such that $\bdy B_r(z)$ does not intersect $\bdy D\setminus L$, we can similarly define the semicircle average $h_r(z)$ over $\bdy B_r(z)\cap D$. 
Similarly to the above, if $h$ is a free-boundary GFF on $\BB H$ and $z\in \BB R$ then $t\mapsto h_{e^{-t}}(z) - h_1(z)$ has the law of $\sqrt 2$ times a standard two-sided linear Brownian motion. See~\cite[Section 6]{shef-kpz} for details.

\subsubsection{Decomposition as a sum of independent fields}

If $\mcl H' \subset \mcl H_0(D)$ is a closed linear subspace, we can define the projection $\op{Proj}_{\mcl H'} h$ of $h$ onto $\mcl H'$ to be the distribution such that $(\op{Proj}_{\mcl H'} h , f)_\nabla = (h , \op{Proj}_{\mcl H'} f)_\nabla$ for each $f\in\mcl H_0(D)$. 
The covariance structure, and hence the law, of the zero-boundary GFF does not depend on the choice of orthonormal basis in~\eqref{eqn-gff-sum}. Consequently, if $\mcl H_0(D) = \mcl H^1 \oplus \mcl H^2$ is decomposed as the direct sum of two orthogonal subspaces, then by taking $\{f_j\}_{j\in\BB N}$ to be the the disjoint union of an orthonormal basis for $\mcl H^1$ and an orthonormal basis for $\mcl H^2$, we get that $h = \op{Proj}_{\mcl H^1} h + \op{Proj}_{\mcl H^2} h$ and the two summands are independent. 
Similar considerations apply for the whole-plane or free-boundary GFF, provided we interpret all of the distributions involved as being defined modulo additive constant. This decomposition has several useful consequences.

\begin{example}[Markov property]
	Let $h$ be the zero-boundary GFF on $D$ and let $V\subset D$ be open.
	The space $\mcl H_0(D)$ is the orthogonal direct sum of the space of functions in $\mcl H_0(D)$ which are supported on $V$ and the space of functions in $\mcl H_0(D)$ which are harmonic on $V$~\cite{shef-gff}. Consequently, $h$ can be decomposed as the sum of a zero-boundary GFF $h^V$ on $V$ and an independent distribution $\frk h^V$ on $D$ which is harmonic on $V$. 
	The distributions $h^V$ and $\frk h^V|_V$ are called the \emph{zero-boundary part} and \emph{harmonic part} of $h|_V$, respectively.
	In the case of the whole-plane or free-boundary GFF, one has a similar Markov property but $\frk h^V$ is only independent from $h^V$ viewed as a distribution modulo additive constant (one can get exact independence if the additive constant is fixed in a way which does not depend on what happens in $V$; see~\cite[Lemma 2.2]{gms-harmonic}). 
\end{example}

\begin{example}[Radial/lateral decomposition: whole-plane case] \label{example-radial-plane}
	Let $\mcl H^{\op{R}}(\BB C)$ (resp.\ $\mcl H^{\op{L}}(\BB C)$) be the space of functions in $\mcl H(\BB C)$ which are constant (resp.\ have mean zero) on each circle centered at zero. By~\cite[Lemma 4.9]{wedges}, $\mcl H(\BB C)$ is the orthogonal direct sum of $\mcl H^{\op{R}}(\BB C)$ and $\mcl H^{\op{L}}(\BB C)$. 
	Therefore, the projections of a whole-plane GFF $h$ onto $\mcl H^{\op{R}}(\BB C)$ and $\mcl H^{\op{L}}(\BB C)$ are independent.
	The projection onto $\mcl H^{\op{R}}(\BB C)$ is the function $h_{|\cdot|}(0)$ whose value on each circle centered at zero is the circle average of $h$ over that circle. 
	This function is defined modulo a global additive constant.
	The projection onto $\mcl H^{\op{L}}(\BB C)$ is $h - h_{|\cdot|}(0)$, which is well-defined not just modulo additive constant (since adding a constant $c$ to $h$ also adds $c$ to its circle average process). This projection is called the \emph{lateral part} of $h$. 
\end{example}

\begin{example}[Radial/lateral decomposition: half-plane case] \label{example-radial-half-plane}
	Let $\mcl H^{\op{R}}(\BB H)$ (resp.\ $\mcl H^{\op{L}}(\BB H)$) be the space of functions in $\mcl H(\BB H)$ which are constant (resp.\ have mean zero) on each semi-circle centered at zero. By~\cite[Lemma 4.2]{wedges}, $\mcl H(\BB H) = \mcl H^{\op{R}}(\BB H) \oplus \mcl H^{\op{L}}(\BB H)$. 
	As in Example~\ref{example-radial-plane}, this shows that the semi-circle average process $h_{|\cdot|}(0)$ is independent from the lateral part $h-h_{|\cdot|}(0)$. 
\end{example}

\subsection{Liouville quantum gravity}\label{subsec:GFF}

Throughout the rest of this section, we fix the LQG parameter $\gamma\in(0,2)$ and set 
\eqb \label{eqn-Q-def}
Q=\frac{\gamma}2 + \frac2\gamma .
\eqe
Let 
\eqb \label{eqn-domain-pairs}
\Dh=\{(D,h): \textrm{$D\subset \CC$  is an open set, $h$ is a distribution on $D$} \} . 
\eqe 
If $(D,h), (\wt D,\wt h) \in \Dh$ and $\phi:\wt D\to D$ is a conformal map, we write
\begin{equation}\label{eq:coord}
(D,h)\overset{\phi}{\sim}_{\gamma} (\wt D,\wt h)\textrm{ if and only if } \wt h = h \circ \phi + Q\log |\phi'|  .
\end{equation}
We write  $(D,h)\sim_{\gamma} (\wt D,\wt h)$ if and only if there exists a conformal map $\phi:\wt D\to D$ such that $(D,h)\overset{\phi}{\sim}_{\gamma} (\wt D,\wt h)$. Then $\sim_\gamma$ defines an equivalence relation on $\Dh$. If $(D,h)\sim_{\gamma} (\wt D,\wt h)$ then we think of $(D,h)$ and $(\wt D,\wt h)$ as two parametrizations of the same $\gamma$-LQG surface.

\begin{definition} \label{def-surface}  
	A \notion{$\gamma$-quantum surface} (a.k.a.\ a \notion{$\gamma$-LQG surface})
	is an equivalence class of pairs $(D,h)\in \Dh$ 
	under the equivalence relation $\sim_\gamma$. 
	An \notion{embedding} of a quantum surface is a choice of representative $(D,h)$ from the equivalence class.
	
	A quantum surface with $k$ marked points is an equivalence class of elements of the form
	$(D,h,x_1,\cdots x_k)$, with $(D,h)\in\Dh$ and $x_i\in \ol D$, under the equivalence relation $\sim_\gamma$ with the further requirement that marked points (and their ordering) are preserved by the conformal map.
	The transform in \eqref{eq:coord} is called a \notion{coordinate change}.
\end{definition}

In order to specify a quantum surface, it suffices to specify an embedding $(D,h)$. 
Therefore, 
the topology of distributions (i.e., the weakest topology under which integrals against smooth compactly supported test functions are continuous) induces a topology on the set of quantum surfaces  whose domains are of a particular conformal type.
For this topology, a sequence of quantum surfaces $\{\mcl S_n\}_{n\in\BB N}$ converges to a quantum surface $\mcl S$ if and only if there is a domain $D$ and embeddings $(D,h_n)$ of $\mcl S_n$ and $(D,h)$ of $\mcl S$ such that $h_n \rta h$ in the distributional sense.  
We equip the space of quantum surfaces with the Borel $\sigma$-algebra for this topology. 
Note that we will not talk about convergence of quantum surfaces with respect to this topology --- the topology is only used to define the $\sigma$-algebra (our convergence statements will always be with respect to a stronger topology\footnote{Note that convergence of quantum surfaces (i.e., the topology of convergence) is defined somewhat differently in different literature \cite{wedges,shef-kpz}. See e.g.\ \cite[Footnote 8]{hp-welding} for an overview. }). 

Roughly speaking, we will work with quantum surfaces where the distribution $h$ is random and locally looks like a GFF.
For concreteness, we provide a formal definition of this notion.
\begin{definition}\label{def:GFF}
	Let  $h$ be a  random distribution on a planar domain $D$. For $z\in D$, we say that $h$ is \notion{GFF-like} near $z$ if 
	there exist a constant  $r>0$ such that the law of $h|_{B_r(z)}$ is absolutely continuous with respect to the law of $(\wt h+g)|_{B_r(z)}$, where $(\wt h,g)$ is a coupling of a zero-boundary GFF $\wt h$ on $\ol{B_r(z)}$ and
	a random continuous function $g$ on $B_r(z)$. 
	If $z\in \bdy D$ and $\bdy D$ is analytic near $z$,  we similarly declare that  $h$ is  \notion{free-boundary GFF-like} near $z$ if $h$  is locally absolutely continuous with respect to a free-boundary GFF plus a continuous function in a similar manner.
\end{definition}

Recall from Section~\ref{sec-circle-average} that for a GFF $h$ on a domain $D$, a point $z\in D$, and a radius $\ep>0$ such that $B_\ep(z)\subset D$, we write $h_\ep(z)$ for the average of $h$ over $\bdy B_\ep(z)$. 
For $\gamma\in(0,2)$, it is proved in \cite{shef-kpz,shef-wang-lqg-coord} that the measure
$\mu_h:=\lim_{\ep\to 0}\ep^{\gamma^2/2} e^{\gamma h_\ep(z)}d^2z$  exists almost surely in the vague topology of Borel measures on $D$, where $d^2 z$ is the Lebesgue measure on $D$.
If $h$ is a random distribution which is GFF-like near $z$, then $\mu_h$ can be defined in $B_r(z)$  with $r$ as in Definition~\ref{def:GFF}.
We call $\mu_h$ the  \notion{$\gamma$-LQG area measure} with respect to $h$, or simply the \emph{quantum area}.

Suppose $h$ is a free-boundary GFF on $D$ and that $\bdy D$ has a linear segment $L$. For $z\in L$,  let $h_\ep(z)$ be the mean value of $h$ on $\bdy B_\ep(z)\cap D$. Then  
$\nu_h:=\lim_{\ep\to 0}\ep^{\gamma^2/4} e^{\gamma h_\ep(z)/2}dz$  exists almost surely, where $dx$ is the Lebesgue measure on $L$~\cite[Section 6]{shef-kpz}.
If $h$ is a random distribution which is GFF-like near $x\in L$, then   $\nu_h$ can be similarly defined near $x$.
We call $\nu_h$ the  \notion{$\gamma$-LQG length measure} with respect to $h$, or simply the \emph{quantum length}.
See~\cite{garban-survey} for expository notes on the results of~\cite{shef-kpz}. 

The measures $\mu_h$ and $\nu_h$ are a special case of a more general family of random measures associated with log-correlated Gaussian fields called \notion{Gaussian multiplicative chaos}, which was initiated in~\cite{hk-gmc,kahane}; see~\cite{rhodes-vargas-review,berestycki-gmt-elementary,aru-gmc-survey} for introductions to this theory. 
We also mention the paper of Shamov~\cite{shamov-gmc} which gives an alternative, more axiomatic approach to GMC and~\cite{rhodes-vargas-review} which proves closely related results to the ones of~\cite{shef-kpz} but in a way which is more directly connected to GMC theory.

Suppose $(D,h)$, $(\wt D,\wt h)$, and $\phi$ are  as in    \eqref{eq:coord}.
If $h$ is a GFF-like near every point of $D$,  then so is  $\wt h$.
By \cite[Proposition 2.1]{shef-kpz}, almost surely  $\mu_{\wt h}$ is the pushforward of $\mu_h$  under $\phi^{-1}$. 
Namely, 
$\mu_{\wt h}(A)=\mu_h(\phi (A))$ for each Borel set $A\subset \wt  D$.
This is the so-called coordinate change formula for the $\gamma$-LQG area measure. 
If $h_\ep(z)$ is defined via  bump function averages instead of circle averages, it is shown in \cite{shef-wang-lqg-coord} that almost surely 
the coordinate change formula holds for all conformal maps simultaneously. 
Similarly, $\nu_{\wt h}$ is a.s.\ the pushforward of $\nu_h$ under $\phi^{-1}$ if
$\bdy D$ and $\bdy \wt D$ are both Jordan domains with piecewise linear  boundaries and $\phi$ extends to a homeomorphism between the closure of the domains.

When $\wt D$ is a simply connected domain 
whose boundary is the image of continuous closed curve (viewed on the Riemann sphere),
$\phi^{-1}$ can still be continuously extended to  $\bdy D$ so that $\phi^{-1}|_{\bdy D}$ provides a parameterization of $\bdy \wt D$.
Then we define  $\nu_{\wt h}$ to be the pushforward of $\nu_h$ under $\phi^{-1}$.
The coordinate change formula ensures that $\nu_{\wt h}$ defined this way does not depend on the choice of $D$.
We can define the quantum length measure  for more general domains,  but the generality presented here is sufficient for our article.

In mating-of-trees theory, the measures $\mu_h$ and $\nu_h$ (rather than the GFF $h$ itself) are the main observables.
We note that $h$ is locally determined by $\mu_h$~\cite{bss-lqg-gff}. 

We will sometimes have occasion to consider an LQG surface decorated by a curve (this is the continuum analog of a random planar map decorated by a statistical mechanics model). 
As such, we make the following definition.

\begin{defn} \label{def-curve-decorated}
	A \notion{curve-decorated quantum surface} is an equivalence class of triples $(D,h,\eta)$ where $D\subset\BB C$ is open, $h$ is a distribution on $D$, and $\eta$ is a curve in $\ol D$, under the equivalence relation whereby two such triples $(D,h,\eta)$ and $(\wt D, \wt h , \wt\eta)$ are equivalent if there is a conformal map $\phi : \wt D \rta \wt h$ such that $ (D,h)\overset{\phi}{\sim}_{\gamma} (\wt D,\wt h)$ in the sense of~\eqref{eq:coord} and $\phi\circ\wt\eta = \eta$. 
	As in Definition~\ref{def-surface}, we similarly define a \notion{curve-decorated quantum surface with $k \in \BB N$ marked points}.
\end{defn}

We can define a topology, hence also a $\sigma$-algebra, on the space of curve-decorated quantum surfaces analogously to the discussion just after Definition~\ref{def-surface} except that we also require that the associated curves converge uniformly under the chosen embeddings.

\subsection{Quantum cones and wedges}\label{subsec:cone-wedge}

The local properties of the LQG measures $\mu_h$ and $\nu_h$ do not depend on which GFF-type distribution we are considering. 
However, there are exact symmetries for certain special LQG surfaces which allow us to connect such surfaces to SLE curves and to random planar maps.  
In this subsection and the next we introduce some of these special LQG surfaces. 

\subsubsection{Quantum cones and thick quantum wedges}

In the definitions below, we will consider drifted Brownian motion conditioned on staying positive or negative.
Suppose $B_t$ is a standard linear Brownian motion starting at 0 and $a>0$.  
Then $(B_t+at)_{t\ge 0}$ conditioned to stay positive can be defined by the weak limit as $\ep\to0$ of the conditional law of $(B_t+at)_{t\ge 0}$  
conditioned to stay above $-\ep$. It is easy to show that the limit exists and can also be described as $(B_{t+\tau}+a(t+\tau))_{t\ge 0}$ where $\tau$ is the last zero of $B_t+at$.  

We will now define quantum cones and quantum wedges, which are the most natural LQG surfaces with the whole-plane and half-plane topology, respectively (e.g., in the sense that they arise as the scaling limits of random planar maps with these topologies; see Section~\ref{subsec:scaling}). 
We give the definitions in the context of the radial/lateral decomposition of Examples~\ref{example-radial-plane} and~\ref{example-radial-half-plane}. 
The following two definitions are~\cite[Definitions 4.10 and 4.5]{wedges}. 
The motivation for the definitions is to make Lemmas~\ref{lem:circle}, \ref{lem:inv}, and~\ref{lem:gamma} below true.

\newcommand{\bh}{\mathsf{h}}

\begin{defn}[Quantum cone] \label{def:cone}
	Let $\alpha \in (-\infty, Q)$. 
	Let $A : \BB R \to \BB R$ be the process $A_t:= B_t + \alpha t$, where $B_t$ is a standard linear Brownian motion conditioned so that $ B_t  - (Q-\alpha) t > 0$ for all $t< 0$. In particular, $B|_{[0,\infty)}$ is an unconditioned standard linear Brownian motion. 
	The \notion{$\alpha$-quantum cone} is the $\gamma$-LQG surface $(\BB C ,\bh , 0,\infty)$ where $\bh$ is the random distribution on $\BB C$ defined as follows.
	If $\bh_r(0)$ denotes the circle average of $\bh$ on $\partial B_r(0)$ (as in Section~\ref{sec-circle-average}), then $t\mapsto \bh_{e^{-t}}(0)$ has the same law as the process $A$; and the lateral part $\bh - \bh_{|\cdot|}(0)$ (Example~\ref{example-radial-plane}) is independent from $\bh_{|\cdot|}(0)$ and has the same law as the analogous process for a whole-plane GFF. That is, $\bh - \bh_{|\cdot|}(0) = \sum_{j=1}^\infty X_j f_j$ where $\{f_j\}_{j\in\BB N}$ is an orthonormal basis for the space $\mcl H^{\op{L}}(\BB C)$ of Example~\ref{example-radial-plane} and $\{X_j\}_{j\in\BB N}$ are i.i.d.\ standard Gaussians.
\end{defn}

\begin{defn}[Quantum wedge] \label{def:wedge}
	Let $\alpha \in (-\infty, Q)$. 
	Let $A : \BB R \to \BB R$ be the process as in Definition~\ref{def:cone} except with $B_t$ replaced by $B_{2t}$ throughout. 
	The \notion{$\alpha$-quantum wedge} is the  LQG surface $(\BB H ,\bh , 0,\infty)$ where $\bh$ is the random distribution on $\BB H$ defined as follows. 
	If $\bh_r(0)$ denotes the semi-circle average of $\bh$ on $\partial B_r(0) \cap \BB H$, then $t\mapsto \bh_{e^{-t}}(0)$ has the same law as the process $A$; and the lateral part $\bh - \bh_{|\cdot|}(0)$ (Example~\ref{example-radial-half-plane}) is independent from $\bh_{|\cdot|}(0)$ and has the same law as the analogous process for a free-boundary GFF on $\BB H$. 
\end{defn}

Note that we have $B_{2t}$ instead of $B_t$ in Definition~\ref{def:wedge} since the variance of the semicircle average process for a free-boundary GFF is twice as large as the variance for the circle average process of a whole-plane GFF (Section~\ref{sec-circle-average}). 

Each of Definition~\ref{def:cone} and~\ref{def:wedge} specifies only one possible embedding of the quantum cone (resp.\ wedge).
The particular embedding used in these definitions is called the \emph{circle average embedding} and is uniquely characterized (modulo rotation about the origin in the whole-plane case) by the requirement that $1 = \sup\{r > 0 : \bh_r(0) + Q\log r = 0\}$. The circle average is especially convenient to work with due to the following lemma, which is immediate from the definitions.  

\begin{lem}\label{lem:circle}
	Let $\bh$ be the circle average embedding of an $\alpha$-quantum cone and let $ h$ be a whole-plane GFF with the additive constant chosen so that $h_1(0) = 0$.
	Then $\bh|_{\D}  \eqD \left( h -\alpha \log |\cdot| \right)_{z\in \D}$. 
	The same holds if $\bh$ is the circle average embedding of an $\alpha$-quantum wedge,  $ h$ is a free-boundary GFF on $\BB H$ with the additive constant chosen so that $h_1(0) = 0$, and $\BB D$ is replaced by $\BB D\cap\BB H$. 
\end{lem}

Quantum cones and quantum wedges possess a certain scale invariance property which is not true for the whole-plane or free-boundary GFF.
This makes these surfaces in some sense more canonical objects than the whole-plane or free-boundary GFF. 
The following lemma is a consequence of~\cite[Propositions 4.7 and 4.13]{wedges}.

\begin{lem}\label{lem:inv}
	Suppose  $(D,\bh,a,b)$ is an embedding of an  $\alpha$-quantum cone (resp.\ wedge) for $\alpha \in (-\infty ,Q)$. 
	For each $c\in\R$,  $(D,\bh+c,a,b)$ and $(D, \bh,a,b)$ have the same law as quantum surfaces. 
Equivalently, if $\bh$ is the circle average embedding, then
	\eqbn
	 \bh(R_c\cdot) + Q\log R_c + c\eqD \bh \quad \text{where} \quad R_c := \sup\{r  > 0 : \bh_r(0) + Q\log r = -c\}.
	\eqen 
\end{lem}

Since adding $c$ to $\bh$ scales the LQG area measure by $e^{\gamma c}$, Lemma~\ref{lem:inv} says that the law of a quantum cone or wedge is invariant under scaling its LQG area measure by a constant factor. This is one piece of evidence for why these surfaces are the ones which arise as the scaling limits of random planar maps. It turns out that Lemmas~\ref{lem:circle} and~\ref{lem:inv} uniquely characterize the law of the $\alpha$-quantum cone or wedge; see~\cite[Propositions 4.8 and 4.14]{wedges}. 

The case when $\alpha = \gamma$ is special since a $\gamma$-quantum cone is the local limit of a quantum surface around an interior quantum typical point. See~\cite[Proposition 4.13]{wedges}. 

\begin{lem}\label{lem:gamma}
	Let $h$ be a zero-boundary GFF on a bounded domain $D$. Conditional on $h$, let $u$ be a point sampled according to $\mu_h(\cdot)/\mu_h(D)$.
	For $\ep>0$, let $r_\ep$ be such the $\mu_h(B_{r_\ep}(u))=\ep$ and let $\phi_\ep(z)=r_\ep z+u$. Let $(D^\ep, h^\ep)$ be such that 
	$(D,h  + \gamma^{-1} \log \ep^{-1}  )\overset{\phi_\ep}{\sim}_{\gamma}  (D^\ep, h^\ep )$ as in \eqref{eq:coord}.  
	Then $h^\ep$ converges in law as $\ep\rta 0$ in the distributional sense to the random distribution $\fh$, where $(\CC,\fh,0,\infty)$ is the embedding of a $\gamma$-quantum cone for which $\mu_{\fh}(\D) = 1$. 
\end{lem}

Heuristically, Lemma~\ref{lem:gamma} holds since near $u$, the field $h$ looks like $h-\gamma\log|\cdot - u|$ where $h$ is a whole-plane GFF, as in Lemma~\ref{lem:circle}.
This intuition can be made rigorous using the so-called rooted measure corresponding to $\mu_h$~\cite{peyriere-rooted-measure,shef-kpz,kahane,rhodes-vargas-review}.
There is also a straightforward analog of Lemma~\ref{lem:gamma} for the $\gamma$-quantum wedge, where we consider a free-boundary GFF $h$ on a Jordan domain $D$ with a linear segment and let $u$ be sampled from the segment according to $\nu_h$ restricted to this linear segment and normalized to be a probability measure~\cite[Proposition 4.7]{wedges}.  
There are also variants of Lemma~\ref{lem:gamma} when we sample from the $\alpha$-LQG measure for $\alpha\in (-2,2)$ instead of the $\gamma$-LQG measure, and we get an $\alpha$-quantum cone instead of a $\gamma$-quantum cone in the limit.

\begin{remark}[Parametrizing by the strip/cylinder] \label{remark-strip}
	It is sometimes convenient to parametrize the $\alpha$-quantum wedge by the infinite strip $\cS = \R \times (0,\pi)$ instead of by $\BB H$. 
	Let $\phi : \cS \rta \BB H$ be defined by $\phi(z) = e^{-z}$ and let $\wt{\bh}$ be the random distribution on $\cS$ such that $(\cS , \wt{\bh})\overset{\phi}{\sim}_\gamma (\BB H , \bh)$. 
	Then $\wt{\bh}$ can be described as follows.
	If we let $X_t$ be the average of $\wt{\bh}$ along the line segment $\{t\}\times [0, \pi]$, then $\{X_t\}_{t\geq 0} \eqD \{B_{2t} -  (Q-\alpha) t\}_{t\geq 0}$, where $B$ is a standard linear Brownian motion, and $\{X_{-t}\}_{t\geq 0}$ is independent from $\{X_t\}_{t\geq 0}$ and has the law of $\{B_{2t} + (Q-\alpha) t\}_{t\geq 0}$ conditioned to stay positive. 
	The lateral part $\wt{\bh} - X_{\re \cdot}$ is independent from $X$ and has the law of $h^{\op{L}}\circ\phi$, where $h^{\op{L}}$ is the lateral part of a free-boundary GFF on $\BB H$ (Example~\ref{example-radial-half-plane}).
	We have a similar description if we parametrize an $\alpha$-quantum cone by the cylinder $\R \times [0,2\pi]$ with $\BB R\times \{0\}$ identified with $\BB R\times\{2\pi\}$, except with $B_{ t}$ instead of $B_{2t}$. 
	One advantage with parametrization using the strip, is that if we do a change of coordinates corresponding to a translation $\psi:\cS\to\cS$ with $\psi(z)=z+b$ and $b \in \BB R$, then the term $Q\log|\psi'|$ is identically equal to zero, which sometimes simplifies calculations. 
\end{remark}

It is also possible to define quantum cones and quantum wedges for $\alpha =Q$. We will only need the quantum wedge case, so we will only give the precise definition in this setting, but we note that the quantum cone definition is similar. 
We first note that as $a \to 0$,  the law of $(B_t+at)_{t\ge 0}$ conditioned to stay positive has a weak limit, which is the 3-dimensional Bessel process. 
By convention, we call this process \emph{Brownian motion conditioned to stay positive} and the negative of this process \emph{Brownian motion conditioned to stay negative}.

\begin{defn}[Quantum wedge for $\alpha=Q$] \label{def:Qwedge} 
	Let $A : \BB R \to \BB R$ be the process $A_t:= B_{2t} + Q t$, where $B_t$ is a standard two-sided linear Brownian motion conditioned so that $ B_t   < 0$ for all $t >  0$.   
	The \notion{$Q$-quantum wedge} is the  LQG surface $(\BB H ,\bh , 0,\infty)$ where $\bh$ is the random distribution on $\BB H$ defined as follows. 
	If $\bh_r(0)$ denotes the semi-circle average of $\bh$ on $\partial B_r(0) \cap \BB H$, then $t\mapsto \bh_{e^{-t}}(0)$ has the same law as the process $A$; and the lateral part $\bh - \bh_{|\cdot|}(0)$ (Example~\ref{example-radial-half-plane}) is independent from $\bh_{|\cdot|}(0)$ and has the same law as the analogous process for a free-boundary GFF on $\BB H$. 
\end{defn}

The conditioning in Definition~\ref{def:Qwedge} is for $t > 0$, rather than for $t <0$ as in Definition~\ref{def:wedge}.
The conditioning in Definition~\ref{def:Qwedge} ensures that $\mu_{\bh}(B_r(0))$ is finite for all $r>0$ and therefore this embedding is often nicer to work with in the case $\alpha=Q$. 
As a consequence, Lemma~\ref{lem:circle} is not true for $\alpha =Q$.
However, one has the obvious analogs of Lemma~\ref{lem:inv} and Remark~\ref{remark-strip} for $\alpha=Q$.

\subsubsection{Thin quantum wedges}

Definition~\ref{def:wedge} has a nontrivial generalization to the case $\alpha \in (Q,Q+\gamma/2)$, whose  motivation will be clear in Section~\ref{subsec:zipper}. 
Let $Z$ be a Bessel process of dimension $\delta\ge 2$ starting from 0.   
By It\^o's calculus, $-2\gamma^{-1}\log Z$   has a unique reparametrization $(X_t)_{t\in \R}$   such that  $0=\inf\{t\in\R\,:\, X_t =0\}$  and  
the quadratic variation of $X$ during $[t',t]$ equals  $2(t-t')$ for each $t'<t$. 
Moreover, the law of $X_t$ is as described in Remark~\ref{remark-strip} where $\alpha$ is such that 
\begin{equation}\label{eq:Bessel}
2(Q-\alpha)= \gamma(\delta-2).
\end{equation}
This gives a way of constructing an $\alpha$-quantum wedge using a $\delta$-dimensional Bessel process for all $\alpha\le Q$.

\begin{figure}[ht!]
	\begin{center}
		\includegraphics[scale=.8]{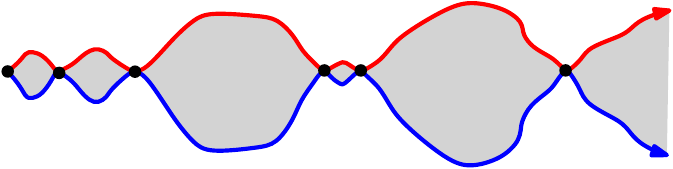} 
		\caption{Sketch of a thin quantum wedge. The surface consists of an infinite string of ``beads", each of which has the topology of the disk, finite total LQG area, and two marked points, strung together end-to-end. The figure is not entirely accurate since in actuality there are infinitely many small beads between any two given beads (similarly to the situation for excursions of a Brownian motion or a Bessel process). However, the total area and the total boundary length of the beads which come before any fixed bead is finite. Note that the embedding in the figure is \emph{not} the same as the one discussed after Definition~\ref{def:thin-wedge} (under the latter embedding, the entire right boundary of the quantum wedge is mapped to a straight line). 
		}\label{fig-thin-wedge}  
	\end{center}
\end{figure}

Now we extend this construction to $\alpha\in (Q,Q+\gamma/2)$. We still let  $\delta$ satisfy \eqref{eq:Bessel} and let $Z$ be the Bessel process starting from 0 as before. 
In this case, $\delta\in(1,2)$, hence $Z$ hits zero infinitely many times.  The It\^o excursion decomposition of $Z$ away from $0$
provides a Poisson point process (p.p.p)  on $\cE\times (0,\infty)$, where $\cE$ is the space of \emph{excursions}, namely, non-negative continuous functions $[0,\infty) \rta [0,\infty)$ which are zero at time 0 and for all sufficiently large times.
For the points $(e,u)$ appearing in this  p.p.p., $2\gamma^{-1}\log e$   has a unique reparametrization $(X^e_t)_{t\in \R}$   such that  $X^e_0=\max\{X^e_t:t\in \R\}$  and  
the quadratic variation of $X^e$ during $[t',t]$ equals  $2(t-t')$ for each $t'<t$.

\begin{definition}[Thin wedge] \label{def:thin-wedge}
	Fix $\alpha\in (Q,Q+\gamma/2)$ and recall the infinite strip $\cS = \R \times (0,\pi)$ as in Remark~\ref{remark-strip}. Consider the Bessel process $Z$ of dimension $\delta$ and its corresponding p.p.p.\ as above.
	For each $(e,u)$ in  the p.p.p., let $(\mcl S , h^e , +\infty, - \infty)$ be the quantum surface where $h^e$ is defined as follows. 
	The average of $h^e$ on each vertical segment $\{t\} \times [0,\pi]$ is equal to $X_t^e$. 
	The lateral part $h^e - X_{\re \cdot}^e$ is independent from $X^e$ and $\{ h^{e'}\,:\, e'\neq e\}$ and has the law of $h^{\op{L}}\circ\phi$, where $h^{\op{L}}$ is the lateral part of a free-boundary GFF on $\BB H$ (Example~\ref{example-radial-half-plane}). 
	The countable ordered collection of quantum surfaces $(\cS,h^e, +\infty,-\infty)$ is called an \notion{$\alpha$-quantum wedge}.
\end{definition}
 
The quantum surfaces $(\cS,h^e, +\infty,-\infty)$ in Definition~\ref{def:thin-wedge} are called the \notion{beads} of the quantum wedge.
In the setting of Definition~\ref{def:thin-wedge}, consider the region $S=\{ (s,y)\in [0,\infty)^2  : 0 <y <  Z_s  \}$ under the graph of $Z$.
Then $S$ has countably many components, each of which is homeomorphic to the disk and intersects $ \{0 \}\times [0,\infty)$ at precisely two marked points.  
These components are in bijection with  excursions $e$ of the p.p.p.\ corresponding to $Z$.  
Definition~\ref{def:thin-wedge} can be thought of as a procedure for associating a quantum surface structure to each component, which in turn gives a quantum surface structure to $\mcl S$. 
Hence for $\alpha \in (Q,Q+\gamma/2)$, an $\alpha$-quantum wedge no longer has the topology of the half-plane (unlike for $\alpha \leq Q$).
We call quantum wedges for $\alpha \leq Q$ \notion{thick} and quantum wedges for $\alpha \in (Q , Q+\gamma/2)$ \notion{thin}.  

See Figure~\ref{fig-thin-wedge} for an illustration of a thin quantum wedge. One motivation for the definition is that thin wedges arise when we cut a quantum wedge or a quantum cone by an SLE$_{\gamma^2}$-type curve which intersects itself or the domain boundary. This is explained in Section~\ref{subsec:zipper}.

\subsection{Finite-volume quantum surfaces}
\label{subsec:finite}

There are several special quantum surfaces which (unlike quantum cones and wedges) have finite total LQG area. In this subsection we will review the definition of one of these quantum surfaces, namely the \emph{quantum disk}. A similar definition gives the \emph{quantum sphere}. 
The definitions of these surfaces are somewhat less intuitive than the definitions of quantum cones and thick quantum wedges, and the details of the definitions are not used in subsequent sections.
As such, the reader may wish to skip this section on a first read. 
As explained in Section~\ref{subsec:rv}, alternative (equivalent) definitions of these surfaces based on the path integral perspective are given in~\cite{dkrv-lqg-sphere,hrv-disk}.

We first define an infinite measure on quantum surfaces, then condition this measure on certain events to obtain probability measures. 
Our exposition is similar to that of~\cite[Appendix A]{gwynne-miller-sle6}. 
As in Remark~\ref{remark-strip}, is convenient to define a quantum disk parameterized by the infinite strip $\mcl S = \BB R\times (0,\pi)$ since the field takes a simpler form in this case (one can parameterize by the unit disk instead by applying a conformal map and using~\eqref{eq:coord}). 
The following definition is given in~\cite[Section~4.5]{wedges}. 

\begin{defn} \label{def-infinite-disk}
	For $\gamma \in (0,2)$, the \emph{infinite measure on quantum disks} is the measure $ \mcl M^{\op{disk}}$ on doubly marked quantum surfaces $(\mcl S ,h^{\op{D}} , -\infty,\infty)$ defined as follows. ``Sample" $e$ from the infinite excursion measure of a Bessel process of dimension $3-\frac{4}{\gamma^2}$ (see~\cite[Remark 3.7]{wedges}). Let $X^e : \BB R\rta \BB R$ be equal to $2\gamma^{-1} \log e$ reparameterized to have quadratic variation $2\,dt$ (note that this process is only defined modulo translations -- different translations give equivalent quantum surface). 
	The average of $h^{\op{D}}$ on each vertical segment $\{t\} \times (0,\pi)$ is equal to $X_t^e$. 
	The lateral part $h^{\op{D}} - X_{\re \cdot}^e$ is independent from $X^e$ and  has the law of $h^{\op{L}}\circ\phi$, where $h^{\op{L}}$ is the lateral part of a free-boundary GFF on $\BB H$ (Example~\ref{example-radial-half-plane}) and $\phi:\cS\to\BB H$ is defined by $\phi(z)=e^{-z}$. 
\end{defn}

\begin{remark} \label{remark-beads}
	By~\eqref{eq:Bessel}, when $\gamma \in (\sqrt 2 ,2)$ and $\alpha = 2Q-\gamma$, the measure $\mcl M^{\op{disk}}$ is exactly the same as the intensity measure of the Poisson point process used to define the the (thin) $\alpha$-quantum wedge. 
	Hence in this case the beads of an $\alpha$-quantum wedge are quantum disks.
\end{remark} 

For each $t> 0$, the measure $\mcl M^{\op{disk}}$ assigns finite mass to the set of surfaces whose corresponding excursion $e$ has time length at least $t$ (as the Bessel excursion measure assigns finite mass to those Bessel excursions with length at least $t$). From this, one can deduce that for each $  \ell > 0$, $\mcl M^{\op{disk}}$ assigns finite mass to the set of surfaces with LQG boundary length at least $\ell$. 
Hence the following definition makes sense (at least for Lebesgue-a.e.\ $\ell > 0$).

\begin{defn} \label{def-disk} 
	Let $\mcl M^{\op{disk}}$ be as in Definition~\ref{def-infinite-disk}. For $\ell > 0$, the \emph{quantum disk with boundary length $\ell$} is the regular conditional distribution of the measure $  \mcl M^{\op{disk}} $ given that $\nu_{h^{\op{D}}}(\bdy\mcl S) = \ell$. 
\end{defn}

The above definitions give us doubly marked quantum disks. One can define singly marked or unmarked quantum disks by forgetting one or both of the marked points. 
It is shown in~\cite[Proposition~A.8]{wedges} that the marked points for a doubly marked quantum disk are independent samples from the $\gamma$-LQG boundary length measure if we condition on the disk viewed as an unmarked quantum surface. Equivalently, suppose that $(\mcl S, h^{\op{D}} , -\infty,+\infty)$ is a quantum disk, that $x,y \in \partial \mcl S$ are picked independently from the $\gamma$-LQG boundary measure $\nu_{h^{\op{D}}}$, and that $\phi \colon \mcl S \to \mcl S$ is a conformal transformation with $\phi(-\infty) = x$ and $\phi(+\infty) = y$.  Then the fields $ h^{\op{D}} \circ \phi + Q \log|\phi'|$ and $h^{\op{D}}$ have the same law modulo a horizontal translation of~$\mcl S$.

A priori the regular conditional laws of the infinite quantum disk measure given the boundary length only make sense for a.e.\ $\ell > 0$. However, as explained just after~\cite[Definition~4.21]{wedges}, the quantum disk possesses a scale invariance property which allows us to define these regular conditional laws for every choice of $\ell > 0$. In particular, if $(\mcl S  , h^{\op{D}} , -\infty,\infty)$ is a quantum disk with boundary length $\ell$ and $C >0$, then $(\mcl S , h^{\op{D}} + C , -\infty,\infty)$ is a quantum disk with boundary length $e^{\gamma C} \ell$.

The \emph{quantum sphere} is a finite-volume LQG surface with the topology of the Riemann sphere which can be defined in a similar manner to the quantum disk above, except that it is natural to condition on the LQG area rather than the LQG boundary length; see~\cite[Section 4.5]{wedges}.

There are analogs of Lemma~\ref{lem:gamma} for both the quantum disk and the quantum sphere. See~\cite[Appendix A]{wedges}.
For example, the quantum sphere can be obtained as the weak limit of a zero-boundary GFF on a bounded domain conditioning on a rare event~\cite[Proposition A.11]{wedges}. 	
This limit is also considered  in \cite{rv-tail,wong-tail}. There  the precise asymptotic of the probability of rare event is obtained, which  can be viewed as a convergence of quantum surfaces at the partition function level.

\subsubsection{Path integral approach}
\label{subsec:rv}

There is an entirely different approach to constructing LQG surfaces from the one considered in this article which was initiated by David, Kupiainen, Rhodes, and Vargas~\cite{dkrv-lqg-sphere}.
This approach is based on a rigorous version of Polyakov's path integral~\cite{polyakov-qg1} wherein one directly makes sense of the partition function of the random surface coming from Polyakov's action.
Reviewing the details of this path integral approach is beyond the scope of this survey. 
We simply point out that for any given topological surface $S$, the path integral approach in principle gives a systematic way of constructing the canonical LQG surface with $S$-topology.
This has been carried out when $S$ is the sphere~\cite{dkrv-lqg-sphere},  disk~\cite{hrv-disk}, torus~\cite{drv-torus}, annulus~\cite{remy-annulus}, and
closed Riemann surface of genus $g \geq 2$~\cite{grv-higher-genus}. 
The advantage of this approach is that it is directly related to the conformal bootstrap of Liouville conformal field theory, which is in principle integrable. 
Important developments in this direction an be found in~\cite{krv-dozz,remy-fb-formula,rz-gmc-interval}.
See~\cite{vargas-dozz-notes} for lecture notes on the path integral approach to LQG.

When the surface $S$ is the sphere, it is proved in \cite{ahs-sphere} that the aforementioned Bessel process definition of the quantum sphere from~\cite{wedges} is equivalent the one based on path integrals in~\cite{dkrv-lqg-sphere}. The analogous statement for the quantum disk is proven in~\cite{cercle-quantum-disk}. When $S$ is a not simply connected, so far the the canonical LQG surface has not been constructed along the lines of the ideas surveyed in this paper.  It is an open question to find mating-of-trees type results for LQG surfaces with non-simply-connected topology; see Problem~\ref{prob-higher-genus}.

\subsection{Schramm-Loewner evolution}\label{subsec:sle}

In this subsection we will review the definition of ordinary SLE$_\kappa$, SLE$_\kappa(\rho)$, and space-filling SLE$_\kappa$. 
All of the results and definitions in this subsection can be taken as black boxes elsewhere in the paper. 
In particular, some of the results discussed in this section are proven using the theory of imaginary geometry~\cite{ig1,ig2,ig3,ig4}, but the reader does not need to understand this theory to understand the rest of this survey. Neither imaginary geometry nor Loewner evolution theory is mentioned outside of this subsection.  

\subsubsection{Ordinary SLE$_\kappa$}
The Schramm-Loewner evolution (SLE$_\kappa$) is a one-parameter family of random fractal curves introduced  by Schramm~\cite{schramm0}.
We give only a brief introduction and refer to~\cite{werner-notes,lawler-book,bn-sle-notes} for more detail. 
Fix $\kappa> 0$.
Let $\{B_t\}_{t\geq 0}$ be a standard linear Brownian motion and let $\{g_t\}_{t\geq 0}$ be the unique family of conformal maps which satisfies the \emph{Loewner equation} with \emph{driving function} $W_t = \sqrt\kappa B_t$, namely 
\begin{equation}\label{eq:Loewner}
\p_tg_t(z)=\frac{2}{g_t(z)-  W_t} ,\qquad g_0(z)=z . 
\end{equation} 
Then each $g_t$ maps a subdomain $H_t$ of $\BB H$ to $\BB H$.
The maps $\{g_t\}_{t\geq 0}$ are called the \emph{Loewner maps}.\footnote{In some other literature, the term \emph{Loewner map} is used for the functional which takes in the driving function and outputs the growing family of hulls. We will not make explicit reference to this functional here, so there is no danger of ambiguity.}

It is shown in~\cite{schramm-sle} that there is a curve $\eta$ from 0 to $\infty$ in $\ol{\BB H}$ such that the \emph{hull} $\BB H\setminus H_t$ is the set of points in $\BB H$ which are disconnected from $\infty$ by $\eta([0,t])$.
This curve $\eta$ is defined to be the chordal SLE$_\kappa$ on $(\BB H , 0,\infty)$ with the capacity parametrization (this means that $\lim_{z\rta\infty} z(g_t(z) -z) = 2t$).  

Chordal $\SLE_\kappa $ on $(\BB H, 0,\infty)$ under the capacity parameterization is uniquely characterized by the following properties.
\begin{itemize}
	\item \textbf{Scale invariance.} $(\eta(at))_{t \ge 0}\eqD(a^{1/2}\eta(t))_{t\ge 0}$. 
	\item \textbf{Domain Markov property.} For each stopping time $\tau$ for $\sqrt\kappa B_t$, the conditional law given $(\eta(t))_{t\in[0,\tau]}$ of the image of $(\eta(t+\tau))_{t\ge 0}$ under $g_\tau(\cdot) -W_\tau$ has the same law  as $\eta$..
\end{itemize}
Given a simply connected domain $D$ whose boundary is parameterizable as a closed curve, and two points $a\neq b$ on $\bdy D$, 
let $\phi:\BB H\to D$ be a conformal map such that $\phi(0)=a$ and $\phi(\infty)=b$. 
By the scale invariance property of SLE$_\kappa$, the law of $\phi(\eta)$ as a curve in $\ol D$ \emph{modulo monotone time reparametrization} does not depend on the choice of $\phi$. We call a curve with this law \notion{chordal SLE$_\kappa$ on $(D,a,b)$}.  

If $a \in \bdy D$ and $b\in  D$ instead,
we may similarly define  \notion{radial SLE$_\kappa$ on $(D,a,b)$} using radial Loewner chains.  See \cite[Sections 4.2 and 6.5]{lawler-book}.
We can also define \emph{whole-plane SLE$_\kappa$} between any two specified distinct points $a,b\in\BB C \cup\{\infty\}$~\cite[Sections 4.3 and 6.6]{lawler-book}.
We skip the construction as well as other details on the Loewner chain definition of SLE as they are not essential to understand this article.  
In fact,  mating of trees provides a framework to study SLE without any direct use of Loewner chains (see Section~\ref{subsec:app-sle} for some examples). 

It is shown in \cite{schramm-sle} that, chordal, radial, and whole-plane $\SLE_\kappa$ as $\kappa$ varies has three topological phases.
When $\ul\kappa \in(0,4]$, SLE$_\kappa$ curves are simple and do not touch the domain boundary except at their endpoints. When $\kappa\in(4,8)$, SLE$_\kappa$ curves are non-simple and touch the domain boundary but are not space-filling.\footnote{Here and throughout the rest of the paper, we will typically write $\ul\kappa$ for the SLE parameter when it is constrained to be in $(0,4]$ and $\kappa$ for the SLE parameter when it is constrained to be larger than 4 or when it is unconstrained. The reason for this convention is that we will more often consider $\kappa > 4$. Note that this differs from the convention in~\cite{ig1,ig2,ig3,ig4}.} 
When $\kappa\ge 8$, SLE$_\kappa$ curves are space-filling. 

\subsubsection{SLE$_\kappa(\rho)$} 
\label{sec-sle-rho}

Let $\rho^{\op L} , \rho^{\op R}  > -2$ and $\kappa > 0$. \notion{Chordal SLE$_\kappa(\rho^{\op L};\rho^{\op R})$} is a variant of SLE$_\kappa$ where one keeps track of two extra marked \emph{force points}, which was first studied in~\cite{lsw-restriction,sw-coord} (one can also consider more than 2 force points, but we will only need two here).  
Chordal SLE$_\kappa(\rho^{\op L};\rho^{\op R})$ on $(\BB H , 0,\infty)$ with force points started at $x^{\op L} < 0$ and $x^{\op R} > 0$ is the random curve $\eta$ which generates the Loewner evolution (recall~\eqref{eq:Loewner}) whose driving function $W_t$ solves the system of SDEs
\eqb
dW_t = \sqrt \kappa dB_t + \sum_{q\in \{L,R\}} \frac{\rho^q}{W_t - V_t^q} \,dt ,\quad dV_t^q = \frac{2}{V_t^q - W_t}\,dt
\eqe
with the initial conditions $W_0 = 0$, $V_0^{\op L}  =x^{\op L}$, $V_0^{\op R} = x^{\op R}$. 
See~\cite[Theorem 1.3 and Section 2.2]{ig1} for a proof that the solution to this SDE exists and that the resulting Loewner evolution is driven by a curve.
Note that SLE$_\kappa(0;0)$ is just ordinary SLE$_\kappa$. 

We will only need to consider the case when the force points are started immediately to the left and right of the origin, i.e, at $0^-$ and $0^+$, which can be obtained from the case above by taking a limit as $x^{\op L}\rta 0^-$ and $x^{\op R}\rta 0^+$. In this case $V_t^{\op L}$ (resp.\ $V_t^{\op R}$) is the image of the leftmost (resp.\ rightmost) point of $\eta([0,t]) \cap \BB R$ under the Loewner map $g_t$. 
From now on, we will assume that the force points are immediately to the left and right of the starting point whenever we talk about SLE$_\kappa(\rho^{\op L};\rho^{\op R})$. 

As for ordinary SLE$_\kappa$, the law of SLE$_\kappa(\rho^{\op L};\rho^{\op R})$ is invariant under spatial scaling, but unlike for ordinary SLE$_\kappa$ it does not satisfy the domain Markov property. For a simply connected domain $D\subset\BB C$ with distinct marked boundary points $a,b$, we define SLE$_\kappa(\rho^{\op L};\rho^{\op R})$ on $(D,a,b)$ to be the image of SLE$_\kappa(\rho^{\op L};\rho^{\op R})$ on $(\BB H , 0,\infty)$ under a conformal map (the choice of conformal map does not matter due to scale invariance). 

For $\ul\kappa \leq 4$, chordal SLE$_{\ul\kappa}(\rho^{\op L};\rho^{\op R})$ is always a simple curve, but for some values of $\rho^{\op L},\rho^{\op R}$ it can have boundary intersections. The SLE$_{\ul\kappa}$ curve locally looks like an SLE$_{\ul\kappa}$ curve away from its boundary intersections (in the sense of absolute continuity). The following lemma is proven in~\cite[Lemma 15]{dubedat-duality} (see also~\cite[Section 4]{ig1}).

\begin{lem} \label{lem-sle-rho-hit}
	For $\ul\kappa\leq 4$, SLE$_{\ul\kappa}(\rho^{\op L};\rho^{\op R})$ on $(\BB H  , 0,\infty)$ hits $(-\infty,0)$ (resp.\ $(0,\infty)$) if and only if $\rho^{\op L} < \ul\kappa/2-2$ (resp.\ $\rho^{\op R} < \ul\kappa/2-2$).
\end{lem}

In the case when SLE$_{\ul\kappa}(\rho^{\op L};\rho^{\op R})$ does hit the boundary, the boundary intersection is a totally disconnected uncountable fractal set. Its Hausdorff dimension is computed in~\cite{miller-wu-dim}. 

For $\kappa > 4$, we have $\kappa-6>-2$ and so we can define chordal SLE$_\kappa(\kappa-6;0)$ as above (we will typically drop the 0 in the parentheses and specify whether the force point is to the left or the right of the starting point instead). It follows from~\cite[Theorem 3]{sw-coord} that this process has the following special target invariance property which is important for the construction of conformal loop ensembles in~\cite{shef-cle} and for the construction of space-filling SLE just below.

\begin{lem} \label{lem-target-invariant}
	Let $\kappa > 4$, let $D\subset \BB C$ be simply connected and let $a,b,b'\in\bdy D$ be distinct.
	There is a coupling of chordal SLE$_\kappa(\kappa-6)$ curves on $(D,a,b )$ and $(D,a,b')$ such that the two curves agree until the first time at which they disconnect $b$ from $b'$. 
\end{lem}

If $\rho > -2$, one can also define \notion{whole-plane SLE$_\kappa(\rho)$} as a random continuous curve between any two points in $\BB C \cup\{\infty\}$; see~\cite[Section 2.1]{ig4} or~\cite[Section 6.3]{wedges}.  
Whole-plane SLE$_\kappa(\rho)$ from 0 to $\infty$ can be described as a whole-plane Loewner evolution with a certain explicit driving function, and its law is scale invariant.
As for chordal SLE$_\kappa(\rho^L;\rho^R)$, whole-plane SLE$_\kappa(\rho^L;\rho^R)$ has a distinguished ``force point" which one keeps track of. To explain this, let $K_t$ be the set of points disconnected from $\infty$ by the curve at time $t$ and view $\bdy K_t$ as a collection of prime ends.
For each $t\geq 0$, the force point is the image under the Loewner map $g_t : \BB C\setminus K_t \rta \BB C\setminus \ol{\BB D}$ of the unique point on $\bdy K_t$ other than the tip of the curve where the arcs of $\bdy K_t$ traced by the left and right sides of the curve meet. In the case when the curve is simple, this special point of $\bdy K_t$ is just the origin. 
Whole-plane SLE$_\kappa(\rho)$ between two distinct points in the Riemann sphere is the image of whole-plane SLE$_\kappa(\rho)$ from 0 to $\infty$ under an appropriate M\"obius transformation. 
The following lemma is a consequence of~\cite[Lemmas 2.4 and 2.6]{ig4}.

\begin{lem} \label{lem-sle-rho-whole-plane}
	For $\ul\kappa\leq 4$, whole-plane SLE$_{\ul\kappa}(\rho)$ has self-intersections if and only if $\rho < \ul\kappa/2-2$. 
\end{lem}

In the case when $\eta$ intersects itself, it does so infinitely many times and the complement of $\eta$ is a countable union of Jordan domains. 
The boundary of each of these Jordan domains is the union of two segments of $\eta$ (each of which looks like an SLE$_{\ul\kappa}$ curve) which intersect only at their endpoints. See Figure~\ref{fig-sle-rho} for illustrations of chordal SLE$_{\ul\kappa}(\rho^{\op{L}} ; \rho^{\op{R}})$ and whole-plane SLE$_{\ul\kappa}(\rho)$ curves in the regime when the weights are smaller than $\ul\kappa/2-2$.

\begin{figure}[ht!]
	\begin{center}
		\includegraphics[scale=0.92]{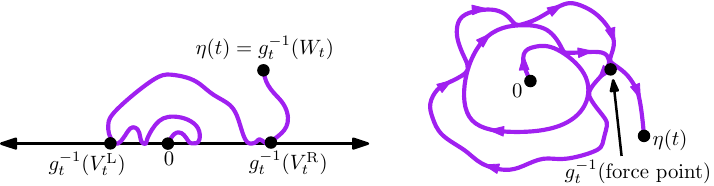}
	\end{center}
	\caption{\label{fig-sle-rho} 
		\textbf{Left:} Sketch of a segment of chordal SLE$_{\ul\kappa}(\rho^{\op{L}} ,\rho^{\op{R}})$ for $\ul\kappa \leq 4$ and $\rho^{\op L} ,\rho^{\op R} \in (-2,\ul\kappa/2-2)$. In actuality the boundary intersection is an uncountable, totally disconnected fractal set which intersects every neighborhood of $0$. If only one of $\rho^{\op L}$ or $\rho^{\op R}$ is less than $\ul\kappa/2-2$, the curve only intersects the boundary to one side of the origin. 
		\textbf{Right:} Sketch of a segment of whole-plane SLE$_{\ul\kappa}(\rho)$ for $\rho < \ul\kappa/2-2$. The arrows indicate the direction that the curve is traveling. 
		Unlike SLE$_\kappa$ for $\kappa \in (4,8)$, the curve can only intersect itself after winding around the origin so there are non-trivial segments of the curve which do not include any self-intersection points. 
	}
\end{figure}

\begin{remark}
	Both chordal SLE$_{\ul\kappa}(\rho^{\op L};\rho^{\op R})$ and whole-plane SLE$_{\ul\kappa}(\rho)$ for $\ul \kappa < 4$ arise as \emph{flow lines} of the Gaussian free field in the theory of imaginary geometry~\cite{ig1,ig2,ig3,ig4}. For $\kappa > 4$, such curves instead arise as \emph{counterflow lines}. 
\end{remark}

\subsubsection{Definition of space-filling SLE$_\kappa$}
\label{sec-space-filling}

Recall that for $\kappa \in (4,8)$, SLE$_\kappa$ hits (but does not cross) itself and makes ``bubbles", but does not fill space. 
In the mating-of-trees theory for $\gamma\in (\sqrt 2 ,2)$, it is crucial to have a space-filling variant of $\SLE_\kappa$ for $\kappa\in(4,8)$.
Intuitively, space-filling SLE$_\kappa$ for $\kappa \in (4,8)$ is obtained by starting with ordinary chordal SLE$_\kappa$ and iteratively filling in the ``bubbles" it makes by chordal SLE$_\kappa$ curves, then concatenating these curves in an appropriate order. 
The definition is the continuum analog of the construction of the discrete Peano curve $\lambda$ in Section~\ref{subsec:site}, with ordinary chordal SLE$_\kappa$ curves playing the role of percolation interfaces. 
We emphasize that the construction given in this section is equivalent to, but presented somewhat differently than, the original construction of space-filling SLE$_\kappa$ in~\cite{ig4}. The construction given here makes some of the properties of space-filling SLE$_\kappa$ as well as the link to the discrete setting more transparent. 

For pedagogical purposes, we first present the construction of  space-filling $\SLE_\kappa$ for $\kappa=6$, 
as it avoids one technicality and it is directly linked to Section~\ref{subsec:site} (see~Section~\ref{subsec:scaling}).  
For a given Jordan domain $D$ and  $a\in \p D$,
we first iteratively sample  two random orderings $\prec_+$ and $\prec_-$ on $\Q^2\cap D$, which we call the counterclockwise and clockwise ordering, respectively. 
These orderings will correspond to the order in which the counterclockwise and clockwise, respectively, space-filling SLE hits points of $\Q^2\cap D$.
See Figure~\ref{fig-sle6-decomp} for an illustration.

Suppose we want to define $\prec_\bullet$ with $\bullet=+$ or $-$. 
We start by assigning an orientation to $\p D$.
When $\bullet =+$ (resp.\ $\bullet=-$) we assign clockwise (resp.\ counterclockwise) orientation to $\p D$.\footnote{At first sight it may seem counterintuitive that we call the space-filling SLE counterclockwise (resp.\ clockwise) if in the iterative construction of the total ordering, $\bdy D$ is oriented clockwise (resp.\ counterclockwise). The reason for this convention is that, as we will see later, with this choice, the counterclockwise (resp.\ clockwise) SLE covers up the boundary of the domain in a counterclockwise (resp.\ clockwise) order.}
Let $b$ be a point on $\p D$ with maximal Euclidean distance to $a$, i.e., $b$ is diametrically opposite from $a$. 
Let $\ol{ab}$ be the  segment on $\bdy D$ from $a$ to $b$ with the \emph{same} orientation as $\p D$.
Let $\eta_D$ be an $\SLE_6$ on $(D,a,b)$. 

Suppose $\cB$ is a connected component of $D\setminus\eta_D$. If $\bdy \cB\cap \ol{ab}\neq \emptyset$, we call $\cB$ a \emph{dichromatic bubble}.\footnote{The convention for ``dichromatic" versus ``monochromatic" can be remembered by thinking of curves with clockwise orientation as being colored blue and curves with counterclockwise orientation as being colored red. Then dichromatic bubbles are the ones whose boundaries have both red and blue segments; see Figure~\ref{fig-sle6-decomp}.}
Otherwise, we call $\cB$ a \emph{first generation monochromatic bubble}, in which case we assign $\p \cB$ the orientation determined by the direction in which it is traced by $\eta_D$ and let its \emph{root} $x_\cB$ be the place where $\eta_D$ starts (equivalently, finishes) tracing $\bdy\cB$.

\begin{figure}[ht!]
	\begin{center}
		\includegraphics[scale=0.92]{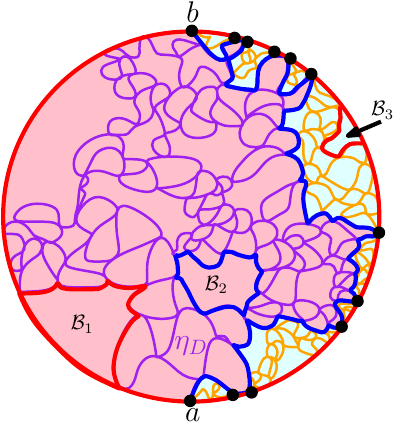}
	\end{center}
	\caption{\label{fig-sle6-decomp} 
		Illustration of the first stage of the inductive construction of the ordering $\prec_-$. Curves with clockwise (resp.\ counterclockwise) orientation are colored blue (resp.\ red). Points in first generation monochromatic bubbles (resp.\ dichromatic bubbles) are colored pink (resp.\ light blue). In each dichromatic bubble, the SLE$_6$ curve $\eta_\cB$ in $\cB$ between the two marked points $x_\cB$, $\wh x_\cB$ is shown in orange. The complementary connected components of this curve are the second-generation monochromatic bubbles. We have also outlined three representative monochromatic bubbles $\cB_1,\cB_2,\cB_3$. The space-filling SLE$_6$ curve will fill in $\cB_1$ before $\cB_2$ (since both are first generation and $\eta_D$ finishes tracing $\bdy\cB_1$ before $\bdy\cB_2$) and will fill in $\cB_2$ before $\cB_3$ (since $\cB_3$ is second generation).
		The restriction of $\prec_-$ to each monochromatic bubble with clockwise (resp.\ counterclockwise) oriented boundary is defined in the same way as $\prec_+$ (resp.\ $\prec_-$).
	}
\end{figure} 

For a dichromatic bubble $\cB$, let  $x_\cB$ and $\wh x_\cB$ be the last and first, respectively, point on $\bdy \cB$ visited by $\eta_D$. 
Conditioning on $\eta$, let $\eta_\cB$ be a chordal $\SLE_6$ on $(\cB,x_\cB,\wh x_\cB)$. 
We assume that these $\eta_\cB$'s for varying choices of $\cB$ are conditionally independent given $\eta$.
If $\cB'$ is a connected component of  $\cB\setminus \eta_\cB$, we call $\cB'$ a  \emph{second generation monochromatic bubble}.
We assign $\p \cB'$ the orientation in which it is traced by $\eta_\cB$ and  let its \emph{root} $x_{\cB'}$ be 
the last point traced by $\eta_{\cB}$.

The probability that each of the SLE$_6$ curves above hits a rational point is zero. Therefore, almost surely, $\Q^2\cap D$ is the disjoint union of 
\eqbn
\{\cB\cap \Q^2:\textrm{ $\cB$ is a first or second generation monochromatic bubble} \}.
\eqen
If $\cB$ and $\cB'$ are two monochromatic bubbles, then for $q\in \Q^2\cap \cB$ and $q'\in \Q^2\cap \cB'$,  we declare that
\begin{itemize}
	\item if $\cB$ and $\cB'$ are first generation, then $q\prec_\bullet q'$ if  $\eta_D$ finishes tracing $\p \cB$ before $\p \cB'$;
	\item if $\cB$ is first generation  and $\cB'$ is second generation, then $q\prec_\bullet q'$;
	\item if $\cB$ and $\cB'$ are both second  generation and in different dichromatic bubbles, then $q\prec_\bullet q'$  if  $\eta_D$ finishes tracing $\p \cB$ after  $\p \cB'$;
	\item if $\cB$ and $\cB'$ are both second  generation and in the same dichromatic bubble $\wt \cB$, then $q\prec_\bullet q'$ if  $\eta_{\wt \cB}$ finishes tracing $\p \cB$ before $\p \cB'$.
\end{itemize}
The above rules allow us to compare points in different monochromatic bubbles.

We now need to compare points within the same monochromatic bubble. We do this by induction. 
For each monochromatic bubble $\mcl B$, we sample an SLE$_6$ curve $\eta_\cB$ from $x_\cB$ to a point on $\bdy \cB$ which lies at maximal Euclidean distance from $x_\cB$, in such a way that the $\eta_\cB$'s for different choices of $\cB$ are independent.  
We extend $\prec_\bullet$ to the set of pairs $q,q' \in \Q^2 \cap \cB$ in the following manner. 
If $\p \cB$ is oriented clockwise (resp.\ counterclockwise), we define the restriction of $\prec_\bullet$ to $\Q^2\cap \cB$ in the same way we defined $\prec_+$ (resp.\ $\prec_-$) above but using $(\cB,\eta_\cB)$ in place of $(D,\eta_D)$. We require that these restricted orderings are conditionally independent given $\eta_D$ and the curves $\{\eta_\cB : \text{$\cB$ is a dichromatic bubble}\}$. 

We then iterate this procedure countably many times.
It is implicit in the construction of space-filling SLE$_\kappa$ in~\cite{ig4} that any distinct  points $q,q' \in \Q^2 \cap D$ lie in different monochromatic bubbles after finitely many steps, and hence $\prec_\bullet$ is defined for $q,q'$.
This defines the two desired random total orderings on $\Q^2\cap D$. We note that certain variants of the above construction would still give an ordering $\prec_\bullet$ with the same law; for example, by locality of SLE$_6$, we could change the way in which we choose the points $\wh x_{\cB}$ for $\cB$ a monochromatic bubble.

	\begin{remark}
		The space-filling SLE$_6$ is an exact continuum analog of the Peano curve on a site-percolated triangulation (see Section \ref{subsec:site}); namely the ordering $\prec_\bullet$ corresponds to the ordering defined by the function $\lambda$. The SLE$_6$ $\eta_D$ corresponds to the percolation interface $\lambda^{ab}$ in the discrete setting, and the relative ordering of the complementary components of this curve is similar in the discrete and the continuum settings. Boundary segments which are ordered clockwise (resp.\ counterclockwise) are blue (resp.\ red), and the terms monochromatic and dichromatic in this section refer to the boundary data of the associated discrete disk.
	\end{remark}

For $\kappa\in (4,8)$ with $\kappa\neq 6$ and $\bullet\in \{+,-\}$, we can sample $\prec_\bullet$ via the exact same iterative procedure except with one modification. 
Instead of chordal $\SLE_6$ on $(D,a,b)$, we take $\eta_D$ to be the $\SLE_\kappa(\kappa-6)$ on $(D,a,b)$ with the force point at $a^+$ (resp.\ $a^-$) if $\bullet=-$ (resp.\  $\bullet=+$).
We similarly take each $\eta_\cB$ for monochromatic bubbles $\cB$ to be chordal SLE$_\kappa(\kappa-6)$. 
The curves $\eta_\cB$ for dichromatic bubbles are still ordinary chordal SLE$_\kappa$.

\begin{thm}\label{thm:space-filling}
	Fix $\kappa\in (4,8)$ and  $\bullet=\{+,-\}$. Let $D$ be a Jordan domain and $a\in\p D$.
	Suppose $\prec_\bullet$ is the ordering on $\Q^2\cap D$ sampled  as above. 
	Then there almost surely exists a unique (modulo monotone parametrization) space-filling curve $\eta$ which does not cross itself or trace its past for a non-trivial interval of time, such that $\eta$ visits the points $\Q^2\cap D$ in the order of $\prec_\bullet$ and $\eta^{-1}(\Q^2 \cap D)$ is a dense set of times.
	Moreover, $\eta$ is continuous when parametrized so that the Lebesgue measure of $\eta([0,t])$ is $t$ for each $t\geq0$. 
\end{thm}

\begin{defn} \label{def:space-filling}
	If $\bullet=+$ (resp.\ $\bullet=-$), we call the curve $\eta$ of Theorem~\ref{thm:space-filling} the
	counterclockwise (resp.\ clockwise) \notion{space-filling $\SLE_\kappa$ loop} on $(D,a)$.
\end{defn}

Theorem~\ref{thm:space-filling} is proved in \cite[Theorem 1.6]{ig4}.
In fact, our definition of $\prec_\bullet$ coincides with the one in \cite{ig4} in terms of the so-called branching $\SLE_\kappa(\kappa-6)$ introduced in \cite{shef-cle}.
As explained in \cite[Section~4.3]{ig4},  the branching $\SLE_\kappa(\kappa-6)$ definition of $\prec_\bullet$ is in turn equivalent to a construction using the so-called \emph{flow lines} in the imaginary geometry prescribed by a zero-boundary GFF on $D$ plus a particular harmonic function. Using tools from imaginary geometry,   Theorem~\ref{thm:space-filling} is proved in this context.

\begin{remark}\label{rmk:CLE}
	In the iterative construction of $\prec_\bullet$ above, for each first generation dichromatic bubble $\cB$, let $\eta^\cB$ be the segment of $\eta_D$ in between the times it hits the marked points $x_\cB$ and $\wh x_\cB$. 
	Concatenating $\eta^\cB$ and $\eta_\cB$ gives an oriented loop $\ell_\cB$. 
	Let $\Gamma_a^b=\{\ell_\cB: \textrm{ $\cB$ is a first generation dichromatic bubble} \}$.
	For each monochromatic bubble, we can iterate the construction of $\Gamma^{a}_b$. 
	At the end the union of all loops forms the so-called \emph{conformal loop ensemble} ($\CLE_\kappa$), a collection of non-crossing SLE$_\kappa$-type loops whose law is conformally invariant which was first introduced in~\cite{shef-cle}. 
	This defines the canonical coupling of $\CLE_\kappa$ and the space-filling $\SLE_\kappa$ loop where each process determines the other one.  
\end{remark}

For $\kappa\ge 8$ and $(D,a)$ as in Theorem~\ref{thm:space-filling}, we define the clockwise space-filling $\SLE_\kappa$ on $(D,a)$ as the weak limit of chordal $\SLE_\kappa$ on $(D,a^-_n, a^+_n)$, where $a^+_n$ (resp.\ $a^{-}_n$) is converging to $a$ from the clockwise (resp.\ counterclockwise) direction. We define the counterclockwise space-filling $\SLE_\kappa$ on $(D,a)$ with $a^-_n$ and $a^+_n$ swapped. The existence of the limits can be checked using the flow line description of space-filling SLE$_\kappa$ from~\cite[Section 1.2.3]{ig4}; see~\cite[Appendix A.3]{bg-lbm} for details.

For all $\kappa>4$, we can also define  the  whole-plane variant of the space-filling $\SLE_\kappa$  as the local limit of space-filling $\SLE_\kappa$ on $(D,a)$.

\begin{lem}\label{lem:whole-plane}
	For $\kappa>4$, a Jordan domain $D$, $a\in \p D$, and $z\in D$, let $\wt \eta$ be a space-filling $\SLE_\kappa$ on $(D,a)$. 
	Suppose $\eta_n$ is a parameterization of  $\wt \eta$ such that $\eta_n(0)=z$ and in each unit of time $\eta_n$ traverses a region of Lebesgue area $n^{-2}$.
	Then $n(\eta_n-z)$ converge to a random curve $(\eta(t))_{t \in \R}$ in the local uniform topology, where  $\eta(0)=0$, and  in each unit of time $\eta$ traverses a region of unit Lebesgue area. The law of the limiting curve is independent of the orientation of $\wt\eta$. Moreover, the law of $\eta$ (viewed modulo time parameterization) is invariant under translations, rotations, and scalings of space.
\end{lem}

\begin{defn} \label{def:whole-plane}
	We call the curve $\eta$ of Lemma~\ref{lem:whole-plane} the \notion{whole-plane space-filling $\SLE_\kappa$ from $\infty$ to $\infty$} parametrized by the Lebesgue measure. 
\end{defn}

Lemma~\ref{lem:whole-plane} is another easy consequence of imaginary geometry. See, e.g., \cite[Lemma~2.3]{hs-euclidean} and its proof.

\subsubsection{Basic properties of space-filling SLE$_\kappa$}
\label{sec-space-filling-properties}

If $\eta$ is a whole-plane space-filling SLE$_\kappa$ from $\infty$ to $\infty$, then a.s.\ $\lim_{t\to \pm \infty}\eta(t)=\infty$. Setting  $\eta(\pm  \infty)=\infty$, we can  therefore view $\eta$ as a continuous curve on the Riemann sphere $\CC\cup \{\infty\}$. It is an immediate consequence of the flow line construction in~\cite{ig4} that the law of this curve is reversible: $\eta(-\cdot) \eqD \eta$. 

A.s.\ the curve $\eta$ visits $0$ only once, namely at time $t=0$.
The set  $\eta([0,\infty)  )\cap \eta((-\infty,0])$ is the union of two non-self-crossing curves $\eta^0_{\op L}$ and $\eta^0_{\op R}$ from $0$ to $\infty$. 
By convention, we assume that $\eta((-\infty,0])$  is on the  left (resp.\ right) side $\eta^0_{\op L}$ (resp.\ $\eta^0_{\op R}$). See Figure \ref{fig-space-filling}. 

\begin{definition}\label{def:bdy}
	We call  $\eta^0_{\op L}$ and $\eta^0_{\op R}$ the \notion{left boundary} and the \notion{right boundary}, respectively, of $\eta$ at $0$. 
\end{definition}

The imaginary geometry construction of space-filling SLE$_\kappa$ in~\cite{ig4} implies that $\eta_0^{\op L}$ and $\eta_0^{\op R}$ can be described as two flow lines of the same whole-plane GFF whose angles differ by $\pi$. 
Equivalently, using~\cite[Theorems 1.1 and 1.11]{ig4}, the joint law of $\eta_0^{\op L}$ and $\eta_0^{\op R}$ can be described as follows, where here $\ul\kappa = 16/\kappa \in (0,4)$. 
\begin{itemize}
	\item The law of the left outer boundary $\eta_0^{\op{L}}$ is that of a whole-plane SLE$_{\ul\kappa}(2-\ul\kappa)$ from 0 to $\infty$.
	\item Conditional on $\eta_0^{\op{L}}$, the law of the right outer boundary $\eta_0^{\op{R}}$ is that of a chordal SLE$_{\ul\kappa}(-\ul\kappa/2,-\ul\kappa/2)$ from 0 to $\infty$ in $\BB C\setminus \eta_0^{\op{L}}$ (with force points immediately to the left and right of zero). 
\end{itemize}
The same is true with the roles of $\eta_0^{\op{L}}$ and $\eta_0^{\op{R}}$ interchanged. 
This result is an instance of \emph{SLE duality}~\cite{dubedat-duality,zhan-duality1,zhan-duality2,ig1,ig4}

From this description and Lemma~\ref{lem-sle-rho-hit}, it follows that for $\kappa\ge 8$, $\eta_0^{\op L}$ and  $\eta_0^{\op R}$ are simple curves that do not intersect except at $0$ and $\infty$. 
Therefore, both $\eta([0,\infty])$ and $ \eta([-\infty,0])$ are homeomorphic to the closed half-plane $\ol{\BB H}$. 
If we condition on $\eta_0^{\op L}$ and $\eta_0^{\op R}$, then the conditional law of $\eta|_{[0,\infty)}$ (resp.\ the time reversal of $\eta|_{(-\infty,0]}$) is that of a chordal SLE$_\kappa$ from 0 to $\infty$ in $\eta([0,\infty))$ (resp.\ $\eta((-\infty,0])$) and these curves are conditionally independent (see~\cite[Footnote 4]{wedges}). 

When $\kappa\in (4,8)$, $\eta^0_{\op R}$ touches $\eta^0_{\op L}$ infinitely many times from both sides, but $\eta_0^{\op{L}}$ and $\eta_0^{\op{R}}$ do not cross and $\eta_0^{\op L} \cap \eta_0^{\op R}$ contains no non-trivial interval. 
Therefore the interior of each of $\eta([0,\infty)) $ and $ \eta((-\infty,0])$ is an infinite chain of Jordan domains. 
See Figure~\ref{fig-space-filling} for an illustration.  
In this case, if we condition on $\eta_0^{\op L}$ and $\eta_0^{\op R}$, then the conditional law of $\eta|_{[0,\infty)}$ (resp.\ the time reversal of $\eta|_{(-\infty,0]}$) is that of a concatenation of conditionally independent chordal space-filling SLE$_\kappa$ curves in the connected components of the interior of $\eta([0,\infty))$ (resp.\ $\eta((-\infty,0])$) and these curves are conditionally independent (see~\cite[Footnote 4]{wedges}).

\begin{figure}[ht!]
	\begin{center}
		\includegraphics[scale=0.92]{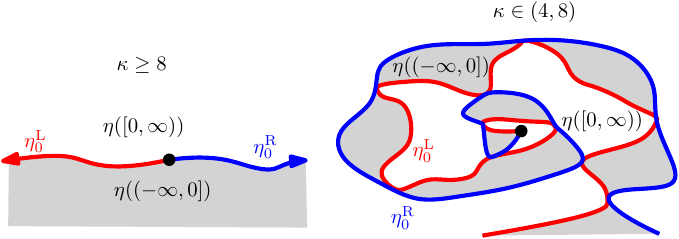}
	\end{center}
	\caption{\label{fig-space-filling} 
		The left and right outer boundaries of $\eta((-\infty,0])$ (red and blue) in the case when $\kappa \geq 8$ (left) and $\kappa \in (4,8)$ (right). Note that the pictures are not entirely accurate since the set of intersections of $\eta_0^{\op L}$ and $\eta_0^{\op R}$ for $\kappa \in (4,8)$ is in fact an uncountable fractal set.
	}
\end{figure} 

By the translation invariance of the law of $\eta$ (see Lemma~\ref{lem:whole-plane}),  for each fixed $z\in \CC$, we  define  the left boundary    $\eta_z^{\op L}$  and the  right boundary $\eta_z^{\op R}$ of $\eta$ at $z$ in the same way as above, and all of the above properties remain true a.s.\ with $z$ in place of 0. 
However, there exist random $z\in\BB C$ for which these properties are not true. 
For example, suppose $z\neq 0$ is on $\eta_0^{\op L}$. Then $\eta$ will visit $z$ at least twice, once before time 0 and once afterwards.
When $\kappa\in(4,8)$, there are even more intriguing special subsets, such as double points, cut points, points on an ordinary SLE$_\kappa$ curve, etc.

The construction of space-filling SLE$_\kappa$ in~\cite{ig4} shows that the curves $\eta_z^{\op{L}}$ and $\eta_z^{\op{R}}$ are all flow lines of a common whole-plane GFF.
From this it follows that a.s.\ $\eta_z^{\op L}$ and $\eta_w^{\op L}$ for distinct pairs of points $z$ and $w$ eventually merge into each other, and similarly with R in place of L. Hence the collections of curves $\{\eta^{\op L}_z\}_{z\in \Q^2}$ and $\{\eta_z^{\op R}\}_{z\in \Q^2}$ 
	each have the structure of a tree.
When $\kappa=8$, these trees constitute the joint scaling limit of the UST on $\Z^2$ and its dual~\cite{schramm0,lsw-lerw-ust,hs-euclidean}. As in the discrete setting of Section~\ref{subsec:UST}, we can think of the whole-plane space-filling $\SLE_\kappa$ curve $\eta$ as the Peano curve snaking in between the two trees.

Besides the aforementioned qualitative properties, some useful quantitative ones have also been established. 
In what follows, we fix $\kappa \in (4,\infty)$ and let $\eta$ be a whole-plane space-filling SLE$_\kappa$ parametrized by Lebesgue measure. 
The following proposition, which is a combination of~\cite[Proposition 3.4 and Remark 3.9]{ghm-kpz}, says that space-filling SLE segments are ``roughly spherical" in the sense that they are contained between two Euclidean balls of comparable radii (up to $o(1)$ errors in the exponent).

\begin{prop}[\cite{ghm-kpz}] \label{prop-sle-ball}
	Fix $\zeta \in (0,1)$ and a bounded open set $U\subset\BB C$.
	Except on an event of probability decaying faster than any positive power of $\ep$, the following is true.
	For each $a,b\in\BB R$ such that $\eta([a,b])\subset U$ and the Euclidean diameter satisfies $\op{diam}\eta([a,b])\leq \ep$, the space-filling SLE segment $\eta([a,b])$ contains a Euclidean ball of radius at least $[\op{diam}\eta([a,b])]^{1+\zeta}$. 
\end{prop}

The following proposition, which is~\cite[Proposition 6.2]{hs-euclidean}, gives an upper bound on how long it takes $\eta$ to fill in a given region of space.
\begin{prop}[\cite{hs-euclidean}] \label{prop-sle-fill}
	There exists $\xi = \xi(\kappa) >0$ such that 
	\eqbn
	\BB P\left[\BB D\subset \eta([-M,M]) \right] \geq 1 - O (M^{-\xi}) \quad \text{as $M\rta \infty$}. 
	\eqen
\end{prop}

\section{Continuum mating of trees}
\label{sec-mating}
In this section we will discuss several deep theorems which connect SLE and LQG. The first class of theorems, discussed in Section~\ref{subsec:zipper}, tell us that cutting a quantum wedge by an appropriate type of SLE$_{\ul\kappa}$ curve for $\ul\kappa=\gamma^2$ gives a pair of \emph{independent} quantum wedges parametrized by the regions on either side of the curve, and a similar statement holds if we instead cut a quantum cone. 
Sections~\ref{sec-mating-main} and~\ref{sec-mating-geometric} discuss two variants of the mating-of-trees theorem (the continuum analog of mating-of-trees bijections), in the infinite-volume and finite-volume settings, respectively.
Section~\ref{sec-mating-chordal} discusses a different sort of mating-of-trees theorem which represents a quantum wedge decorated by a chordal SLE$_\kappa$ curve for $\kappa = 16/\gamma^2 \in (4,8)$ as a mating of two trees of disks.

\subsection{Conformal welding for quantum wedges }\label{subsec:zipper}

The first rigorous connection between SLE and LQG, which constitutes the starting point of mating-of-trees theory, is the following theorem of Sheffield \cite{shef-zipper}.

\begin{figure}[t]
	\begin{center}
		\includegraphics[scale=0.92]{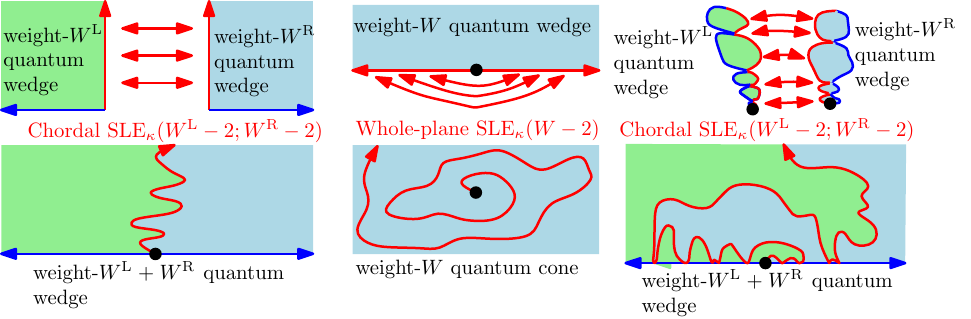}
	\end{center}
	\caption{\label{fig-welding}
		\textbf{Left:} Conformally welding two quantum wedges of weights $W^{\op L}$ and $W^{\op R}$ according to LQG length along half of their boundaries produces a quantum wedge of weight $W^{\op L} + W^{\op R}$ decorated by a chordal SLE$_\kappa(W^{\op L} -2 ; W^{\op R} - 2)$. 
		\textbf{Middle:} Conformally welding the two sides of the boundary of a weight-$W$ quantum wedge according to LQG length produces a weight-$W$ quantum cone decorated by a whole-plane SLE$_\kappa(W)$. 
		\textbf{Right:} Same picture as on the left but in the case when the wedges being welded together are thin and the wedge of weight $W^{\op L}+W^{\op L}$ is thick.  
	}
\end{figure} 

\begin{thm}[\cite{shef-zipper}]\label{thm:zipper0}
	Let $(\BB H,\bh ,0,\infty)$ be the circle average  embedding of a $(\gamma-2\gamma^{-1})$-quantum wedge (Definition~\ref{def:wedge}).
	Let $\eta$ be a chordal $\SLE_{\ul\kappa}$ on $(\BB H,0,\infty)$ where $\ul\kappa=\gamma^2$, sampled independently from $\bh$. 
	Let $\BB H^{\op L}$ and $\BB H^{\op R}$ be the connected components of $\BB H \setminus\eta$ lying to the left and right of $\eta$, respectively.
	Let $\cW^{\op L}$ and $\cW^{\op R}$ be the quantum surfaces $(\BB H^{\op L},\bh|_{\BB H^{\op L}}, 0,\infty)$ and $(\BB H^{\op R},\bh|_{\BB H^{\op R}}, 0,\infty)$, respectively. 
	\begin{itemize}
		\item $\cW^{\op L}$ and $\cW^{\op R}$ are independent $\gamma$-quantum wedges. 
		\item The $\gamma$-LQG length measures on $\eta$ as viewed from the left and right sides of $\eta$ coincide, i.e., if for $\bullet \in \{\op L , \op R\}$,   $(\BB H,\bh^{\bullet},0,\infty)$ is the circle average embedding of $\cW^{\bullet}$ and 
		$\phi_\bullet : \BB H^\bullet\to \BB H$ is such that $(\BB H,\bh^\bullet ,0,\infty)\overset{\phi_\bullet}{\sim}_{\gamma} (\BB H^{\bullet},\bh|_{\BB H^{\bullet}}, 0,\infty)$,
		then a.s.\ $\nu_{\bh^{\op L}} ([0,\phi _{\op L}(x)]) =\nu_{\bh^{\op R}}([\phi_{\op R}(x),0])$ for each $x\in \eta$.
		\item $(\bh,\eta)$ is a.s.\ determined  by $\cW^{\op L}$ and $\cW^{\op R}$ (see the beginning of Section~\ref{sec-sle-lqg} for notation).
	\end{itemize}
\end{thm}

As a consequence of the second assertion of Theorem~\ref{thm:zipper0}, the boundary length measure $\nu_\bh$ is well-defined on $\eta$ in the sense that we get the same length measure whether we measure lengths from the left or right side of $\eta$. 
From this we infer that if $\wt h$ (resp.\ $\wt\eta$) is any distribution whose law is locally absolutely continuous with respect to the law of the GFF (resp.\ an SLE$_{\ul\kappa}$ curve) and $\wt h$ and $\wt\eta$ are independent, then the $\gamma$-LQG length measure $\nu_{\wt h}$ on $\wt\eta$ is well-defined. 
The $\gamma$-LQG length measure on an SLE$_{\ul\kappa}$ curve defined in this way can also be described as a Gaussian multiplicative chaos measure with the Minkowski content of the SLE$_{\ul\kappa}$ curve (which was shown to exist in~\cite{lawler-rezai-nat}) as the base measure; see~\cite{benoist-lqg-chaos}. 
 
The proof of the first two assertions of Theorem~\ref{thm:zipper0} are quite difficult and we will not attempt to review the proof of these (but see \cite{berestycki-lqg-notes} for some expository notes).
The measurability statement in Theorem~\ref{thm:zipper0} is a neat observation based on  the conformal removability of $\SLE_{\ul\kappa}$  with $\ul{\kappa}\in(0,4)$ 
(see \cite[Proof of Theorem~1.4]{shef-zipper}). We now explain how this argument goes. 
A set $A\subset\BB C$ is \emph{conformally removable} if every homeomorphism $\BB C\rta \BB C$ which is conformal on $\BB C \setminus A$ is in fact conformal on all of $\BB C$ (so is a complex affine map). 
The a.s.\ conformal removability of SLE$_{\ul\kappa}$ curves follows by combining results in~\cite{schramm-sle,jones-smirnov-removability}. 

Suppose there exists $(\wt h,\wt \eta)$ such that $(\wt h,\wt \eta,\cW^{\op L},\cW^{\op R})\eqD(h,\eta,\cW^{\op L},\cW^{\op R})$. Define $(\wt \phi_{\op L}, \wt\phi_{\op R})$.
in the same way as $(\phi_{\op L},\phi_{\op R})$. Then by the second assertion in Theorem~\ref{thm:zipper0}, 
$\wt \phi_{\op L}^{-1}\circ\phi_{\op L}$ and $\wt \phi_{\op R}^{-1}\circ\phi_{\op R}$  together extend to a homeomorphism $\phi:\BB H\to \BB H$, which is conformal everywhere on $\BB H\setminus\eta$.  The  conformal removability of $\eta$ ensures that such a map $\phi$ has to be conformal everywhere. 
Since we are assuming both $h$ and $\wt h$ are the circle average embeddings, we must have that $\phi$ is the identity map and hence $(\wt h,\wt\eta)=(h,\eta)$. 

See~\cite{hp-welding,mmq-welding} for the variant of Theorem~\ref{thm:zipper0} when $(\gamma,\ul\kappa) = (2,4)$. We remark that the paper~\cite{astala-welding} studies a variant of the setup of Theorem~\ref{thm:zipper0} where one conformally welds an LQG surface to a Euclidean surface, but in this case the curve which is the gluing interface is not an SLE (see also Problem~\ref{prob-mismatched-welding}). 

Theorem~\ref{thm:zipper0} is a continuum analog of cutting/gluing statements for random planar maps. See~\cite[Section 2.2]{shef-kpz} for a conjectural example in the case of spanning-tree weighted random planar maps and Section~\ref{sec-kappa6} for an overview of a rigorous example in the case of uniform random planar maps. 
  
Theorem~\ref{thm:zipper0} is vastly generalized in \cite{wedges}.  To state it we first introduce a new parametrization of quantum wedges and cones. For $\alpha<Q+\gamma/2$, we call 
\begin{equation}\label{eq:weight}
W\defeq \gamma(Q+\frac\gamma2-\alpha)
\end{equation}
the \notion{weight} of an $\alpha$-quantum wedge. In the terminology introduced at the end of Section~\ref{subsec:cone-wedge}, $W\geq \gamma^2/2$ corresponds to a thick wedge while $W\in(0,\gamma^2/2)$ corresponds to a thin wedge. 
The weight $W$ is sometimes a more convenient parameter to work with than $\alpha$ since it is additive under cutting and gluing operations (see theorems below).
We also recall the definition of SLE$_{\ul\kappa}(\rho^{\op L} ; \rho^{\op R})$ from Section~\ref{sec-sle-rho}. 
With this notation, the following theorem (\cite[Theorem~1.2]{wedges}) explains how to cut a thick wedge using $\SLE_{\ul\kappa}(\rho^{\op L};\rho^{\op R})$. 

\begin{thm}[Cutting a thick wedge]\label{thm:zipper1}
	Fix   $\rho^{\op L}>-2$ and $\rho^{\op R}>-2$. Set $W^{\op L}=\rho^{\op L}+2$, $W^{\op R}=\rho^{\op R}+2$, and $W=W^{\op L}+W^{\op R}$. Suppose $W\ge \gamma^2/2$. 
	Let $(\BB H,\bh,0,\infty)$ be the circle average embedding of a  quantum wedge of weight $W$. 	
	Let $\eta$ be a chordal $\SLE_{\ul\kappa}(\rho^{\op L};\rho^{\op R})$ from 0 to $\infty$ in $\BB H$, where $\ul\kappa = \gamma^2$, sampled independently from $\bh$. 
	Let $\BB H^{\op L}$ (resp.\ $\BB H^{\op R}$) be the union of the connected components of $\BB H \setminus\eta$ which lie to the left (resp.\ right) of $\eta$.
	Let $\cW^{\op L}$, $\cW^{\op R}$ be the quantum surface $(\BB H^{\op L},\bh|_{\BB H^{\op L}}, 0,\infty)$ and $(\BB H^{\op R},\bh|_{\BB H^{\op R}}, 0,\infty)$  respectively. 
	\begin{itemize}
		\item $\cW^{\op L}$ and $\cW^{\op R}$ are independent  quantum wedges of weights  $W^{\op L}$ and $W^{\op R}$, respectively. 
		\item The parametrizations of $\eta$ induced by the  quantum length measure on the right boundary of $\cW^{\op L}$ and the left boundary of $\cW^{\op R}$ agree a.s.\ (as in Theorem~\ref{thm:zipper0}).
		\item  $(\bh,\eta)$ is almost surely determined by $(\cW^{\op L},\cW^{\op R})$. 
	\end{itemize}
\end{thm}
Since  a $(\gamma-2\gamma^{-1})$-quantum wedge has weight $2+2=4$,
Theorem~\ref{thm:zipper0} is the special case of Theorem~\ref{thm:zipper1} where $\rho^{\op L}=\rho^{\op R}=0$ and $W^{\op L}=W^{\op R}=2$.

Theorem~\ref{thm:zipper1} is consistent with the  topological properties of $\SLE_{\ul{\kappa}}(\rho^{\op L};\rho^{\op R})$ in Section~\ref{sec-sle-rho}.
Note that $\cW^{\op L}$ is a thick wedge if and only if $W^{\op L}\ge \gamma^2/2=\ul{\kappa}/2$.
Since $\rho^{\op L} = W^{\op{L}}-2$, this is consistent with Lemma~\ref{lem-sle-rho-hit}.
When $\rho^{\op L}\in (-2,\ul{\kappa}/2 -2)$,  the two curves intersect infinitely many times, 
thus enclosing a region with the topology of the thin wedge $\cW^{\op L}$.  The same statement holds if ${\op L}$ is replaced by $\op R$.

The last assertion in Theorem~\ref{thm:zipper1} is equivalent to the fact that $(\BB H , \bh,\eta,a,b)$  is a.s.\ determined by $(\cW^{\op L},\cW^{\op R})$ as a curve-decorated quantum surface (Definition~\ref{def-curve-decorated}). 
Therefore, in Theorem~\ref{thm:zipper1}  it is not essential to assume that $\cW$ is under the circle average embedding or  any other particular embedding.
\begin{remark}[Cutting a thin wedge]\label{thm:zipper2}
	There is also a variant of Theorem~\ref{thm:zipper1} where $W\in (0,\gamma^2/2)$ so that $\cW$ is a thin wedge.  In this case, we take $\eta$ to be  the concatenation of independent $\SLE_{\ul\kappa}(\rho^{\op L};\rho^{\op R})$ curves in each bead of $\cW$. We also require that $\eta$ is independent of the fields $\{h^e\}$ under the embedding of $\cW$ given by Definition~\ref{def:thin-wedge}.
	This uniquely specifies the law of $(\cW,\eta)$ as a curve-decorated quantum surface.
	This way, all the three assertions in Theorem~\ref{thm:zipper2} remain true.
\end{remark}

There is also  a variant of Theorem~\ref{thm:zipper0} about cutting quantum cones by whole-plane $\SLE_{\ul\kappa} (\rho)$ \cite[Theorem~1.5]{wedges}.
Given $\alpha<Q$, we define  the \notion{weight} of  an $\alpha$-quantum cone by
\begin{equation}\label{eq:weight-cone}
W\defeq2\gamma(Q-\alpha).
\end{equation}

\begin{thm}\label{thm:zipper3}
	Fix   $\rho>-2$ and set $W=\rho+2$. 
	Let $(\CC,\bh,0,\infty)$ be the circle average  embedding of a quantum cone of weight $W$. Let $\eta$ be a whole-plane $\SLE_{\ul\kappa}(\rho)$ on $(\CC,0,\infty)$ where $\ul\kappa=\gamma^2$, sampled independently from $\bh$. 
	\begin{itemize}
		\item The quantum surface $\mcl W  := (\CC\setminus\eta, \bh  |_{\CC\setminus\eta} ,0,\infty)$ is a quantum wedge of weight $W$. 
		\item The parametrizations of $\eta$ induced by the  quantum length measure on the left and right boundaries of $\mcl W$ agree a.s.
		\item $\mcl W$ a.s.\ determines $(\bh,\eta)$ modulo rotation about the origin.
	\end{itemize}
\end{thm}

When $W\ge \gamma^2/2$, $\cW$ is a thick wedge so $\CC\setminus\eta$ is homeomorphic to the half-plane.  
When  $W\in(0,\gamma^2/2)$ so that $\rho\in(-2, \frac{\ul{\kappa}}{2}-2)$, $\cW$ is a thin wedge so $\CC\setminus\eta$ is an  ordered collection of Jordan domains  with two marked points (see the discussion below Lemma~\ref{lem-sle-rho-whole-plane}). 
This is consistent with the  topological properties of  whole-plane $\SLE_{\ul{\kappa}}(\rho)$ discussed in Section~\ref{sec-sle-rho}.

Both $\SLE_{\ul{\kappa}}(\rho^{\op L},\rho^{\op R})$ and whole-plane $\SLE_{\ul\kappa}(\rho)$ curves in the welding theorems above are conformally removable. Therefore, the measurability statements in these theorems follow from the argument below Theorem~\ref{thm:zipper0}.

There is a dynamical formulation of Theorem~\ref{thm:zipper0} which is also instrumental.
\begin{prop}\label{prop:zipper}
	In the setting of Theorem~\ref{thm:zipper0}, if $\eta$ is under the quantum natural parametrization, 
	then the laws of the curve-decorated quantum surfaces $(\BB H\setminus \eta[0,t],\bh, \eta(t),\infty, \eta|_{[t,\infty)})_{t\ge 0}$ are stationary in $t$.
\end{prop}
See Section~\ref{sec-mating-chordal} for a generalization of Theorem~\ref{thm:zipper0} and Proposition~\ref{prop:zipper} to chordal $\SLE_\kappa$   on a quantum wedge where $\kappa=16/\gamma^2\in (4,8)$.

\subsection{Infinite-volume mating of trees} 
\label{sec-mating-main}

\subsubsection{The whole-plane mating-of-trees theorem}

Let $(\BB C , \bh , 0, \infty)$ be a $\gamma$-quantum cone with the circle average embedding and let $\mu_\bh$ and $\nu_\bh$ be its $\gamma$-LQG area measure and boundary length measure, respectively. Recall from the discussion just after Theorem~\ref{thm:zipper0} that $\nu_\bh$ makes sense as a measure on any SLE$_{\ul\kappa}$-type curve sampled independently from $\bh$ for $\ul\kappa = \gamma^2$. Let $\eta$ be a whole-plane space-filling SLE$_\kappa$ from $\infty$ to $\infty$ sampled independently from $\bh$ and then parametrized in such a way that $\eta(0) = 0$ and $\mu_\bh(\eta([t_1,t_2])) = t_2-t_1$ whenever $t_1<t_2$. Note that this parametrization is natural in light of Remark~\ref{rmk:conj}. 

Define a process $L : \BB R \rta \BB R$ in such way that $L_0 = 0$ and for $t_1,t_2 \in \BB R$ with $t_1<t_2$, 
\allb \label{eqn-peanosphere-bm}
L_{t_2}- L_{t_1} &= \nu_\bh\left( \text{left boundary of $\eta([t_1,t_2]) \cap \eta([t_2,\infty))$} \right) \notag \\
&\qquad - \nu_\bh\left( \text{left boundary of $\eta([t_1,t_2]) \cap \eta((-\infty , t_1])$} \right) ,
\alle 
so that $L$ gives the net change of the LQG length of the left outer boundary of $\eta$ relative to time 0. 
Define $R : \BB R\rta \BB R$ similarly but with ``right" in place of ``left" and set $Z := (L,R)$. See Figure~\ref{fig-peanosphere-bm-def} for an illustration. 
The process $Z$ is the continuum analog of the encoding walk in the mating-of-trees bijections in Section~\ref{sec-bijections}.

\begin{figure}[ht!]
	\begin{center}
		\includegraphics[scale=0.92]{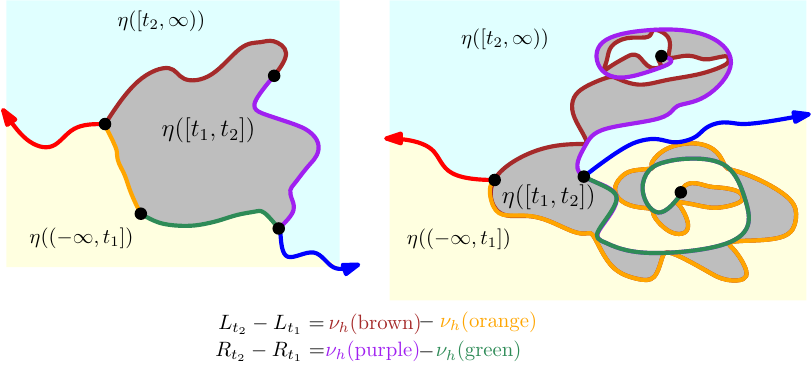}
	\end{center}
	\caption[Definition of the peanosphere Brownian motion]{\label{fig-peanosphere-bm-def} 
		Illustration of the definition of the peanosphere Brownian motion $Z = (L,R)$ in the case when $\kappa \geq 8$ (left) and the case when $\kappa \in (4,8)$ (right). 
		The left (resp.\ right) outer boundary of $\eta((-\infty,t_1])$ is the union of the red and orange (resp.\ blue and green) curves and the left (resp.\ right) outer boundary of $\eta((-\infty,t_2])$ is the union of the red and brown (resp.\ blue and purple) curves.
	}
\end{figure} 

The following fundamental theorem is proven in~\cite[Theorems 1.9 and 1.11]{wedges}. See also~\cite{kappa8-cov} for the computation of the correlation in the case when $\kappa  > 8$. 

\begin{thm}[Continuum mating of trees~\cite{wedges}] \label{thm-mating}
	There is a deterministic constant $\BB a = \BB a(\gamma) > 0$ such that the process $Z = (L,R)$ evolves as a correlated two-dimensional Brownian with variances and covariances
	\eqb \label{eqn-bm-cov}
	\op{Var}(L_t) = \op{Var}(R_t) = \BB a |t| , \quad \text{and} \quad \op{Cov}(L_t,  R_t) = -\BB a |t| \cos\left(\frac{\pi \gamma^2}{4} \right) , \quad\forall t\in\BB R .
	\eqe
	Furthermore, $Z$ a.s.\ determines $(\bh,\eta)$ modulo rotation about the origin. 
\end{thm}
The last statement of the theorem means that there is a measurable function $\Phi$ from $\{$continuous paths $\BB R\rta \BB R^2$ $\}$ to $\{\text{distribution/curve pairs}\}$ such that a.s.\ $\Phi(Z)$ and $(h,\eta)$ differ by a rotation about the origin.

We will now explain why the above theorem gives us a construction of SLE and LQG known as the mating-of-trees construction. 
Each Brownian motion $L$ and $R$ encodes a continuum random tree (CRT) \cite{aldous-crt1}. 
The Brownian motion $L$ or $R$ is the continuum analogs of the contour function of a discrete planar tree in Section~\ref{subsec:UST}.
As explained in Figure~\ref{fig-peanosphere} these CRT's can be glued together to form a topological space (homeomorphic to $\BB C$) decorated by a space-filling curve. 
This object is called the \emph{infinite-volume peanosphere} constructed from $Z$. 
As a consequence of Theorem~\ref{thm-mating}, the curve-decorated topological space $(\BB C,\eta)$ is homeomorphic (via a curve-preserving homeomorphism) to the infinite-volume peanosphere constructed from $Z$.

The variance constant $\BB a$ is not computed explicitly in~\cite{wedges}, but so far most of the applications of Theorem~\ref{thm-mating} do not depend on the value of this constant.

\begin{figure}[t]
	\begin{center}
		\includegraphics[scale=0.65]{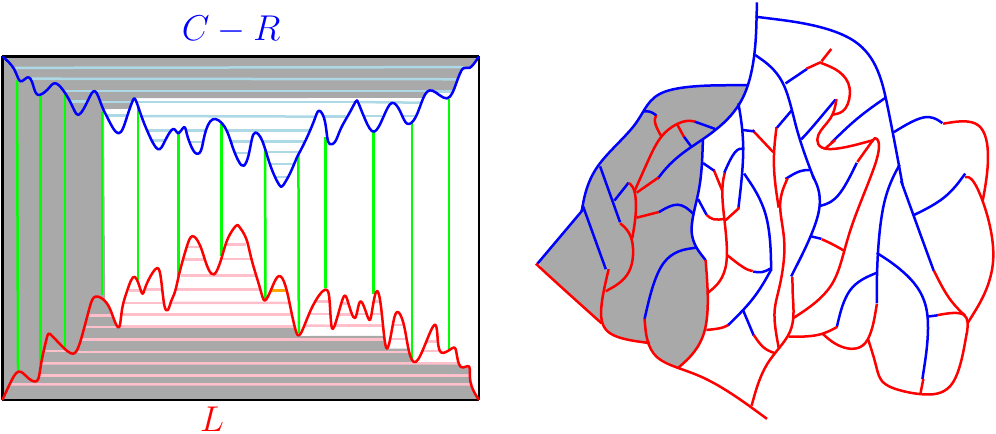}
	\end{center}
	\caption[Peanosphere definition]{\label{fig-peanosphere} 
		The peanosphere construction of \cite{wedges} shows how to obtain a topological sphere decorated by a space-filling curve by gluing together two correlated Brownian excursions $L,R \colon [0,1] \to [0,\infty)$. As explained in~\cite[Section 8.2]{wedges},  a similar construction works for  a pair of correlated Brownian motions, which corresponds to the ``local limit" of the Brownian excursion construction at a point in the interior of $[0,1]$. Let $C>0$ be a random constant chosen so that the graphs of $C-R$ and $L$ do not intersect. We define an equivalence relation on the square $[0,1] \times [0,C]$ by identifying any two points which lie on the same horizontal line segment under the graph of $L$ or above the graph of $C-R$; or on the same vertical line segment between the two graphs. Several such segments are shown in the figure. As explained in \cite{wedges}, it follows from Moore's theorem~\cite{moore} that the topological quotient space under this equivalence relation is a topological sphere decorated by a space-filling curve $\eta$, where $\eta(t)$ is the image of $(t,L_t)$ (equivalently, $(t,C-R_t)$) under the quotient map. 
		This topological space decorated by a space-filling curve is called the \emph{peanosphere}. Heuristically speaking, the peanosphere is obtained by gluing together the continuum random trees (CRT's) associated with $L$ and $R$, and the curve $\eta$ snakes between these two trees (strictly speaking, however, the image of each of the CRT's under the quotient map is the whole space). 
		\cite[Theorem 1.1]{sphere-constructions} shows that there is a canonical embedding of the peanosphere into $\BB C$ whereby the curve $\eta$ maps to a space-filling SLE$_\kappa$ from $\infty$ to $\infty$ parametrized by $\gamma$-LQG mass with respect to an independent quantum sphere. By Theorem~\ref{thm-mating}, the same holds for the infinite-volume peanosphere, but with the quantum sphere replaced by a $\gamma$-quantum cone. }
\end{figure}

\subsubsection{Outline of the proof of Theorem~\ref{thm-mating}}

The first step of the proof of Theorem~\ref{thm-mating} is the following lemma, which is a consequence of the conformal welding results from~\cite{shef-zipper,wedges}, as explained in Section~\ref{subsec:zipper}. 
For the statement, we recall that the weight of an $\alpha$-quantum wedge is $\gamma\left(Q +\frac{\gamma}{2} - \alpha \right)$.

\begin{lem} \label{lem-past-future-ind}
	The curve-decorated quantum surfaces
	\eqb
	\left( \eta((-\infty,0]) ,  \bh|_{ \eta((-\infty,0])} , \eta|_{(-\infty,0]}   \right) \quad \text{and} \quad \left( \eta([0,\infty) ) ,  \bh|_{ \eta([0,\infty))} , \eta|_{[0,\infty)}   \right)
	\eqe
	are independent.
	Each has the law of a quantum wedge of weight $2-\gamma^2/2$ (equivalently, $\alpha =3\gamma/2$) decorated by a chordal space-filling SLE$_\kappa$ between its marked points (or a concatenation of chordal space-filling SLE$_\kappa$'s in each of its beads, in the case when the wedge is thin). 
\end{lem} 

Note that the quantum wedges in Lemma~\ref{lem-past-future-ind} are thick (resp.\ thin) if $\gamma \in (0,\sqrt 2]$ (resp.\ $\gamma \in (\sqrt 2 ,2)$). 
This is consistent with the two topological phases of space-filling SLE$_\kappa$ discussed in Section~\ref{sec-space-filling-properties}.

\begin{proof}[Proof of Lemma~\ref{lem-past-future-ind}] 
	We first recall from Section~\ref{sec-space-filling-properties} that the left and right outer boundaries of $\eta((-\infty,0])$ (i.e., the interface between $\eta((-\infty,0])$ and $\eta([0,\infty))$) can be described as follows. The law of the left outer boundary $\eta_0^{\op L}$ is that of a whole-plane SLE$_{\ul\kappa}(2-\ul\kappa)$ from 0 to $\infty$ with $\ul\kappa = 16/\kappa \in (0,4)$ and conditional on $\eta_0^{\op L}$, the law of the right outer boundary $\eta_0^{\op R}$ is that of a chordal SLE$_{\ul\kappa}(-\ul\kappa/2,-\ul\kappa/2)$ from 0 to $\infty$ in $\BB C\setminus \eta_0^{\op L}$.  
	Recall from~\eqref{eq:weight-cone} that a $\gamma$-quantum cone has weight $4-\gamma^2$.  
	By the above description of $\eta_0^{\op L}$, one can apply the welding theorem for quantum cones (Theorem~\ref{thm:zipper3}) to find that law of the quantum surface $(\BB C\setminus \eta_0^{\op L} ,  \bh|_{\BB C\setminus \eta_0^{\op L}} , 0, \infty)$ is that of a quantum wedge of weight $4-\gamma^2$, in the notation~\eqref{eq:weight} (note that this wedge is thick if and only if $\gamma\leq \sqrt{8/3}$).
	By the above description of the conditional law of $ \eta_0^{\op R}$ given $\eta_0^{\op L}$, one can then apply the welding theorem for quantum wedges (Theorem~\ref{thm:zipper1} and Remark~\ref{thm:zipper2}) to the curve $\eta_0^{\op R}$ on the quantum wedge $(\BB C\setminus \eta_0^{\op L} , \bh|_{\BB C\setminus \eta_0^{\op L}})$ to obtain the lemma. 
\end{proof}

The next step is to show (using a relatively straightforward scaling argument) that the law of the pair $(\bh,\eta)$ is translation invariant in the following sense.

\begin{lem} \label{lem-mating-translate}
	For each $t\in\BB R$, the curve-decorated quantum surface $\left(\BB C , \bh(\cdot + \eta(t)) , 0,\infty , \eta(\cdot + t) - \eta(t) \right) 
	$ has the same law as $\left( \BB C ,\bh , 0,\infty,\eta \right) $, i.e., these two five-tuples agree in law modulo rotation and scaling. 
\end{lem}

We will give an outline of the proof; see~\cite[Lemma 8.3]{wedges} for a full argument.

\begin{proof}[Proof outline Lemma~\ref{lem-mating-translate}]
Let $\wt h$ be a zero-boundary GFF on $\BB D$ and let $\wt\eta$ be a space-filling SLE$_\kappa$ loop based at $1\in\bdy\BB D$, sampled independently from $\wt h$ and then parametrized by $\mu_h$-mass. 
Using Lemmas~\ref{lem:gamma} and~\ref{lem:whole-plane}, one can check that if $u \in\BB D$ is sampled from the probability measure $\mu_{\wt h} / \mu_{\wt h}(\BB D)$, then $(\wt h , \wt\eta)$ converges to $(\bh,\eta)$ when we ``zoom in near $u$", in a sense similar to the convergence in Lemmas~\ref{lem:gamma} and~\ref{lem:whole-plane}. 
If $\tau$ is the first time when $\wt\eta$ hits $u$, then (since $\wt\eta$ is parametrized by $\mu_{\wt h}$-mass) the conditional law of $ \tau$ given $(\wt h , \wt\eta)$ is uniform on $[0,\mu_{\wt h}(\BB D)]$.
Therefore, when $\ep$ is small conditional laws of $\wt\tau+\ep t $ and $\wt\tau$ given $(\wt h ,\wt\eta)$ are close in the total variation sense.
Passing this through to the local limit gives Lemma~\ref{lem-mating-translate}. 	
\end{proof}

By Lemmas~\ref{lem-past-future-ind} and~\ref{lem-mating-translate}, the conclusion of Lemma~\ref{lem-past-future-ind} remains true with $(-\infty,0]$ and $[0,\infty)$ replaced by $(-\infty,t]$ and $[t,\infty)$ for any $t\in\BB R$. 
The processes $(Z-Z_t)|_{(-\infty,t]}$ and $(Z-Z_t)|_{[t,\infty)}$ are determined by the past and future curve-decorated quantum surfaces $\left( \eta((-\infty,t]) ,  \bh|_{ \eta((-\infty,t])} , \eta|_{(-\infty,t]}   \right)$ and $\left( \eta([t,\infty) ) ,  \bh|_{ \eta([t,\infty))} , \eta|_{[t,\infty)}   \right)$, respectively. 
It follows that $Z$ has independent, stationary increments. 
Furthermore, $Z$ obeys Brownian scaling since adding a constant $C$ to the field $\bh$ scales areas by $e^{\gamma C}$ and boundary lengths by $e^{\gamma C/2}$. 

Hence $Z$ must be a Brownian motion with some covariance matrix. 
The correlation of the two coordinates of $Z$ can be computed by comparing exponents for cut times of the past/future wedges to corresponding exponents for $\pi/2$-cone times of correlated Brownian motion; see~\cite[Lemma 8.5]{wedges} (for the case when $\gamma \in [\sqrt 2 , 2)$) or~\cite{kappa8-cov} (for the case when $\gamma \in (0,\sqrt 2)$). 

The proof that $Z$ a.s.\ determines $(\bh,\eta)$ modulo rotation (given in~\cite[Section 9]{wedges}) proceeds by showing that, roughly speaking, the conditional variance of the law $(\bh,\eta)$ given $Z$ is zero. 
The proof is based on an analysis of the adjacency graph of space-filling SLE ``cells" $\eta([x-\ep,x])$ for $x\in\ep\BB Z$, which is a.s.\ determined by $Z$; see Section~\ref{sec-strong-coupling}.  The proof in~\cite{wedges} does not explicitly describe the functional which takes in $Z$ and outputs the equivalence class of $(\bh,\eta)$ modulo rotation and scaling.
However, this functional \emph{can} be made explicit using the results of~\cite{gms-tutte}; see Section~\ref{sec-tutte-conv}. 

It is implicit in the proof of Theorem~\ref{thm-mating} (and can also be deduced from Theorem~\ref{thm-mating} and Lemma~\ref{lem-past-future-ind}) that $Z$ determines $(\bh , \eta)$ in a local manner, in the following sense. 

\begin{lem} \label{lem-local-msrble}
	For each $a < b$, the curve-decorated quantum surface $(\eta([a,b]), \bh|_{\eta([a,b])} , \eta|_{[a,b]})$ is a.s.\ determined by $(Z-Z_a)|_{[a,b]}$, i.e., this triple is determined by $(Z-Z_a)|_{[a,b]}$ modulo conformal maps. 
\end{lem}

We note, however, that $(Z-Z_a)|_{[a,b]}$ does \emph{not} determine the manner in which this curve-decorated surface is embedded into $\BB C$. In particular, it does not determine $\eta([a,b])$. 

\subsubsection{Geometric encoding of $\eta$ by $Z$}
\label{sec-mating-geometric}

The Brownian motion $Z$ encodes many geometric features of $\eta$ in an explicit way. Here are some examples. Careful justifications of these examples and further examples can be found in~\cite[Section 2]{ghm-kpz} and~\cite[Section~6]{bhs-site-perc}. 
\begin{itemize}
	\item For $t \in \BB R$, the times $s\geq t$ at which $\eta$ hits the left (resp.\ right) outer boundary of $\eta((-\infty , t])$ are precisely the times when $L$ (resp.\ $R$) attains a running minimum relative to time $t$. 
	\item The times after $t$ at which $\eta$ hits a point when the left and right outer boundaries of $\eta([t,\infty))$ meet are the same as the times when $L$ and $R$ attain a simultaneous running minimum relative to time $t$. 
	\item The times when $\eta$ finishes filling in a bubble which it has cut off from its target point are precisely the \emph{$\pi/2$-cone times} of $Z$, i.e., the times $t\in\BB R$ for which there exists $t' < t$ such that $L_s \geq L_t$ and $R_s \geq R_t$ for each $s\in [t',t]$. The smallest $t'$ for which this condition holds is the time at which $\eta'$ starts filling in the corresponding bubble. Such an interval $[t',t]$ is called a \emph{$\pi/2$-cone interval} for $Z$. 
	\item When $\kappa\in(4,8)$, for a fixed $t\in \R$, consider  the set of times $s \in (t,\infty)$ which are not contained in the interior of any $\pi/2$-cone interval for $Z$ which is itself contained in $(t,\infty)$. 
	Then $\eta$ restricted to this set is the trace of a whole-plane $\SLE_\kappa(\kappa-6)$ on $(\CC,\eta(t), \infty)$.
\end{itemize} 
It follows from~\cite[Theorem 1]{shimura-cone} (see also~\cite{evans-cone}) that times for $Z$ as in the second and third examples exist if and only if $\op{corr}(L,R) > 0$, equivalently $\gamma > \sqrt 2$. This is consistent with the fact that segments of space-filling SLE$_\kappa$ are homeomorphic to closed disks for $\kappa\geq 8$ but not for $\kappa\in (4,8)$.

\subsection{Finite-volume mating of trees} 
\label{sec-mating-finite}

There are also variants of the peanosphere construction which encode space-filling SLE on a quantum sphere, disk, or wedge in terms of a pair of correlated Brownian motions with certain conditioning. 
We state the disk version; see Figure~\ref{fig-disk-bm} for an illustration. 
The sphere version is similar, and appears as~\cite[Theorem 1.1]{sphere-constructions}.
The wedge version is~\cite[Theorem 1.3]{ag-disk}. 

Let $(\BB D  ,h^{\op{D}} ,1)$ be a quantum disk with boundary length $\ell$ (see Definition~\ref{def-disk})
with a single marked boundary point. Also let $\kappa = 16/\gamma^2$ and let $\eta^{\op{D}}$ be a clockwise space-filling SLE$_{\kappa }$ loop  on $(\D, 1)$ (see Definition~\ref{def:space-filling}), 
sampled independently of $h^{\op{D}}$ and then parametrized by $\gamma$-LQG mass with respect to $h^{\op{D}}$, 
so that $\mu_{h^{\op{D}}}(\eta^{\op{D}}([0,t])) =  t$ for each $t\in [0,\mu_{h^{\op D}} (\D)]$. 
For $t\in [0,\mu_{h^{\op D}} (\D)]$, let $L_t$ be the $\gamma$-LQG length of the left outer boundary of $\eta^{\op{D}}([0,t])$ plus the $\gamma$-LQG length of $\bdy\BB D\setminus\eta^{\op{D}}([0,t])$. 
Also let $R_t$ be the $\gamma$-LQG length of the right outer boundary of $\eta^{\op D}([0,t])$. 
See Figure~\ref{fig-disk-bm} for an illustration.

\begin{figure}[ht!]
	\begin{center}
		\includegraphics[scale=.8]{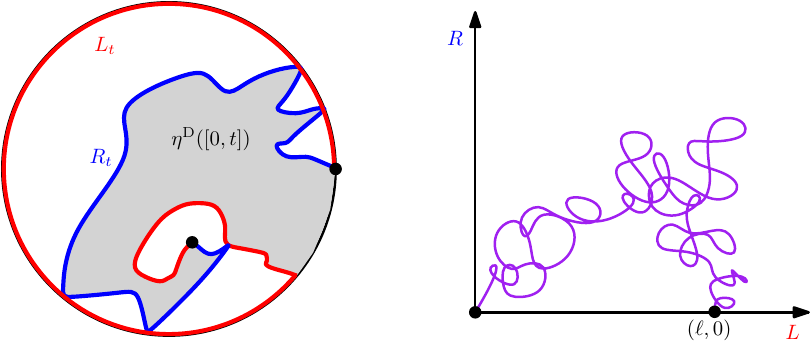} 
		\caption{\label{fig-disk-bm} \textbf{Left:} A space-filling SLE$_\kappa$ loop in $\BB D$ from 1 to 1 run up to some time $t>0$. The quantity $L_t$ (resp.\ $R_t$) is the $\gamma$-LQG length of the red (resp.\ blue) arc. In the figure, $\gamma \in (\sqrt 2 , 2)$ and $\kappa \in (4,8)$ --- the topology of the curve is simpler in the case when $\kappa \geq 8$. In this latter case, $\BB D\setminus \eta^{\op{D}}([0,t])$ is connected for each $t$. \textbf{Right:} The process $(L,R)$ is a correlated two-dimensional Brownian motion started from $(1,0)$ and conditioned to stay in $[0,\infty)^2$ until it reaches the origin. 
		}
	\end{center}
\end{figure}

The following theorem is proven in~\cite[Theorem 2.1]{sphere-constructions} for $\gamma \in (\sqrt 2 ,2)$ and in~\cite[Theorem 1.1]{ag-disk} for general $\gamma \in (0,2)$. For $\kappa=6$ it is the continuum analogue of Proposition \ref{prop:site-mating}.

\begin{thm}[Continuum mating of trees for the disk] \label{thm-mating-disk}
	The process $Z$ defined just above has the law of a Brownian motion with variances and covariances as in~\eqref{eqn-bm-cov} started from 
	$(\ell,0)$ and conditioned to stay in $[0,\infty)^2$ until hitting $(0,0)$  (this singular conditioning is made sense of in~\cite[Section~4]{ag-disk}, building on~\cite{shimura-cone}).
\end{thm}

In the case when $\gamma \in (\sqrt 2 ,2)$, Theorem~\ref{thm-mating-disk} can be deduced from Theorem~\ref{thm-mating} using the fact that the quantum surface obtained by restricting the field $\bh$ of Theorem~\ref{thm-mating} to one of the ``bubbles" filled in by $\eta$ is a quantum disk, and the restriction of $\eta$ to the time interval during which it fills in such a bubble is a space-filling SLE$_{\kappa }$ loop in the bubble.

For $\gamma \in (0,\sqrt 2]$, the curve $\eta$ does not fill in bubbles so the proof is more difficult. 
The idea is to condition on the event that the future quantum wedge $(\eta([0,\infty)) , \bh|_{\eta([0,\infty))} , 0,\infty)$ has an ``approximate pinch point" along its boundary, and show that the curve-decorated quantum surface obtained by restricting $\bh$ and $\eta$ to the ``pinched off" surface approximates the pair $(h^{\op{D}} ,\eta^{\op{D}} )$ when one takes an appropriate limit.

\subsection{Ordinary SLE$_\kappa$ for $\kappa \in (4,8)$ as a mating of trees of disks} 
\label{sec-mating-chordal}

For $\gamma\in(\sqrt 2 , 2)$ and corresponding $\kappa  =16/\gamma^2 \in (4,8)$, there is a variant of Theorem~\ref{thm-mating} which applies to chordal SLE$_\kappa$ instead of space-filling SLE$_\kappa$. To explain this, let $(\BB H , \bh , 0, \infty)$ be a $\frac{4}{\gamma} - \frac{\gamma}{2}$-quantum wedge (which is thick) with the circle average embedding, as in Definition~\ref{def:wedge}. 
Let $\eta$ be a chordal SLE$_\kappa$ ($\kappa = 16/\gamma^2$) from 0 to $\infty$ in $\BB H$, sampled independently from $\bh$.
The natural quantum time parametrization of $\eta$ is the so-called \emph{quantum natural time}, which can be thought of as the ``quantum Minkowski content" of $\eta$. See~\cite[Definition 6.23]{wedges} for a precise definition. 

Henceforth assume that $\eta$ is parametrized by quantum natural time and for $t\geq 0$, let $L_t$ (resp.\ $R_t$) be the net change in the $\gamma$-LQG length of the left (resp.\ right) outer boundary of $\eta([0,t])$ relative to time 0; see Figure~\ref{fig-bdy-process-sle}. 
The processes $L$ and $R$  are c\'adl\'ag, with downward jumps corresponding to the times when $\eta$ disconnects a bubble from $\infty$ (the magnitude of the downward jump is the $\gamma$-LQG length of the boundary of the bubble).

\begin{figure}[ht!]
	\begin{center}
		\includegraphics[scale=.8]{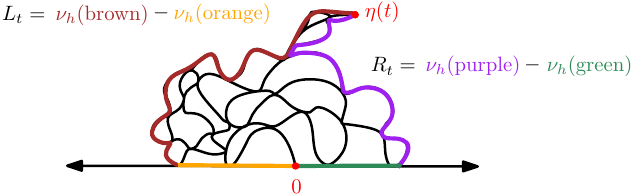} 
		\caption{\label{fig-bdy-process-sle} The definition of the left/right boundary length process for SLE$_\kappa$ on a quantum wedge. The pair $(L,R)$ evolves as a pair of independent $\kappa/4$-stable processes with only downward jumps. 
		}
	\end{center}
\end{figure}

The following theorem is a combination of~\cite[Theorem 1.18 and Corollary 1.19]{wedges}. 

\begin{thm}[Continuum mating of trees of disks] \label{thm-mating-chordal}
	$(L,R)$ is a pair of $\kappa/4$-stable processes with only downward jumps, i.e., the L\'evy measure for each of $L$ and $R$ is $c |x|^{-\kappa/4-1}\BB 1_{(x  < 0)} \, dx$ for a non-explicit normalizing constant $c$. 
	Let $t\geq 0$ and condition on $(L,R)|_{[0,t]}$. The conditional law of the quantum surface obtained by restricting $\bh$ to the unbounded connected components of $\BB H\setminus\eta([0,t])$, with marked points $\eta(t)$ and $\infty$, is that of a $(\frac{4}{\gamma} -\frac{\gamma}{2})$-quantum wedge. 
	The conditional laws of the quantum surfaces obtained by restricting $\bh$ to the bounded connected component of $\BB H\setminus\eta([0,t])$, each marked by the point where $\eta$ finishes tracing its boundary, are independent singly marked quantum disks with boundary lengths equal to the magnitudes of the corresponding downward jump of $\eta$. 
	All of these surfaces are conditionally independent given $(L,R)|_{[0,t]}$. 
	Finally, $(L,R)|_{[0,t]}$ together with the above collection of quantum surfaces a.s.\ determines $(\bh,\eta)$. 
\end{thm}

Theorem~\ref{thm-mating-chordal} is a generalization of Theorem~\ref{thm:zipper0} and Proposition~\ref{prop:zipper} to the case when $\kappa \in (4,8)$. 
The theorem implies that the curve-decorated topological space $(\BB H ,\eta)$ is the homeomorphic to the space obtained by gluing together a certain pair of trees of disks constructed from $L$ and $R$. See~\cite[Figure 1.19]{wedges}.  These trees of disks are closely related to the stable looptrees introduced in \cite{curien-kortchemski-looptree-def}.
 
As in Theorem~\ref{thm-mating}, the proof of Theorem~\ref{thm-mating-chordal} does not give an explicit way of recovering $(\bh,\eta)$ from $(L,R)|_{[0,t]}$ and the above quantum surfaces (equivalently, from the trees of disks). A strengthened form of the measurability statement in Theorem~\ref{thm-mating-chordal} is proven in~\cite{mmq-welding}, which implies that there is only one way to weld together two trees of disks \emph{which produces an SLE curve}. However, this strengthened version is still weaker than the one in Theorem~\ref{thm:zipper0} based on conformal removability. Indeed, it is still an open question to determine whether $\SLE_\kappa$ is conformally removable for $\kappa\in (4,8)$. If we knew that this were the case then we could say that there is only one possible way of gluing together the trees of disks, regardless of the type of curve produced.

In the case when $\gamma =\sqrt{8/3}$ and $\kappa = 6$, Theorem~\ref{thm-mating-chordal} is the continuum analog of the so-called \emph{peeling process} for a uniform infinite planar map with the topology of the half-plane decorated by a percolation configuration.
In this analogy, the left/right boundary length processes $L$ and $R$ correspond to so-called \emph{horodistance} processes.
See~\cite{angel-peeling,angel-curien-uihpq-perc,richier-perc} for more on this peeling process and~\cite{curien-glimpse} for a similar peeling process using SLE$_6$ instead of a percolation interface. 
The relationship between Theorem~\ref{thm-mating-chordal} and peeling processes is used in~\cite{gwynne-miller-char,gwynne-miller-perc} to prove the convergence of percolation on half-planar maps to SLE$_6$ on $\sqrt{8/3}$-LQG.

\section{Applications}
\label{sec-applications}

\subsection{Convergence of random planar maps: peanosphere convergence and beyond}\label{subsec:rmp}
In this section we explain how mating-of-trees theory allows one to show that certain random planar maps decorated by statistical mechanics models converge to SLE-decorated LQG surfaces in a certain topology. 
For expository convenience, we first explain this in the case of an infinite-volume spanning-tree weighted planar map $(M^\infty, e^\infty , T^\infty)$, as discussed at the end of Section~\ref{subsec:UST}. 
Let $\mcl Z^\infty$ be the two-sided simple random walk on $\BB Z^2$ associated with $(M^\infty, e^\infty , T^\infty)$ via the infinite-volume version of the Mullin bijection.
Then $\mcl Z^\infty$ converges in law in the scaling limit to a standard planar Brownian motion $Z$.
By Theorem~\ref{thm-mating}, $Z$ encodes a $\sqrt 2$-quantum cone $(\BB C,h,0,\infty)$ decorated by an SLE$_8$ curve $\eta$ from $\infty$ to $\infty$. 
The process $\mcl Z^\infty$ (resp.\ $Z$) a.s.\ determines $(M^\infty, e^\infty , T^\infty)$ (resp.\ $(h,\eta)$).
Hence, we can interpret the convergence of random walk to Brownian motion as saying that the infinite spanning-tree weighted planar map converges in law to an SLE$_8$-decorated $\sqrt2$-quantum cone with respect to a certain topology, namely the one where two decorated surfaces are close if their corresponding encoding functions are close. 
We call this type of convergence \emph{convergence in the peanosphere sense}.  

Besides the two bijections discussed in Section~\ref{sec-bijections},  other decorated random planar maps models known to have a mating-of-trees encoding (and to converge to SLE-decorated LQG in the peanosphere sense) include the following. In each case, the corresponding value of $\gamma$ is listed in the parentheses. 
\begin{enumerate}
	\item Planar maps decorated by an instance of the critical Fortuin-Kasteleyn (FK) cluster model~\cite{shef-burger} ($\gamma$ ranges from $\sqrt 2$ to 2, depending on the FK parameter $q \in (0,4)$);
	\item Planar maps decorated by active spanning trees~\cite{gkmw-burger} ($\gamma$ ranges from $0$ to $\sqrt 2$, depending on the activity parameter);
	\item Bipolar-oriented planar maps~\cite{kmsw-bipolar} ($\gamma=\sqrt{4/3}$ for uniform bipolar oriented maps, $\gamma$ ranges from $0$ to $\sqrt 2$ for biased bipolar-oriented maps);
	\item Schnyder wood-decorated simple triangulations \cite{lsw-schnyder-wood} ($\gamma=1$).
\end{enumerate}
We omit the precise definitions of these models and their encodings, which  can be found in the corresponding references. In all of the encodings, it is possible to start from the given model to construct a (non-uniform) spanning tree on the planar maps. The planar map is then encoded by the walk on $\BB Z^2$ whose coordinates are (minor variants of) the contour functions of this tree and its corresponding dual tree. 

Once one has a mating-of-trees encoding for a certain random planar map model, proving convergence in the peanosphere sense for these models is a matter of checking that the  2D random walk converges to the 2D Brownian motion whose two coordinates have the correct correlation. For planar maps decorated by the critical FK model or by an active spanning tree, the encoding (the so-called \emph{hamburger-cheeseburger bijection}) is closely related to the Mullin bijection, but the random walks involved are \emph{non-Markovian}.
Hence it requires some non-trivial work to prove the Brownian motion convergence (this convergence is proven in~\cite{shef-burger} and~\cite{gkmw-burger}, respectively). 
For percolated triangulations and for planar maps decorated by bipolar orientations or Schnyder woods, the random walks involved have independent increments, therefore the peanosphere convergence is trivial once the mating-of-trees bijection is known. The fact that the walk has independent increments will also play an important role in Section~\ref{sec-strong-coupling}.

These peanosphere convergence results constitute the \emph{first} link between the decorated planar map models and SLE/LQG.  
In the case of the critical FK cluster model, such a link is expected because on a regular lattice the critical FK cluster boundaries are supposed to converge in law to a conformal loop ensemble (CLE$_\kappa$) with $\kappa \in (4,8)$~\cite{shef-cle}.
In other cases, realizing such a link is conceptually important and sheds lots of light on these models. 
For example, in the bipolar orientation and Schnyder wood cases, the corresponding models on a regular lattice are the 6-vertex on $\BB Z^2$ and the 20-vertex model on the triangular lattice, respectively. 
The peanosphere convergence results mentioned above motivate the conjecture that the Peano curve associated with the 6-vertex (resp.\ 20-vertex) model converges in the scaling limit to SLE$_{12}$ (resp.\ SLE$_{16}$). Here we note that $12 = 16/(\sqrt{4/3})^2$ and $16 = 16/(1)^2$. See \cite{kmsw-6vertex} for some simulations supporting this conjecture in the case of the bipolar orientation.
Peanosphere convergence results are also the first results in which $\SLE_\kappa$ curves with $\kappa>8$ arise as the scaling limit of natural combinatorial models.

\subsubsection{Improvements on peanosphere convergence}

In many cases, once peanosphere convergence is established, it is possible to use model-specific arguments to obtain the convergence of additional natural observables.
In the rest of this subsection we review some results of this type. See Sections~\ref{sec-strong-coupling} and~\ref{sec-kappa6} for additional results about planar maps which are proven building on peanosphere convergence.

The first works in this direction are \cite{gms-burger-cone,gms-burger-local,gms-burger-finite,blr-exponents,chen-fk} concerning random planar maps decorated by the critical FK cluster model. In this model, each planar map is decorated with a subset of edges  which are declared to be \emph{open}. Edges which are not open are called \emph{closed}.  The natural observable for this model is the interfaces between open and closed clusters.
Based on the mating-of-trees encoding in~\cite{shef-burger} (a.k.a.\ the hamburger-cheeseburger bijection), 
it is proved in~\cite{gms-burger-cone,gms-burger-local,gms-burger-finite} that many quantities associated with the FK interfaces converge to their counterparts for $\gamma$-LQG coupled with $\CLE_\kappa$.  
These quantities include the areas and  boundary lengths of the complementary connected components\footnote{In the discrete setting, all the loops are simple but the clusters have ``fjords'' which lead to separate complementary components in the scaling limit. 
} of the macroscopic loops and the adjacency structure of these loops. The key observation here is that these quantities can be expressed in terms of a special set of times associated to a positively correlated two-dimensional Brownian motion called \emph{cone times}~\cite{evans-cone,shimura-cone}.
The work~\cite{blr-exponents} focuses on the tail exponent of quantities similar  to those considered in \cite{gms-burger-cone} and verifies the KPZ relation (see~Section~\ref{subsec:app-sle}). The work~\cite{chen-fk} studies the infinite-volume limit of the model.

For the site-percolated loopless triangulation considered in Section~\ref{subsec:site}, the loop ensemble is also the natural observable to look at. 
Given the simplicity of the bijection, all of the known scaling limit results for planar maps decorated by the critical FK cluster model can be proved in this case with simpler arguments. This is carried out in \cite{bhs-site-perc}. Moreover,
\cite{bhs-site-perc} also proves the scaling limit of the so-called \emph{pivotal measure}, which is the counting measure on  macroscopic percolation pivotal  points. 
This is an important step in proving the convergence of the uniform triangulation under the so-called \emph{Cardy embedding}. See Section~\ref{sec-kappa6} for more details. 

Both bipolar orientations and Schnyder woods are random orientation models, where each edge comes with a direction. Ignoring boundary conditions, in a  bipolar orientation, every vertex has exactly two outgoing edges, while in a Schnyder wood every vertex has exactly three outgoing edges. In the latter case, Euler's formula yields that the underlying map has to be a simple  triangulation.
Both of the two models exhibit richer structure than just the spanning trees involved in their corresponding mating-of-trees bijections. For example, the Schnyder wood-decorated triangulation can be decomposed into three spanning trees.
These spanning trees give rise to three different encoding walks for the same random planar map, say $\mcl Z_1,\mcl Z_2,\mcl Z_3$. 
In \cite{lsw-schnyder-wood}, it is shown that $\mcl Z_1,\mcl Z_2,\mcl Z_3$ converge jointly in law in the scaling limit to three coupled planar Brownian motions which encode three coupled $\SLE_{16}$ curves on the \emph{same} $\gamma=1$-LQG surface in the setting of Theorem~\ref{thm-mating}.

Moreover, the coupling of the three $\SLE_{16}$ curves can be naturally described  in  terms of imaginary geometry. 
Indeed, fix $z\in\BB C$ and consider the left and right outer boundaries of each $\SLE_{16}$ stopped upon hitting $z$.
This gives six simple SLE$_1$-type curves emanating from $z$.
The joint law of these six curves can be described as six imaginary geometry flow lines (in the sense of~\cite[Theorem 1.1]{ig4}) of the same whole-plane GFF $h$ with angles spaced apart by $\pi/3$. 
A similar result was proven in \cite{ghs-bipolar} for uniform bipolar oriented triangulations, in which case the limiting object is a pair of coupled $\SLE_{12}$ curves on a $\sqrt{4/3}$-LQG surface.

Schnyder woods were first introduced by Schnyder~\cite{schnyder-poset,schnyder-embedding} to give an  efficient algorithm for embedding triangulations in the grid in such a way that edges are non-crossing straight line segments.  
Based on the joint convergence of the three trees above, it is proved in \cite{lsw-schnyder-wood} that the Schnyder wood-decorated triangulation has a scaling limit under the Schnyder embedding which can be described in terms of SLE and LQG.
We emphasize that the Schnyder embedding is not close to being ``discrete conformal". In particular, the embedded triangulation has macroscopic faces and the scaling limit of the counting measure on vertices of the embedded triangulation is \emph{not} given by the LQG area measure.

\subsection{Mated-CRT maps and strong coupling} 
\label{sec-strong-coupling}
Mated-CRT maps are discretizations of the definition of the infinite-volume peanosphere described in Figure~\ref{fig-peanosphere}, which produce a random planar map instead of a topological space. 
Let $\gamma \in (0,2)$ and let $Z = (L,R)$ be a pair of correlated, two-sided Brownian motions with correlation $-\cos(\pi\gamma^4/4)$, as in Theorem~\ref{thm-mating}. For $\ep > 0$, the \emph{mated-CRT map} with cell size $\ep$ associated with $Z$ is the graph whose vertex set is $\ep\BB Z$, with two vertices $x_1,x_2 \in \ep\BB Z$ connected by an edge if and only if either
\allb \label{eqn-inf-adjacency}
&\max\left\{ \inf_{t\in [x_1-\ep , x_1]} L_t ,\, \inf_{t\in [x_2-\ep , x_2]} L_t \right\} \leq \inf_{t\in [x_1, x_2-\ep]} L_t \quad \text{or}  \notag \\
&\qquad\qquad \max\left\{ \inf_{t\in [x_1-\ep , x_1]} R_t ,\, \inf_{t\in [x_2-\ep , x_2]} R_t \right\} \leq \inf_{t\in [x_1, x_2-\ep]} R_t .
\alle
There are two edges connecting $x_1$ and $x_2$ if $|x_1-x_2| > \ep$ and the conditions for both $L$ and $R$ in~\eqref{eqn-inf-adjacency} holds.
We note that by Brownian scaling, the law of the mated-CRT map does not depend on $\ep$, but it is convenient to view the maps with different values of $\ep$ as being coupled with the same Brownian motion. 
See Figure~\ref{fig-mated-crt-map}, left, for a geometric interpretation of~\eqref{eqn-inf-adjacency}. 

As explained in Figure~\ref{fig-mated-crt-map}, the mated-CRT map possesses a canonical planar map structure. In fact, it is a.s.\ a triangulation: see~\cite[Figure 1]{gms-harmonic}. This can also be seen from the SLE/LQG description of the mated-CRT map, as described below. Mated-CRT maps are used implicitly (but not referred to as such) in~\cite[Section 9]{wedges} and play a fundamental role in the papers~\cite{ghs-dist-exponent,ghs-map-dist,gms-tutte,gms-harmonic,gm-spec-dim,gh-displacement,dg-lqg-dim,gp-dla}. 

\begin{figure}[t!]
	\begin{center}
		\includegraphics[scale=.65]{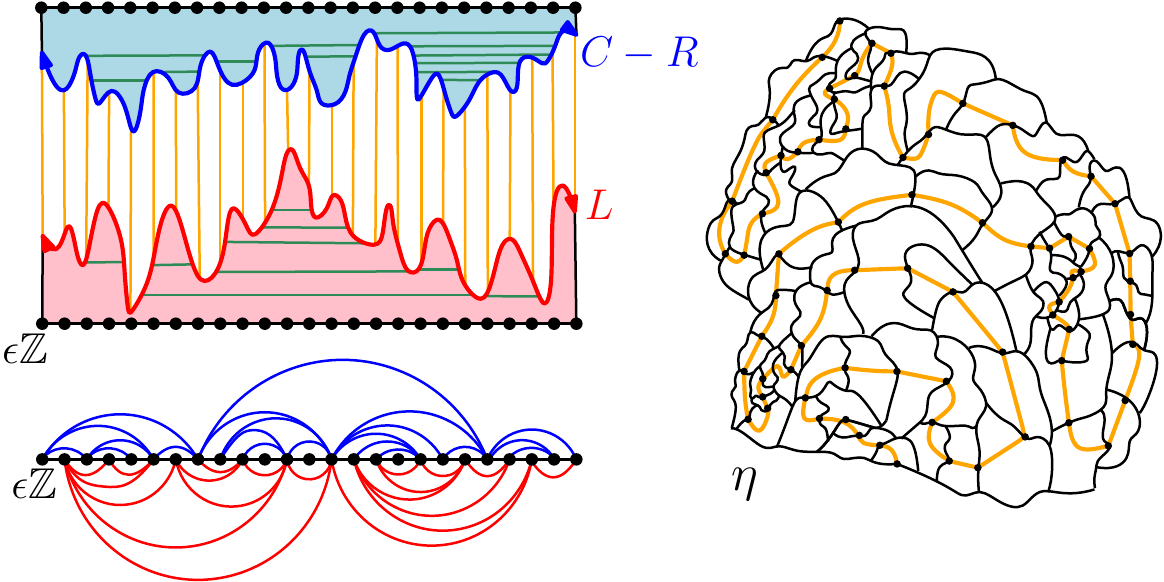}
		\vspace{-0.01\textheight}
		\caption{\textbf{Top Left:} To construct the mated-CRT map $\mcl G^\ep$ geometrically, one can draw the graph of $L$ (red) and the graph of $C-R$ (blue) for some large constant $C > 0$ chosen so that the parts of the graphs over some time interval of interest do not intersect. One then divides the region between the graphs into vertical strips (boundaries shown in orange) and identifies each strip with the horizontal coordinate $x\in \ep \BB Z$ of its rightmost point. Vertices $x_1,x_2\in \ep \BB Z$ are connected by an edge if and only if the corresponding strips are connected by a horizontal line segment which lies under the graph of $L$ or above the graph of $C-R$. One such segment is shown in green in the figure for each pair of vertices for which this latter condition holds.
		\textbf{Bottom left:} One can define a canonical planar map structure on the mated-CRT map by (a) connecting $x-\ep$ and $x$ by a line segment for each $x\in\ep\BB Z$ and then (b) drawing an arc above (resp.\ below) the real line joining $x_1$ and $x_2$ whenever $|x_1-x_2| >\ep$ and the corresponding strips are joined by a horizontal line segment above the graph of $C-R$ (resp.\ below the graph of $L$).  
			\textbf{Right:} The mated-CRT map can be realized as the adjacency graph of \emph{cells} $\eta([x-\ep,x])$ for $x\in \ep\BB Z$, where $\eta$ is a space-filling SLE$_\kappa$ for $\kappa=16/\gamma^2$ parametrized by $\gamma$-LQG mass with respect to an independent $\gamma$-LQG surface. Here, the cells are outlined in black and the order in which they are hit by the curve is shown in orange. 
			Note that the different figures do not correspond to the same mated-CRT map realization. 
		}\label{fig-mated-crt-map}
	\end{center}
	\vspace{-1em}
\end{figure}

Mated-CRT maps are a particularly natural class of random planar maps to study since they provide a bridge between the main discrete and continuum models involved in the theory of random planar maps. Consequently, such maps are an extremely useful tool for analyzing a much larger class of random planar maps, as we now explain. 
\medskip

\noindent\textbf{Discrete connection.}
The definition of the mated-CRT map is a semi-continuous analog of the mating-of-trees bijections discussed in Section~\ref{sec-bijections}. 
Due to the convergence of random walk to Brownian motion (i.e., peanosphere convergence) one can view the mated-CRT map as a coarse-grained approximation of these other planar map models. 
In fact, for each of the mating-of-trees bijections discussed there, the condition for two vertices of the map to be connected by an edge is a discrete analog of the mated-CRT map adjacency condition~\eqref{eqn-inf-adjacency}. This is explained carefully for each of the models in~\cite[Section 3]{ghs-map-dist}.  

Suppose now that $(\mcl M , \mcl T)$ is a decorated random planar map which can be encoded by a mating-of-trees bijection wherein the encoding walk $\mcl Z : \BB Z\rta\BB Z^2$ has i.i.d.\ increments with an exponential tail. Let $Z$ be the Brownian motion used to define the mated-CRT map, with $\gamma$ chosen so that the correlation of the coordinates of $\mcl Z$ is $-\cos(\pi\gamma^2/4)$ (i.e., the same as the correlation of the two coordinates of $Z$). 
Since scaling $L$ and $R$ does not affect the definition of the mated-CRT map, we can assume without loss of generality that the variances of the corresponding coordinates of $Z$ and $\mcl Z$ agree. 
The strong coupling result of Zaitsev~\cite{zaitsev-kmt} (which is the multivariate analog of the KMT coupling theorem~\cite{kmt}) says that one can couple $\mcl Z$ and $Z$ in such a way that
\eqb \label{eqn-strong-coupling}
\sup_{t\in [-n,n] \cap \BB Z} |  Z_t - \mcl Z_t| = O(\log n) ,\quad 
\text{with extremely high probability.}
\eqe  
This leads to a coupling of $(\mcl M , \mcl T)$ with the mated-CRT maps associated with $Z$. 
Combining this with an analysis of the geometry of the mating-of-trees bijection, one can prove the following informal theorem. See~\cite[Theorems 1.5 and 1.9]{ghs-map-dist} and~\cite[Lemma 4.3]{gm-spec-dim} for precise statements.

\begin{thm}[Strong coupling for mated-CRT maps, informal statement] \label{thm-strong-coupling}
	Let $\mcl G$ be the mated-CRT map with cell size $\ep =1$. For $n\in\BB N$, let $\mcl G_n$ be the subgraph of $\mcl G$ induced by $[-n,n] \cap \BB Z$ and let $\mcl M_n$ be the submap of $\mcl M$ corresponding to $\mcl Z|_{[-n,n] \cap \BB Z}$ (the precise definition depends on the particular bijection). 
	There are universal constants $p,q > 0$ such that if we couple $\mcl M$ and $\mcl G$ so that~\eqref{eqn-strong-coupling} holds, then there is a correspondence between vertices of $\mcl G$ and vertices of $\mcl M$ such that with extremely high probability, the following is true. 
	\begin{itemize}
		\item If $x,y$ are vertices of $\mcl G_n$ and $x',y'$ are the corresponding vertices of $\mcl M_n$, then the $\mcl G_n$-graph distance between $x$ and $y$ and the $\mcl M_n$-graph distance between $x'$ and $y'$ differ by a factor of at most $(\log n)^p$.
		\item If $f$ is a function from the vertex set of $\mcl G_n$ to $\BB R$ and $f'$ is the corresponding function on $\mcl M_n$, then the discrete Dirichlet energies of $f$ and $f'$ differ by a factor of at most $(\log n)^q$. 
	\end{itemize}
\end{thm}

Theorem~\ref{thm-strong-coupling} allows us to reduce many statements about random planar maps which can be encoded by a mating-of-trees bijection wherein the encoding walk has i.i.d.\ increments to  statements about the mated-CRT map, which is in some cases much easier to analyze due to the continuum theory discussed below. The planar maps $\mcl M$ for which Theorem~\ref{thm-strong-coupling} is currently known to apply include the UIPT (Section~\ref{subsec:site}) and infinite-volume limits of planar maps weighted by the number of spanning trees (Section~\ref{subsec:UST}), bipolar orientations~\cite{kmsw-bipolar}, or Schnyder woods~\cite{lsw-schnyder-wood} they admit. 
For any walk $\mcl Z : \BB Z^2 \rta \BB Z$ with i.i.d.\ increments having an exponential tail, one can construct a planar map to which Theorem~\ref{thm-strong-coupling} applies via direct discretization of the mated-CRT map definition~\eqref{eqn-inf-adjacency}, but such a map is not always combinatorially natural. 
Theorem~\ref{thm-strong-coupling} is \emph{not} currently known to apply in the setting of the hamburger-cheeseburger bijection of~\cite{shef-burger,gkmw-burger} since in this case the encoding walk does not have i.i.d.\ increments. 
\medskip

\noindent\textbf{Continuum connection.}
The connection between the mated-CRT map and continuum theory is clear from Theorem~\ref{thm-mating}. 
Indeed, if we let $\eta$ be the space-filling SLE$_\kappa$ parametrized by $\gamma$-LQG mass corresponding to $Z$ in the setting of Theorem~\ref{thm-mating}, then the adjacency condition~\eqref{eqn-inf-adjacency} is equivalent to the condition that the $\ep$-LQG mass \emph{cells} $\eta([x_1-\ep,x_1])$ and $\eta([x_2-\ep,x_2])$ intersect along a non-trivial boundary arc. This is proven rigorously in~\cite[Section 8.2]{wedges}.
Hence the mated-CRT map is isomorphic to the adjacency graph of cells $\eta([x-\ep,x])$ for $x\in \ep\BB Z$ via the isomorphism $x\mapsto \eta([x-\ep,x])$. 

\begin{remark}
	For $\kappa \geq 8$, the cells $\eta([x_1-\ep,x_1])$ and $\eta([x_2-\ep,x_2])$ intersect if and only if they share a non-trivial boundary arc, and each such cell is homeomorphic to a closed disk. 
	For $\kappa \in (4,8)$, it is possible for two space-filling SLE cells to intersect each other in a fractal set, but not share a non-trivial boundary arc. In this case, the cells are not simply connected and they have ``cut points", so their interiors are not connected. See Figure~\ref{fig-peanosphere-bm-def}, right. 
\end{remark}

The above isomorphism gives a natural embedding of the mated-CRT map into the complex plane by sending $x\in\ep\BB Z$ to $\eta(x)$. 
This embedding is \emph{not} an explicit functional of the mated-CRT map $\mcl G^\ep$ (see, however, Section~\ref{sec-tutte-conv}) but it has many nice properties. 
For example, due to Proposition~\ref{prop-sle-ball}, the cells of the mated-CRT map are ``roughly spherical" in the sense that they are contained between Euclidean balls whose radii are comparable up to an $o(1)$ error in the exponent. 
Furthermore, one has bounds for the maximum and minimum Euclidean sizes of the cells which come from Proposition~\ref{prop-sle-ball} and basic estimates for the LQG measure: 

\begin{prop}[Cell diameter estimates] \label{prop-cell-diam}
	Fix a bounded open set $U $ with $\ol U\subset \BB D$ and a small parameter $\zeta\in (0,1)$. There exists $\alpha = \alpha(\zeta) > 0$ such that with probability $1-O_\ep(\ep^\alpha)$, 
	\eqb
	\ep^{\frac{2}{(2-\gamma)^2} + \zeta } \leq \op{diam} \eta([x-\ep,x]) \leq \ep^{\frac{2}{(2+\gamma)^2} - \zeta} ,\quad\forall x\in \ep\BB Z\: \text{with} \: \eta([x-\ep,x]) \cap U\not=\emptyset. 
	\eqe
\end{prop}

The upper bound in Proposition~\ref{prop-cell-diam} is proven~\cite[Lemma 2.7]{gms-harmonic}. The lower bound is immediate from an elementary upper bound for the LQG areas of Euclidean balls and the fact that the LQG area of $\eta([x-\ep,x])$ is $\ep$.  

The SLE/LQG embedding of the mated-CRT map allows one to prove many properties of this map. 
The reason for this is as follows. A common way to study random planar maps is to embed the map into $\BB C$ in some way, then consider how the embedded map interacts with paths and functions in $\BB C$. A number of papers have used this strategy with the circle packing embedding (see~\cite{stephenson-circle-packing} for an introduction); see, e.g.,~\cite{benjamini-schramm-topology,gn-recurrence,abgn-bdy,gill-rohde-type,angel-hyperbolic,cg-liouville,lee-conformal-growth,lee-uniformizing}. 
Since mated-CRT map cells are ``roughly spherical", one can apply similar techniques to these papers with the SLE/LQG embedding of the mated-CRT map in place of the circle packing. 
However, thanks to SLE/LQG theory, the embedding of the mated-CRT map is in some ways much better understood than the circle packing.
For example, for most natural combinatorial random planar map models (e.g., uniform maps and maps weighted by statistical mechanics models) it is a major open problem to prove an analog of Proposition~\ref{prop-cell-diam} with the circles of the circle packing (appropriately normalized) in place of SLE cells. 

One can use the above continuum techniques to prove statements about the mated-CRT map, then transfer to any other random planar maps for which Theorem~\ref{thm-strong-coupling} applies. 
Some of the results which have been proven using this approach are as follows.
\begin{itemize}
	\item (Graph distances) There is an exponent $d_\gamma > 2$, depending only on $\gamma$, which describes a number of natural quantities related to graph distances in mated-CRT maps and planar maps to which Theorem~\ref{thm-strong-coupling} applies. For example, the graph distance ball of radius $r$ in any such map typically has of order $r^{d_\gamma + o_r(1)}$ vertices~\cite[Theorem 1.6]{dg-lqg-dim}. One also has reasonably sharp upper and lower bounds for $d_\gamma$ which match only for $\gamma=\sqrt{8/3}$, in which case $d_{\sqrt{8/3}}=4$~\cite{ghs-map-dist,dzz-heat-kernel,dg-lqg-dim,gp-lfpp-bounds,ang-discrete-lfpp}.\footnote{To be more precise, \cite{dzz-heat-kernel} proves the existence of an exponent $d_\gamma$ for a continuum discretization of LQG distances (called \emph{Liouville graph distance}), \cite{dg-lqg-dim} proves that this same exponent describes distances in the mated-CRT map and proves bounds for it, and~\cite{ghs-map-dist} compares distances in mated-CRT maps and other maps. See also~\cite{gp-lfpp-bounds,ang-discrete-lfpp} for improved bounds for $d_\gamma$ (and other related quantities) and~\cite{ghs-dist-exponent} for additional bounds for distances in mated-CRT maps. } This provides the first non-trivial bounds for distances in planar maps weighted by spanning trees, bipolar orientations, or Schnyder woods. It is shown in~\cite[Corollary 1.7]{gp-kpz} that $d_\gamma$ coincides with the Hausdorff dimension of the continuum LQG metric as constructed in~\cite{dddf-lfpp,gm-uniqueness}.
	\item (Random walk) For mated-CRT maps and maps for which Theorem~\ref{thm-strong-coupling} applies, the probability that simple random walk returns to its starting point after $n$ steps is of order $n^{-1 + o_n(1)}$ (i.e., the \emph{spectral dimension} is 2) and the typical distance traveled by the simple random walk after $n$ steps is $n^{1/d_\gamma + o_n(1)}$. This is proven in~\cite{gm-spec-dim,gh-displacement}, building on estimates for the mated-CRT map from~\cite{gms-harmonic} (the lower bound for the return probability was first proven by Lee~\cite{lee-conformal-growth}). 
	In the case of the UIPT, where $d_\gamma=d_{\sqrt{8/3}}=4$, this confirms a conjecture of Benjamini and Curien \cite{benjamini-curien-uipq-walk}.
	\item (External DLA) On the infinite spanning-tree weighted random planar map, the graph-distance diameter of an external diffusion limited aggregation (DLA) cluster run for $n$ steps is typically of order $n^{\frac2{d_{\sqrt 2}} + o_n(1)}$~\cite{gp-dla}. 
\end{itemize}

\subsection{The Tutte embedding of the mated-CRT map converges to LQG} \label{sec-tutte-conv}
There are various canonical ways of embedding a random planar map into $\BB C$ (i.e., drawing its vertices and edges in the plane so that no edges cross).  
For several embeddings one can argue heuristically that they approximate a conformal map.
Examples include circle packing~\cite{stephenson-circle-packing}, Riemann uniformization, and the Tutte embedding (which we define just below). 
It is expected that if we embed a random planar map in the $\gamma$-LQG universality class in one of these ``discrete conformal" ways, then the counting measure on the vertices of the embedded map, appropriately rescaled, should converge in law to the $\gamma$-LQG measure. Moreover, various curves associated with statistical mechanics models on the map should converge to SLE.  (Recall Section~\ref{subsec:scaling}.)
Here we describe the first rigorous proof of this type of convergence, in the case of the Tutte embedding of the mated-CRT map.

We defined the mated-CRT map with the topology of the whole plane in Section~\ref{sec-strong-coupling}, but the Tutte embedding is somewhat easier to define if we work with the disk topology. 
To define the mated-CRT map with the disk topology, we let $Z = (L,R)$ be a pair of correlated Brownian motions with correlation $-\cos(\pi\gamma^2/4)$ started from $(1,0)$ and conditioned to stay in the first quadrant until time 1 and satisfy $Z_1 = (0,0)$. In the setting of Theorem~\ref{thm-mating-disk}, this conditioned Brownian motion encodes a unit area, unit boundary length quantum disk $(\BB D , h^{\op{D}} , 1)$ decorated by a space-filling SLE$_\kappa$ loop $\eta^{\op{D}}$, parameterized by $\mu_{h^{\op{D}}}$-mass. 
For $n\in\BB N$, we define the mated-CRT map $\mcl G^{1/n}$ associated with $Z$ to be the random planar map with vertex set $ (\frac{1}{n}\BB Z^2) \cap (0,1]$, with two vertices $x_1,x_2$ connected by an edge if and only if~\eqref{eqn-inf-adjacency} holds (with $\ep = 1/n$). 
Exactly as in the whole-plane case, Theorem~\ref{thm-mating-disk} implies that $\mcl G^{1/n}$ can equivalently be defined as the adjacency graph of $1/n$-LQG mass cells $\eta^{\op{D}}([x-1/n,x])$. 
We define the \emph{boundary} of $\mcl G^{1/n}$ by
\eqb \label{eqn-mated-crt-map-bdy}
\bdy\mcl G^{1/n} := \left\{ y \in  (\frac{1}{n} \BB Z) \cap (0,1] : \inf_{t \in [y-1/n,y]} L_t \leq \inf_{t\in [y,1]} L_t \right\} ,
\eqe
which is the same as the set of $y\in (\frac{1}{n} \BB Z) \cap (0,1]$ for which the cell $\eta^{\op{D}}([y-1/n,y])$ shares a non-trivial arc with $\bdy\BB D$. 

To construct the Tutte embedding of the mated-CRT map $\mcl G^{1/n}$, first sample $\BB t$ uniformly at random from $[0,1]$ (independently from $Z$) and for $n\in\BB N$, let $\BB x^n \in   (\frac{1}{n} \BB Z) \cap (0,1]$ be the vertex of $\mcl G^{1/n}$ with $\BB t \in [\BB x^n - 1/n , \BB x^n]$.  
We will define our embedding so that $\BB x^n$ is mapped (approximately) to $0\in\BB D$. 
Number the vertices of $\bdy\mcl G^{1/n}$ (in increasing order) as $\{y_1,\dots,y_k\}$. For $j = 1,\dots,k$ let $\frk p(y_j)$ be the conditional probability given $\mcl G^{1/n}$ that a simple random walk started from $\BB x^n$ first hits the boundary at a vertex in the boundary arc $\{y_1,\dots,y_j\}$.
The boundary vertices $y_1, y_2, \ldots, y_k$ are then mapped in counterclockwise order around the complex unit circle via $y_j \mapsto e^{2\pi i \frk p(y_j)}$. This makes it so that the hitting probability of the random walk started from $\BB x^n$ approximates the uniform measure on the unit circle. One then maps the interior vertices of $\mcl G^{1/n}$ into the unit disk via the discrete harmonic extension of the boundary values.
Equivalently, we require that the position of each interior vertex is equal to the average of the positions of its neighbors.
 This embedding of $\mcl G^{1/n}$ is called the \notion{Tutte embedding} of the mated-CRT map centered at $\BB x^n$.
The following is~\cite[Theorem 1.1]{gms-tutte}.

\begin{thm}[\cite{gms-tutte}] \label{thm-tutte-conv}
	Fix $\gamma \in (0,2)$ and define the conditioned Brownian motion $Z$, the associated quantum disk $(\BB D  ,h^{\op{D}} , 1)$ and space-filling SLE$_\kappa$ curve $\eta^{\op{D}}$, and the mated-CRT maps with the disk topology $\mcl G^{1/n}$ for $n\in\BB N$ as above. 
	For $n\in\BB N$, let $\mu^n$ be the random measure on $\ol{\BB D}$ which assigns mass $1/n$ to each of the $n$ points in the Tutte embedding of the mated-CRT map centered at $\BB x^n$. We have the following convergence in probability as $n \to \infty$. 
	\begin{enumerate}
		\item The measures $\mu^n$ converge weakly to the $\gamma$-LQG measure $\mu_{h^{\op{D}}}$ induced by $h^{\op{D}}$. 
		\item The space-filling path on the embedded mated-CRT map which comes from the left-right ordering of the vertices converges uniformly to the space-filling $\op{SLE}_{\kappa}$ curve $\eta^{\op{D}}$. 
		\item The conditional law given $\mcl G^{1/n}$ of the simple random walk on the embedded map started from $\BB x^n$ and stopped upon hitting $\bdy\mcl G^{1/n}$ converges to the law of Brownian motion started from 0 and stopped upon hitting $\bdy\BB D$, modulo time parameterization. 
	\end{enumerate} 
\end{thm} 

Theorem~\ref{thm-tutte-conv} also applies to the so-called \emph{Smith embedding} of the mated-CRT map, where one tiles a square by rectangles corresponding to the edges of the map. See~\cite[Theorem 1.3]{gms-tutte}. 

Recall that the results of~\cite{wedges} show that, in the setting of Theorems~\ref{thm-mating} and~\ref{thm-mating-disk}, the Brownian motion $Z$ determines $(h^{\op{D}},\eta^{\op{D}})$ (modulo rotation and scaling), but not in an explicit way. Theorem~\ref{thm-tutte-conv} makes this determination explicit since the mated-CRT maps $\mcl G^{1/n}$ are explicit functionals of $Z$ and the Tutte embedding is an explicit functional of the mated-CRT map.

It is easy to simulate mated-CRT maps and compute their Tutte embeddings, so Theorem~\ref{thm-tutte-conv} gives an efficient way to simulate SLE/LQG (see~\cite[Footnote 9 and Figure 3]{gms-tutte}). Several such simulations made by J.\ Miller can be found at \url{http://statslab.cam.ac.uk/~jpm205/tutte-embeddings.html}. 

Unlike the results in Section~\ref{sec-strong-coupling}, it is not known how to transfer Theorem~\ref{thm-tutte-conv} from the mated-CRT map to other random planar map models since the polylogarithmic multiplicative errors in Theorem~\ref{thm-strong-coupling} are not suitable for proving scaling limits.

The basic idea of the proof of Theorem~\ref{thm-tutte-conv} is as follows. Recall the a priori SLE/LQG embedding $x\mapsto \eta^{\op{D}}(x)$ of the mated-CRT map. 
It is obvious that under this embedding, the counting measure on vertices converges to $\mu_{h^{\op{D}}}$ (since $\eta^{\op{D}}$ is parametrized by $\mu_{h^{\op{D}}}$-mass) and the space-filling curve on the map converges to $\eta^{\op{D}}$. We want to show that the a priori embedding is close to the Tutte embedding when $n$ is large, so that the same convergence statements also hold for the Tutte embedding. 
Since the Tutte embedding is defined by matching hitting probabilities for random walk on $\mcl G^{1/n}$ to hitting probabilities for Brownian motion on $\BB D$, one only needs to prove the following statement. 

\begin{thm}[\cite{gms-tutte}] \label{thm-quenched-clt0}
	As $n\rta\infty$, the conditional law given $\mcl G^{1/n}$ of simple random walk on the mated-CRT map under the SLE/LQG embedding $x\mapsto \eta^{\op{D}}(x)$ (i.e., the simple random walk on the adjacency graph of cells $\eta^{\op{D}}([x-1/n,x])$) stopped when it hits $\bdy \mcl G^{1/n}$ converges in probability to Brownian motion stopped upon hitting $\bdy\BB D$ modulo time parameterization, uniformly in the choice of starting point. 
\end{thm}

The proof of Theorem~\ref{thm-quenched-clt0} is based on a general scaling limit theorem for random walks in highly inhomogeneous planar random environments which are required to be \emph{translation invariant modulo scaling} instead of exactly stationary with respect to spatial translations~\cite[Theorem 3.10]{gms-random-walk}. 
To prove Theorem~\ref{thm-quenched-clt0}, one uses estimates for SLE and LQG to check that the adjacency graph of cells satisfies the hypotheses of this general random walk in random environment theorem.
The proof also works for the mated-CRT map with the whole-plane or sphere topology. We remark that the theorem of~\cite{gms-random-walk} can be used to show analogs of Theorem~\ref{thm-quenched-clt0} for other discretizations of LQG by ``cells" of approximately equal quantum mass. See, e.g.,~\cite{gms-poisson-voronoi}.   

The paper~\cite{bg-lbm} improves on Theorems~\ref{thm-tutte-conv} and~\ref{thm-quenched-clt0} by showing that the law of the embedded walk converges w.r.t.\ the uniform topology to \emph{Liouville Brownian motion}, the natural ``quantum time" parametrization of Brownian motion on an LQG surface~\cite{berestycki-lbm,grv-lbm} (instead of just converging modulo time parametrization). 

\subsection{Dimension and natural measures for SLE}\label{subsec:app-sle}
As explained in Section~\ref{subsec:sle}, a space-filling SLE$_\kappa$ encodes a number of other processes, e.g.\ CLE$_\kappa$, SLE$_{16/\kappa}$, and SLE$_\kappa$.
As explained in Section~\ref{sec-mating-geometric}, the peanosphere Brownian motion $Z$ encodes these processes as well as certain special subsets associated with them (such as double points, cut points, etc.) in a reasonably simple manner. In this section we will see that mating-of-trees theory can be used to find the dimension of these sets, in addition to defining natural measures supported on these sets.

The KPZ formula is an explicit quadratic formula which relates the Euclidean dimension and LQG dimension for random fractals which are independent from the GFF.
The Euclidean dimension can for example be a Hausdorff dimension or a Minkowski dimension, while the LQG dimension can be defined similarly using the LQG metric (rather than the Euclidean metric) or by using various approximations to this metric. The KPZ formula was first derived heuristically in the physics literature \cite{kpz-scaling}. 
It is a useful tool for predicting exponents associated with random fractals since in many cases the heuristic computation of the ``quantum dimension"
boils down to a counting problem for a process on a random planar map, which is much easier than the corresponding problem on a deterministic lattice.
For example, the KPZ formula was used by Duplantier\footnote{The Brownian intersection exponents were also derived earlier using a different method by Duplantier and Kwon in \cite{duplantier-kwon-brownian}.} in~\cite{duplantier-bm-exponents} to predict the values of the Brownian intersection exponents, which were later obtained rigorously by Lawler, Schramm, and Werner using SLE techniques~\cite{lsw-bm-exponents1,lsw-bm-exponents2,lsw-bm-exponents3}.

Several rigorous versions of the KPZ formula been proved in the math literature, e.g.\ in \cite{aru-kpz,grv-kpz,bjrv-gmt-duality,benjamini-schramm-cascades,wedges,shef-renormalization,shef-kpz,gp-kpz,rhodes-vargas-log-kpz}. The peanosphere allows to prove the following rigorous version of the KPZ formula, where $\dim_{\mcl H}(\cdot)$ denotes Hausdorff dimension.

\begin{thm}[\cite{ghm-kpz}]\label{thm-kpz}
	Consider the setting of Theorem \ref{thm-mating} and assume the field $\bh$ has the circle average embedding. Let $X\subset\BB C$ be a random Borel set which is independent from $\bh$ (e.g., a set which is determined by $\eta$, viewed modulo time parametrization). Then, almost surely for each Borel set $\wh X\subset\BB R$ such that $\eta(\wh X)=X$,  
	\eqb
	\dim_{\mcl H}(X) 
	= \Big(2+\frac{\gamma^2}{2}\Big)\dim_{\mcl H}(\wh X)
	- \frac{\gamma^2}{2}\dim_{\mcl H}(\wh X)^2.
	\label{eq:kpz}
	\eqe 
\end{thm}
The dimension $\dim_{\mcl H}(\wh X)$ has a natural interpretation as the ``$\gamma$-LQG dimension'' of $X$ for the following reason: $\dim_{\mcl H}(\wh X)$ is defined in terms of coverings of $\wh X$ by intervals $I$, each of which corresponds to a covering of $X$ by space-filling SLE segments $\eta(I)$. The $\gamma$-LQG area of each of these space-filling SLE segments is equal to the length of $I$. Therefore $\dim_{\mcl H}(\wh X)$ can be expressed in terms of an optimization problem over covers of $X$, where the function we want to minimize depends on the LQG area of the sets in the cover.

The KPZ formula in Theorem \ref{thm-kpz} differs from other rigorously proven versions of the KPZ formulas since it is directly applicable for dimension computations.
This is because the $\gamma$-LQG dimension $\dim_{\mcl H}(\wh X)$ is often easy to calculate.
Indeed, for many natural sets $X$ defined in terms of $\eta$, the time set $\wh X$ has a nice description in terms of $Z$ (see Section \ref{sec-mating-geometric}), and \eqref{eq:kpz} has been applied in several of these cases to calculate the (Euclidean) Hausdorff dimension $\dim_{\mcl H}(X)$ of $X$. Most of these dimensions had been calculated rigorously previously using SLE theory, but the new proofs are very short and simple, and are closer to the physics derivations of the dimensions, which were often based on the KPZ formula. 

For example, the left frontier $X$ of $\eta$ upon hitting the origin has the law of a whole-plane SLE$_{\ul\kappa}(\ul\kappa/2-2)$ curve for $\ul\kappa = 16/\kappa$ (see Section~\ref{subsec:sle}), and the time set $\wh X$ associated with $X$ is the set of running infima for the Brownian motion $L$ relative to time 0 (see the end of Section \ref{sec-mating-geometric}).
The Hausdorff dimension of this time set $\wh X$ is $1/2$. 
Plugging this into Theorem~\ref{thm-kpz} shows that the Hausdorff dimension of an SLE$_{\ul\kappa}(\ul\kappa/2-2)$ curve is $1+\ul\kappa/8$.
By local absolute continuity, the same is true for an ordinary SLE$_{\ul\kappa}$ curve. This provides an alternative proof of a theorem of Beffara~\cite{beffara-dim}. See~\cite[Section 2]{ghm-kpz} for many additional examples. 

Theorem~\ref{thm-kpz} has also been used to prove new dimension formulas for SLE which were not previously in the literature. For example, in~\cite{ghm-kpz} the theorem is used to compute the Hausdorff dimension of the $k$-tuple points of space-filling SLE$_\kappa$ for any $k\in\BB N$. As another example, in~\cite{ghm-conformal-dim}, Theorem~\ref{thm-kpz} together with the KPZ formula for the LQG boundary length measure from~\cite{rhodes-vargas-log-kpz} is used to prove a formula relating the Hausdorff dimension of a subset of $\BB R$ and the dimension of its image under a conformal map from $\BB H$ to a complementary connected component of an SLE curve.

Besides encoding the dimension of sets associated with SLE, the peanosphere Brownian motion $Z$ can also be used to define \emph{measures} supported on such sets. Recall from Theorem \ref{thm-kpz} that any Borel set $X\subset\BB C$ can be associated with a time set $\wh X\subset\BB R$ for the Brownian motion $Z$. We saw above that the dimension of $\wh X$ defines a notion of ``$\gamma$-LQG dimension'' for $X$. Similarly, if one has a method for measuring the ``size'' of $\wh X$, e.g.\ by considering its Minkowski content or Hausdorff measure, one obtains the ``LQG size'' of $X$. 

For example, in the example just above where $X$ is an SLE$_{\ul\kappa}(\ul\kappa/2-2)$ curve we obtain a measure supported on $\wh X$ by considering the local time of the running infima of $L$, and we obtain a measure $\sigma$ on the SLE$_{\ul\kappa}(\ul\kappa/2-2)$ curve $X$ by pushing forward this local time via $\eta$. It can be proved that the local time of $\wh X$ is equal to its $1/2$-dimensional Minkowski content, which is the same as its Hausdorff measure with the appropriate gauge function.
The measure $\sigma$ can be defined equivalently without using the peanosphere Brownian motion $Z$. Indeed, it can be proved that $\sigma$ is the $\gamma$-LQG measure on $X$ as defined in~\cite{shef-zipper} and discussed in Section~\ref{subsec:zipper}.
Equivalently, $\sigma=e^{d\gamma \bh}d\lambda$, where $d=\dim_{\mcl H}(\wh X)=1/2$ is the $\gamma$-LQG dimension of $X$ as determined by \eqref{eq:kpz} and $\lambda$ is the measure defined by the (Euclidean) Minkowski content of $X$~\cite{lawler-rezai-nat}. See \cite{benoist-lqg-chaos} for a proof. The description of the measure in terms of $Z$ is useful for linking the measure to the appropriate counting measure on random planar maps, while the alternative description provides a link to Euclidean occupation measures. 

Here we consider the example where $X$ is the left frontier of $\eta$ upon hitting 0, but similar relationships hold for many other sets determined by $\eta$. See \cite{hs-cardy-embedding} for precise statements and proofs in the case $\kappa=6$, where (non-space-filling) SLE$_{6}$ and double points (a.k.a.\ pivotal points) for SLE$_6$ are considered.

\subsection{Special symmetries for $\gamma=\sqrt{8/3}$ }
\label{sec-kappa6}

In the physics literature, $\sqrt{8/3}$-LQG is sometimes referred to as ``pure'' gravity, and it is particularly natural since it corresponds to random planar maps which are \emph{uniformly} sampled.
Special properties for $\gamma=\sqrt{8/3}$ (equivalently, $\kappa=6$) have allowed to use mating-of-trees theory as a tool for completing two major programs: Miller and Sheffield \cite{lqg-tbm1,lqg-tbm2,lqg-tbm3} proved the equivalence of the Brownian map and $\sqrt{8/3}$-LQG, and Holden and Sun \cite{hs-cardy-embedding} (based on  \cite{bhs-site-perc,ghss-ldp,hlls-cut-pts,hls-sle6,ghs-metric-peano,aasw-type2}) proved that uniform triangulations converge to $\sqrt{8/3}$-LQG under the so-called Cardy embedding. 
Additionally, the quantum zipper results described in Section~\ref{subsec:zipper} for $\gamma=\sqrt{8/3}$ have been used to show that uniform random planar maps decorated by self-avoiding walks converge to $\sqrt{8/3}$-LQG surfaces decorated by SLE$_{8/3}$ in the metric space sense~\cite{gwynne-miller-saw}.

\subsubsection{The $\sqrt{8/3}$-LQG metric}
\label{sec-metric-pure-lqg}

Le Gall \cite{legall-uniqueness} and Miermont \cite{miermont-brownian-map} proved that uniformly sampled quadrangulations (and more generally $p$-angulations for $p=3$ or $p \geq 4$ even) converge in the scaling limit for the Gromov-Hausdorff topology. The limit is a random metric measure space called the \emph{Brownian map}, which has the topology of the sphere~\cite{legall-paulin-tbm} but has Hausdorff dimension 4~\cite{legall-topological}.
The results of Le Gall and Miermont were later adapted to a number of other random planar map models \cite{abraham-bipartite,ab-simple,bjm-uniform,marzouk-brownian-map}, including planar maps with other topologies \cite{bet-mier-disk,curien-legall-plane,gwynne-miller-simple-quad,gwynne-miller-uihpq,bmr-uihpq}. 

Due to its link with random planar maps, the Brownian map is a canonical model for a continuum random surface equipped with a metric and an area measure. Liouville quantum gravity provides another natural method for constructing a random surface. By construction, an LQG surface has a conformal structure, which gives it a canonical embedding (modulo conformal transformations) into the complex plane. Miller and Sheffield in~\cite{lqg-tbm1,lqg-tbm2,lqg-tbm3} proved that the Brownian map and $\sqrt{8/3}$-LQG are \emph{equivalent} in the sense that they can be coupled together so each surface determines the other surface in a natural way. In particular, the Brownian map has a canonical conformal structure which gives it a canonical embedding into the complex plane, while a $\sqrt{8/3}$-LQG surface has a natural metric. The quantum sphere, equipped with its $\sqrt{8/3}$-LQG metric and area measure, has the same law as the Brownian map. 

To construct a metric on a $\sqrt{8/3}$-LQG surface $(D,h)$, Miller and Sheffield used a growth process they call the \emph{Quantum Loewner evolution} (QLE). QLE started at a point $z\in D$ is defined by growing a whole-plane SLE$_6$ started from $z$, but resampling the tip of the SLE$_6$ every $\delta>0$ units of quantum natural time, and then sending $\delta\rta 0$.
QLE behaves similarly to SLE$_6$ in some ways (e.g., the laws of the quantum surfaces parametrized by the complementary connected components have the same law as for SLE$_6$) but it grows simultaneously in all directions instead of just from the tip. Miller and Sheffield proved that there is a metric on $D$ whose metric ball growth started from each $z\in\BB C$ is given by QLE. 
Using an axiomatic characterization of the Brownian map~\cite{tbm-characterization}, they then proved that in the case of the quantum sphere, the resulting metric space is equal in law to the Brownian map.

\begin{remark}
	The $\sqrt{8/3}$-LQG metric defined using QLE has been proven to be equivalent to the $\sqrt{8/3}$-LQG metric defined via direct regularization of the GFF~\cite{gm-uniqueness,dddf-lfpp}; see \cite[Corollary 1.4]{gm-uniqueness}. The latter metric was not yet constructed at the time of Miller and Sheffield's work.
	Furthermore, it is not clear how to relate the metric of~\cite{gm-uniqueness,dddf-lfpp} to the Brownian map without going through~\cite{lqg-tbm1,lqg-tbm2,lqg-tbm3}.
\end{remark}

The construction of the QLE metric and the proof of its equivalence with the Brownian map rely crucially on the following facts from mating-of-trees theory, which are  established in~\cite{sphere-constructions} and are only true for $\gamma=\sqrt{8/3}$. 
\begin{enumerate}[(i)]
	\item If we run an SLE$_6$ on a quantum sphere up to a fixed quantum natural time, then the quantum surfaces parametrized by the complementary connected components of the SLE$_6$ are independent quantum disks conditioned on their boundary length~\cite[Theorem 1.2]{sphere-constructions} (this is a quantum sphere analog of Theorem~\ref{thm-mating-chordal}). \label{item-sle-sphere-markov}
	\item If we condition on the quantum surface parametrized by the complementary connected component whose boundary contains the tip of the SLE$_6$ curve, then the conditional law of the location of the tip is that of a uniform sample from the $\sqrt{8/3}$-LQG length measure on the boundary of the surface~\cite[Proposition 6.4]{sphere-constructions}.  \label{item-sle-sphere-tip}
\end{enumerate} 
Fact~\eqref{item-sle-sphere-markov} is used to establish a Markov property for QLE processes and connect them to the Brownian map. Fact~\eqref{item-sle-sphere-tip} is what allows one to define QLE, since it implies that re-sampling the location of the tip of the SLE$_6$ curve will not change the law of the unexplored region.

\subsubsection{Percolation and Cardy embedding for uniform triangulations}
Holden and Sun \cite{hs-cardy-embedding} provide a first proof for the convergence of uniform random planar maps to LQG under conformal embedding. They introduce a discrete conformal embedding for a random triangulation with boundary into an equilateral triangle $\BB T$ called the \emph{Cardy embedding}, which is defined in terms of percolation crossing probabilities on the triangulation. They show that the measure and the metric on $\BB T$ induced by the counting measure on the vertices and the graph metric, respectively, on the triangulation converge jointly in the scaling limit to the $\sqrt{8/3}$-LQG area measure and metric associated with a quantum disk.

The peanosphere encoding of LQG and SLE plays an important role in two steps of the proof. First, an important input to the program is that a percolation interface on a random planar map converges to a (non-space-filling) SLE$_6$ on a quantum disk, where we view the surfaces as curve-decorated metric measure spaces~\cite{gwynne-miller-perc}. This result relies on Theorem \ref{thm-mating-chordal} and the stronger result \cite[Theorem 1.4]{gwynne-miller-char} that the joint law of $(L,R)$ and the quantum surfaces parametrized by the complementary connected components of the curve in fact uniquely characterize SLE$_6$ on the quantum disk. In \cite{ghs-metric-peano} the result of \cite{gwynne-miller-perc} is upgraded to show convergence of the whole collection of percolation interfaces on the uniform triangulation toward CLE$_6$.

Second, the peanosphere encoding allows to prove convergence of a number of interesting functionals of a percolated triangulation \cite{bhs-site-perc}, since these are encoded nicely by the peanosphere Brownian motion $Z$. For example, one can show convergence of percolation crossing events and the so-called pivotal measure. A thorough understanding of the latter measure is important for showing in \cite{hs-cardy-embedding} that the convergence to CLE$_6$ in the previous paragraph is \emph{quenched}, i.e., the conditional law of the collection of loops given the planar map converges. This is because the pivotal points govern a natural dynamics on percolated triangulations which is called Liouville dynamical percolation \cite{ghss-ldp}.

\subsubsection{Self-avoiding walk on random planar maps}
Recall that Theorem~\ref{thm:zipper0} tells us that conformally welding two $\gamma$-quantum wedges along their boundaries produces a $\gamma-2\gamma^{-1}$-quantum wedge decorated by an SLE$_{\gamma^2}$ curve.
We now briefly discuss an application of this theorem to random planar maps in the special case when $\gamma=\sqrt{8/3}$. 

The \emph{uniform infinite half-plane quadrangulation (UIHPQ)}, which was introduced in~\cite{curien-miermont-uihpq}, describes the local limit of uniform quadrangulations with boundary at a typical boundary vertex when we send the total number of interior vertices, and then the total number of boundary vertices, to infinity. 
It is shown in~\cite{gwynne-miller-uihpq,bmr-uihpq} that the UIHPQ equipped with its graph distance converges in the Gromov-Hausdorff sense to a metric measure space called the \emph{Brownian half-plane}. Using the results of~\cite{lqg-tbm1,lqg-tbm2,lqg-tbm3}, one can show that the Brownian half-plane is equivalent, as a metric measure space, to a $\sqrt{8/3}$-quantum wedge equipped with its $\sqrt{8/3}$-LQG metric and area measure~\cite[Proposition 1.10]{gwynne-miller-uihpq}.

It is shown in~\cite{caraceni-curien-saw} that gluing together two independent UIHPQ's with simple boundaries along their positive boundary rays (i.e., performing the discrete analog of the gluing operation shown in Figure~\ref{fig-welding}, right) yields a uniform infinite quadrangulation of the half-plane decorated by a self-avoiding walk from the root vertex to $\infty$. 
It is shown in~\cite{gwynne-miller-saw} that this SAW-decorated quadrangulation converges to the curve-decorated metric measure space obtained as the metric quotient of two independent Brownian half-planes identified along their boundaries, with respect to a variant of the Gromov-Hausdorff topology for curve-decorated metric measure spaces. 

Due to Theorem~\ref{thm:zipper0} and the equivalence of the $\sqrt{8/3}$-quantum wedge and the Brownian half-plane, the above limiting object can equivalently be described as a $\sqrt{8/3} - 2\sqrt{3/8}$-quantum wedge decorated by an independent chordal SLE$_{8/3}$ curve (see~\cite{gwynne-miller-gluing} for a proof that the metric gluing and conformal welding operations are compatible). 
This gives the first rigorous link between SLE$_{8/3}$ and self-avoiding walk. The convergence of self-avoiding walk on a deterministic lattice to SLE$_{8/3}$ was conjectured in~\cite{lsw-saw} and is still open.

\subsection{Other applications} 
\label{sec-applications-other}
In this section we will present three further applications of the mating-of-trees theory.

Holden and Sun \cite{hs-euclidean} proved a mating-of-trees theorem for SLE$_\kappa$ ($\kappa>4$) in \emph{Euclidean} geometry. Namely, they proved that for a space-filling SLE$_\kappa$ curve $\eta$ one can define its boundary length process $(L,R)$ as in \eqref{eqn-peanosphere-bm}, but with $1+2/\kappa$-dimensional Minkowski content instead of $4/\sqrt{\kappa}$-LQG length. Furthermore, they prove that $\eta$ is measurable with respect to the $\sigma$-algebra generated by $(L,R)$. Unlike in the LQG case, the law of the boundary length process $(L,R)$ has not been identified explicitly; only continuity and a scaling property are established. Building on \cite{lawler-viklund-lerw-nat}, Holden and Sun also established the Euclidean counterpart of peanosphere convergence for the uniform spanning tree on $\BB Z^2$ to SLE$_8$. Although the statements of the results are strongly inspired by \cite{wedges}, the proofs proceed by completely different methods. We remark, however, that Theorem \ref{thm-mating} is used in one step of the proof, namely to establish a moment estimate for the Euclidean boundary length process by using that the Minkowski content of an SLE$_{16/\kappa}$ curve (modulo a continuous reweighting) is equal to the expected LQG length conditioned on the curve. 

Gwynne and Pfeffer~\cite{gp-sle-bubbles} used mating-of-trees theory (in particular, Theorem~\ref{thm-mating-chordal}) to study the geometry of a (non-space-filling) SLE$_\kappa$ curve $\eta$ for $\kappa\in(4,8)$. They considered the adjacency graph of complementary connected components of $\eta$, where each component corresponds to a vertex and two components are connected by an edge if their boundaries intersect. They showed that for $\kappa\in(4,\kappa_0]$, where $\kappa_0\approx 5.6158$, this adjacency graph is a.s.\ connected, in the sense that any two components can be joined by a finite path in the adjacency graph. By Theorem \ref{thm-mating-chordal}, the adjacency graph has an explicit description in terms of two independent $\kappa/4$-L\'evy processes with only negative jumps. The proof is based on an analysis of these L\'evy processes.

Mating-of-trees theory can also be used to prove exact formulas for quantities related to LQG surfaces.
For example, Ang and Gwynne~\cite[Theorem 1.2]{ag-disk} use the mating-of-trees theorem for the disk (Theorem~\ref{thm-mating-disk}) to express the conditional law of the area of a quantum disk given its boundary length in terms of a certain stopping time for conditioned Brownian motion, and thereby obtain an exact formula (modulo knowledge of the constant $\BB a=\BB a(\gamma)$ in Theorem \ref{thm-mating}) for the density of this conditional law with respect to Lebesgue measure. This provides a completely different approach to proving exact formulas for LQG from the conformal field theory  approach used in~\cite{krv-dozz,remy-fb-formula,rz-gmc-interval}.

Miller, Sheffield, and Werner~\cite{msw-simple-cle-lqg,msw-non-simple-cle} used a variant of the mating-of-trees theorem for chordal SLE$_\kappa$ for $\kappa \in (4,8)$ (Theorem~\ref{thm-mating-chordal}) to study conformal loop ensembles (CLE$_\kappa$) for $\kappa \in (8/3,8)$. We recall that CLE$_\kappa$ is a collection of SLE$_\kappa$-type loops whose law is conformally invariant~\cite{shef-cle} (see also Remark~\ref{rmk:CLE}). In particular, they showed that when one samples an independent CLE$_\kappa$ for $\kappa \in (8/3,4) \cup (4,8)$ on an appropriate type of LQG surface, then the quantum surfaces parametrized by the complementary connected components of the closure of the union of the CLE loops are quantum disks. Their work also has applications to the study of CLE$_\kappa$ (without reference to LQG):
\begin{itemize}
\item It is shown in~\cite{cle-percolations} that for $\kappa\in  (4,6)$, there is a $p = p(\kappa) \in (0,1)$ such that if one colors the loops of a CLE$_\kappa$ independently red with probability $p$ or blue with probability $1-p$, then the collection of interfaces between the red and blue clusters is a CLE$_{16/\kappa}$. 
This relationship is a continuum analog of the Edwards-Sokal coupling. In~\cite{msw-non-simple-cle}, the authors use a description of the law of the LQG boundary length process for the curve which explores the CLE loops (analogous to Theorem~\ref{thm-mating-chordal}) to show that $p = 1/ ( 4\cos^2(\pi\kappa/4))$. 
\item The paper~\cite{msw-non-simple-cle} also uses LQG techniques to derive formulas for ``arm exponents" for the above divide-and-color variant of CLE in terms of $p$ and $\kappa$. This leads to predictions for the analogous arm exponents for divide-and-color models starting from percolation or the Fortuin-Kasteleyn model. 
\end{itemize}

\section{Open problems}
\label{sec-open-problems}

In order to apply the theory developed in this paper to a random planar map model, one first needs an encoding of this model by a mating-of-trees bijection.

\begin{prob}  \label{prob-other-peano}
	Find additional mating-of-trees bijections for random planar maps, besides the ones in~\cite{mullin-maps,bernardi-maps,shef-burger,bernardi-dfs-bijection,bhs-site-perc,gkmw-burger,kmsw-bipolar,lsw-schnyder-wood}. 
\end{prob}

It is particularly interesting to find mating-of-trees bijections for which the encoding walk has i.i.d.\ increments, since currently this is needed to apply Theorem~\ref{thm-strong-coupling}. Alternatively, one could also try to expand the maps to which Theorem~\ref{thm-strong-coupling} applies.

\begin{prob}  \label{prob-kmt}
	Prove analogs of the results for graph distance and random walk on random planar maps described in Section~\ref{sec-strong-coupling} in the setting of random planar maps decorated by the critical FK model or the active spanning tree model, which are encoded by the hamburger-cheeseburger bijection~\cite{shef-burger,gkmw-burger}. 
\end{prob}

Note that the encoding walk in the hamburger-cheeseburger bijection is non-Markovian (except in the degenerate case when it reduces to the Mullin bijection), so one cannot directly apply the strong coupling theorem for random walk and Brownian motion from~\cite{zaitsev-kmt,kmt}. 
Possible approaches to Problem~\ref{prob-kmt} include proving an analog of this strong coupling theorem for the hamburger-cheeseburger walk or finding a more robust way of comparing random planar maps to the mated-CRT map which works with a weaker coupling theorem.  

As we have seen in the bipolar orientation and Schnyder wood cases, mating-of-trees theory provides a useful perspective to some classical combinatorial objects. 
To give an another example, we consider all possible configurations of a simple closed curve  intersecting a straight line, modulo topological equivalence. 
Each such configuration is called a \emph{meander}.  Let $M_n$ be the number of meanders where the number of intersections between the curve and the line is $2n$.
The asymptotic  enumeration of  $M_n$ is a well-known difficult problem in combinatorics. 
Di Francesco, Golinelli, and Guitter \cite{dgg-meander}  put the meander problem into the framework of matter coupled with quantum gravity, i.e.\ decorated random planar maps, and 
argued that the central charge of the matter is $\mathbf c=-4$. 
Let $\gamma=\sqrt{\frac{17}{3}-\frac{\sqrt{145}}{3}}$, which is the unique solution in $(0,2)$ for \eqref{eq:parameter} with $\mathbf c=-4$.  
Based on the LQG connection, it is conjectured in \cite{dgg-meander}  that $M_n\approx n^{-\alpha} R^{2n}$, where $\alpha=1+\frac{4}{\gamma^2}=\frac{29+\sqrt{145}}{12}$ and $R$ is a positive constant.
Restricting the closed curve on each side of the horizontal line, we obtain a pair of rainbow patterns similar to the one on the top right of Figure~\ref{fig-mated-crt-map}.
These rainbow patterns give rise to a two dimensional lattice walk $\cZ$. This gives a mating-of-trees approach to the meander problem.
\begin{prob}  \label{prob-meander}
	Prove that the lattice  walk $\cZ$ in the meander problem converges to a 2D Brownian motion\footnote{More precisely, the scaling limit should be the correlated 2D Brownian excursion appearing in the sphere version of the mating-of-tree theorem~\cite[Theorem 1.1]{sphere-constructions}.} 
	such that the correlation between the two coordinates is  $-\cos (\pi\gamma^2/4)$ with $\gamma$ as above.
\end{prob}
In \cite{gms-burger-local} the asymptotic behavior of the partition function of critical FK-decorated  random planar maps is estimated based on the peanosphere convergence. 
It is possible that  by a similar argument, the solution to Problem~\ref{prob-meander} would give the value of  $\alpha$ in the meander problem.
There are also other meander-type combinatorial models which have a quantum gravity connection, see~\cite{DiFancesco-BAMS,ckst-noodle}.

\emph{Update}: As explained in \cite{bgs-meander}, we expect that, contrary to the above conjecture, the scaling limit of a uniform meander is actually LQG with ${\mathbf c} = -4$, decorated by two independent SLE$_8$ curves. Note that here $\kappa$ is not equal to $16/\gamma^2$, so the results of \cite{wedges} do not apply directly. See \cite{bgs-meander} for more explanation.

Convergence results in the peanosphere sense give some notion of convergence for the \emph{pair} consisting of a random planar map and a statistical physics model on top of it. Does it also provide some notion of convergence for these two objects (namely, the random planar map and the statistical physics model) separately? This question can be formulated more precisely as follows.
\begin{prob} \label{prob-upper-lower}
	Let $M$ be a random planar map reweighted by the partition function of some statistical physics model, and let $P_1$ and $P_2$ be two independent samples of this statistical physics model on $M$. Assume that $(M,P_1)$ (equivalently, $(M,P_2)$) converges in the peanosphere sense to an SLE-decorated LQG surface. The two pairs $(M,P_1)$ and $(M,P_2)$ give us two random walks $\cZ^1$ and $\cZ^2$. Does $(\cZ^1,\cZ^2$) converge jointly in the scaling limit to two Brownian motions $(Z^1,Z^2)$ encoding the \emph{same} LQG surface decorated by two \emph{independent} SLE curves?
\end{prob}
Although a positive answer to this question is very natural, it is far from immediate from Theorem \ref{thm-mating} how one would prove it: The random walks $\cZ^1$ and $\cZ^2$ do not distinguish in simple way between the randomness of the map and the randomness of statistical physics model, and the same holds in the continuum. 

Establishing a positive answer to Problem \ref{prob-upper-lower} for the case of percolated triangulations was a key step in the proof of convergence of uniform triangulations under the Cardy embedding \cite{hs-cardy-embedding}. Variants of the same result for other decorated random planar map models would likely lead to similar convergence results under conformal embedding for those, with the Cardy embedding replaced by an appropriate embedding defined using observables of the statistical physics model. See~\cite[Section 1.3]{ghs-metric-peano} for some further discussion.

Recall (Section~\ref{subsec:scaling}) that for a random planar map $\mcl M$ in the $\gamma$-LQG universality class, it is expected that if we embed $\mcl M$ into $\BB C$ in a reasonable way, then the counting measure on the vertices of $\mcl M$, appropriately rescaled, converges to the $\gamma$-LQG measure. So far, this has been proven only for the Tutte and Smith embeddings of the mated-CRT map and the Cardy embedding of a uniform triangulation (see Sections~\ref{sec-tutte-conv} and~\ref{sec-applications-other}).
In light of the strong coupling results between other random planar maps and the mated-CRT map described in Section~\ref{sec-strong-coupling}, it is natural to try to deduce embedding convergence results for other random planar maps from the embedding convergence result for the mated-CRT map. 

\begin{prob} \label{prob-embedding}
	Can the ideas described in Sections~\ref{sec-strong-coupling} and~\ref{sec-tutte-conv} be extended to show that additional random planar map models (such as spanning-tree weighted random planar maps or bipolar oriented random planar maps) converge to LQG under appropriate embeddings?
\end{prob}

Problem~\ref{prob-embedding} is likely to be difficult since scaling limit results generally require up-to-constants estimates whereas the strong coupling techniques described in Section~\ref{sec-strong-coupling} only yield estimates with polylogarithmic multiplicative errors. 
However, there is an important intermediate step toward the proof of embedding convergence results for random planar maps which does \emph{not} require up-to-constants estimates.
Indeed, one of the major obstacles to proving such results is showing that there are no macroscopic faces, i.e., the maximal size of the faces of the embedded map tends to zero as the number of faces tends to $\infty$. 

\begin{prob} \label{prob-max-face}
	Can the mated-CRT map coupling techniques described in Section~\ref{sec-strong-coupling} be used to show that embeddings of various interesting random planar maps have no macroscopic faces?
\end{prob}

It is shown in~\cite[Corollary 1.6]{gms-harmonic} that the maximal face size for the mated-CRT map under the Tutte embedding decays polynomially in the number of vertices. 
A similar result is proven in~\cite[Theorem 1.6]{gjn-macroscopic-circles} with the circle packing in place of the Tutte embedding. 
It is natural to think that these results together with Theorem~\ref{thm-strong-coupling} should lead to a solution to Problem~\ref{prob-max-face} for the Tutte embeddings and circle packings of random planar maps to which Theorem~\ref{thm-strong-coupling} applies. However, the maximal face size does not seem to be a simple functional of graph distances and/or Dirichlet energies so it is not yet clear how to apply Theorem~\ref{thm-strong-coupling} in this setting.

Our next few questions ask about generalizations of the theory presented in this article. 
Currently, most known results in mating-of-trees theory concern simply connected LQG surfaces. 
It is of interest to extend this theory to the non-simply connected case. 
See Section~\ref{subsec:rv} for some references on known results about LQG on non-simply connected surfaces.

\begin{prob} \label{prob-higher-genus}
	Develop a mating-of-trees type  theory for non-simply connected LQG surfaces.
	Prove scaling limit results for  natural model-decorated random planar maps on non-simply connected surfaces, starting from peanosphere convergence and upgrading to more intrinsic topologies.
	Determine the connection between the LQG surfaces arising in these scaling limits with the ones defined via the path integral approach (see \cite{drv-torus,remy-annulus,grv-higher-genus}).
\end{prob}
The key new difficulty here is that for the non-simply connected case the conformal structure of the LQG surface itself will have to be random, namely, there will be a probability measure on the moduli space.
In the path integral approach, the law of the random modulus is taken to be the one derived from the DDK ansatz (recall Section~\ref{sec-overview} for the DDK ansatz and see \cite{dp-string}  for more details).
However,  from the  planar maps scaling limit perspective, it is challenging to derive  any explicit law for the random modulus for the limiting LQG surfaces.

In all of the results presented in this paper, the SLE and LQG parameters are related by $\kappa  \in\{\gamma^2 ,16/\gamma^2 \}$
Our next problem concerns the case where $\gamma$ and $\kappa$ are ``mismatched".

\begin{prob}  \label{prob-mismatch}
	Is there a variant of the conformal welding and cutting results in \cite{shef-zipper,wedges} for $\kappa \not\in\{\gamma^2 ,16/\gamma^2 \}$? Is there a variant of Theorem \ref{thm-mating} for the case where $\kappa\neq16/\gamma^2$? 
\end{prob}
One key difficulty is that when $\gamma^2\neq\kappa$, the two surfaces one obtains when cutting a $\gamma$-LQG surface by an SLE$_\kappa$ are not independent. Therefore we do not expect the boundary length process $Z$ in Theorem \ref{thm-mating} to be Markovian. Perhaps the most interesting possible result along the lines of Problem~\ref{prob-mismatch} would be to give an explicit description of the law of this process.

All of the results in this article concern LQG with $\gamma\in (0,2)$, which can equivalently be described by the central charge of the corresponding matter field, which is $\mathbf c = 25 - 6(2/\gamma+\gamma/2)^2 \in (-\infty,1)$. 
See~\cite{hp-welding,mmq-welding} for an analog of the conformal welding results of~\cite{shef-zipper,wedges} for $\gamma=2$, equivalently $\mathbf c=  1$. 
There is a work in progress by Aru, Holden, Powell, and Sun which will 
prove a variant of the mating-of-trees theorem of~\cite{wedges} in the case when $\gamma=2$. 
However, the case when $\mathbf c \in (1,25)$, which corresponds to a complex value of $\gamma$, is also of substantial mathematical and physical interest. 
In this case, one has $Q = 2/\gamma + \gamma/2 \in (0,2)$, so the definition of an LQG surfaces still make sense. 
Moreover, it is expected that such surfaces can be endowed with a canonical metric.
However, for $\mathbf c \in (1,25)$ one does not have a construction of a canonical area measure and this phase is much less well understood than the case when $\mathbf c \in (-\infty,1]$. See~\cite{ghpr-central-charge,dg-supercritical-lfpp} for some recent progress and discussion on the issues involved. 

\begin{prob} \label{prob-central-charge}
	Can any of the results discussed in this paper be extended to the case when $\mathbf c \in (1,25)$?
\end{prob}

Theorems~\ref{thm:zipper0} and~\ref{thm-mating} provide, via probabilistic techniques, a method for conformally welding together two quantum wedges or two continuum random trees (CRTs). The proof of these results both rely on the fact that we know what is the object obtained after welding, namely an SLE-decorated LQG surface.

\begin{prob} \label{prob-analytic-welding}
	Can one prove existence and uniqueness for the conformal welding problem for a pair of quantum wedges or a pair of CRTs using analytic techniques?
\end{prob}

To answer this question positively, one would need to identify certain almost sure properties of the quantum wedges or the CRTs which allow to define the conformal welding. 
In the setting of Theorem~\ref{thm:zipper0}, uniqueness of the welding problem is established using the conformal removability of SLE$_\kappa$, but this does \emph{not} provide a solution to Problem~\ref{prob-analytic-welding} since one cannot see the removability of the welding interface just from the two surfaces being welded. 
In~\cite{mmq-welding}, uniqueness of the conformal welding was established under certain regularity conditions on the interface curves which in particular apply to SLE$_4$ (which arises in the conformal welding of $\gamma=2$ quantum wedges \cite{hp-welding}) and SLE$_\kappa$ curves for $\kappa\in (4,8)$ (which arise in mating of trees of quantum disks, see Theorem~\ref{thm-mating-chordal}).
Again, however, this paper does not solve Problem~\ref{prob-analytic-welding} since the uniqueness criterion is in terms of the interface curve. 
Lin and Rohde~\cite{lin-rohde-welding} solve a variant of Problem~\ref{prob-analytic-welding} for the conformal welding of a CRT to a Euclidean disk. 
We also note that the conformal welding problem for a pair of trees has been solved for several deterministic Julia sets. See~\cite{milnor-mating} for an introduction to this theory as well as~\cite[Section 2.3]{wedges} and the references therein. 

The results surveyed in this article focus on conformal weldings of two LQG surfaces with the same value of $\gamma$. 
The paper \cite{astala-welding} studies the conformal welding of an LQG surface and a Euclidean disk.

\begin{prob} \label{prob-mismatched-welding}
	What can be said about the conformal welding of a $\gamma_1$-LQG surface and a $\gamma_2$-LQG surface for distinct $\gamma_1,\gamma_2\in (0,2)$?
\end{prob}

A solution to Problem~\ref{prob-analytic-welding} in the case of quantum wedges might also yield the existence and uniqueness of the welding in the setting of Problem~\ref{prob-mismatched-welding} if the welding criterion is a.s.\ true for pairs LQG surfaces with different $\gamma$-values.

\bibliography{cibiblong,cibib,extrabib}
\bibliographystyle{hmralphaabbrv}

\end{document}